\documentclass{amsart}

\usepackage{amssymb}
\usepackage{url}
\usepackage{amscd}
\usepackage{texdraw}

\theoremstyle{plain}

\newtheorem{thm}{Theorem}[section]

\newtheorem*{kl1}{Key Lemma 1}
\newtheorem*{kl2}{Key Lemma 2}
\newtheorem*{kl3}{Key Lemma 3}
\newtheorem*{kl4}{Key Lemma 4}

\newtheorem*{Red-m}{Reduced Main Theorem}

\newtheorem*{main}{Main Theorem}

\newtheorem{lem}{Lemma}[section]
\newtheorem{rmk}{Remark}[section]

\newtheorem{defi}{Definition}[section]
\newtheorem{pro}{Proposition}[section]

\input txdtools

\begin{document}

\date{}

\title[Polynomial Siegel disks  are typically Jordan domains]{Polynomial  Siegel disks  are typically Jordan domains}
\author{Gaofei Zhang}
\address{Department of  Mathematics, Nanjing University, Nanjing,   210093, P. R.
China} \email{zhanggf@hotmail.com}
\thanks{}
\subjclass[2000]{58F23, 37F10, 37F45, 32H50, 30D05}
\begin{abstract}
We prove that for typical rotation numbers polynomial Siegel disks are Jordan domains with boundaries containing at least one of the critical  points, and moreover, when the rotation number is fixed, the boundaries of the Siegel disks depend continuously on the polynomial maps.
\end{abstract}

\maketitle


\section{Introduction}

 It  was conjectured by Douady and Sullivan in early 1980's that the boundaries of  Siegel disks of rational maps are always Jordan curves \cite{DS}. The conjecture remains open, however, there have been many   results  relative to  the conjecture. It holds  when the rotation numbers of the Siegel disks are of  bounded type \cite{Do}\cite{Sh}\cite{Za1}\cite{Zh1}. It also holds for  quadratic Siegel disks  with rotation numbers being of sufficiently high type \cite{IS}.  For certain  rotation numbers,  the boundaries of quadratic Siegel disks  can even be smooth Jordan curves \cite{ABC}\cite{BC}\cite{Per}.
  Besides, it is worth to mention that for holomorphic germs,  there exist relative compact Siegel disks with non-locally connected boundaries \cite{Ch1}.

   The  rotation numbers of the Siegel disks in all  the above results belong to a set of zero Lebesgue measure. In 2002 Petersen and Zakeri proved that for
typical rotation numbers, a quadratic Siegel disk is a Jordan domain \cite{PZ}. To state this theorem more precisely,
 let us introduce a class of irrational numbers.
    Let $C > 0$  and $\Theta_C$ denote the set of all irrational numbers  $0< \theta < 1$  such that
\begin{equation}\label{ci}
\log{a_{n}} \le C \sqrt n, \:\:\forall n\ge 1,\end{equation}
where $a_{1}, a_{2}, \cdots$ are all the coefficients of  the continued fraction of
$\theta$.
Let $$\mathcal{E} = \bigcup_{C >0}\Theta_C.$$
 It is known that  $\mathcal{E}$  is  a full measure subset of $[0,1]$ \cite{Kh}.
In \cite{PZ} Petersen and Zakeri   proved that
for any $\theta \in \mathcal{E}$, the
Siegel disk of $P_\theta(z) = e^{2 \pi i \theta}z + z^2$ is a Jordan
domain whose boundary contains the unique finite critical point of
$P_\theta$. The main purpose of this paper is to generalize this result to
polynomial maps of all degrees.
\begin{main}
All polynomial Siegel disks with rotation numbers belonging to $\mathcal{E}$  are Jordan domains
with boundaries containing at least one of the critical points. Moreover,
 when the rotation number belongs to $\mathcal{E}$ and is fixed,
 the boundaries  of the Siegel disks  depend continuously on the  polynomial maps.
\end{main}

 One of the fundamental tools in our proof is trans-qc surgery. This surgery technique  was pioneered by Haissinsky \cite{Ha}, who used it to transform an attracting basin into a parabolic basin, and
then introduced to the study of Siegel disks by Petersen and Zakeri in \cite{PZ}.  Compared with qc surgery, the main difficulty in performing trans-qc surgery is to verify the integrability of certain degenerate Beltrami differentials.  This often requires some delicate area estimates. In \cite{PZ}  the authors there used Petersen puzzles   to obtain the desired  estimate for the Douady-Ghys premodel.  For a  general premodel, however,  it is not known if the invariant Beltrami differential is integrable or not.
  This is the essential  challenge
in generalizing Petersen-Zakeri's theorem to  polynomial maps of all degrees.

The following is the very general idea of our proof. A detailed outline of the proof will be given in $\S2$.    Suppose $C > 0$ is a fixed constant  and $D$ is a Siegel disk of an arbitrary  polynomial map with rotation number $\theta \in \Theta_C$.  By perturbing $\theta$ we get a sequence of bounded type Siegel disks $D_N$ with rotation numbers $\theta_N \in \Theta_C$ such that $\theta_N \to \theta$. By Shishikura's theorem,  each $\partial D_N$ is a quasi-circle passing through at least one of the critical points.  We shall see if $\theta$ is not of bounded type,  the qc constants of $\partial D_N, N \ge 1$, are not bounded.  This means that the oscillations of these quasi-circles can not be uniformly controlled with respect to the qc constants.  Thus nothing could be obtained if we  let $N$ go to $\infty$ at this point. The key of our proof is to  find an appropriate way to measure the oscillations of  these quasi-circles so that in this way, the oscillations
  can be uniformly controlled.
To do this, we  will introduce   a family of oscillation functions. We   prove that
these oscillation functions are uniformly controlled  for bounded type Siegel disks of a class of special   polynomial maps with the rotation numbers belonging to $\Theta_C$.
  We then show that for a bounded type Siegel disk of an arbitrary polynomial map with the rotation number belonging to $\Theta_C$, the oscillation functions can be controlled, in certain sense,  by those for the special ones.  From this we derive  that the oscillations of the sequence of quasi-circles are uniformly controlled. By passing to a subsequence if necessary, it follows that the sequence of quasi-circles  converge to some Jordan curve which passes through at least one of the critical points. This Jordan curve must be the boundary of $D$. The argument  also implies  that  for a fixed $\theta \in \Theta_C$,  the boundary of the Siegel disk  depend  continuously on the polynomial maps.   This proves the Main Theorem.

 The  following is the organization of the paper.

In $\S2$ we present a detailed outline of the proof. We first
formulate a reduced version of the Main Theorem  by introducing
oscillation functions.  We then state  four key
lemmas.  The proofs of these four  lemmas  form the core
part of the paper. Finally we prove the Reduced Main Theorem by
assuming  these four lemmas.

In $\S3$ we prove that the Reduced Main Theorem implies the Main Theorem.

In $\S4$ we prove   Key Lemma 1.  This lemma asserts that  the oscillation of the boundaries of  bounded type
Siegel disks for a class of special polynomial maps, with rotation numbers belonging to $\Theta_C$,   can be uniformly controlled.   This is the place where we use trans-qc
surgery.  The tool  developed in \cite{Zh2}  will play a  role here. It allows us to make   a uniform area estimate  for the Beltrami differentials of a special class of premodels. Key Lemma 1 then follows from Tukia's theorem on the  compactness property of David homeomorphisms.

In $\S5$ we prove Key Lemmas 3 and 4.   These two lemmas
are used to construct a chain of slices  in the
parameter space.  Each of these slices  is an algebraic Riemann surface determined by a finite system of polynomial equations.
The chain of slices  is a bridge connecting  an arbitrary Siegel polynomial map to those special ones. The oscillation functions are holomorphic in each of these slices.  By maximal and minimal principles of holomorphic functions, the control of the oscillation functions will be passed on along the chain of slices. In this way the oscillation of the boundary of the Siegel disk of an arbitrary polynomial map is controlled by those of the special ones.

In $\S6$ we  establish   a topological  characterization  of a class of polynomial maps
with bounded type Siegel disks.
This class of Siegel polynomial maps play a crucial role in this work.    The proof of this result contains most of  the ingredients needed in the proof of   Key Lemma 2.  After that,   Key Lemma 2 follows by a little more effort.  We use  Key Lemma 2   to perturb certain Siegel polynomial map so that the resulted one can be embedded into an appropriate slice in the parameter space.

In $\S7$, the Appendix of the paper, we present a list of basic properties about bounded type Siegel disks of polynomial maps.
 One of them is Shishikura's theorem which asserts that all bounded type Siegel disks of polynomial maps are quasi-disks with qc constants depending only on the degree and the rotation number.  From Shishikura's theorem it follows that for a fixed bounded type rotation number,  the boundary of  the  Siegel disks moves continuously. This property will be essentially used in our proof.

$\bold{Acknowledgement.}$  Many thanks are due to Prof. Carsten Lunde Petersen  who spent a lot of time discussing  with me on an early version of  the manuscript  during his visit of Nanjing in March, 2012.

\section{Outline of the proof} Throughout the paper we use $\widehat{\Bbb C}$, $\Bbb C$,  ${\Bbb C}^*$, $\Delta$ and $\Bbb T$ denote the Riemann sphere, the complex plane, the punctured complex plane with a puncture at the origin, the unit disk and the unit circle respectively.

Fix an integer $d \ge 2$ and a  $\theta \in \mathcal{E}$ throughout the paper. We may assume that $\theta$ is not of bounded type.  Let
$[a_1, \cdots, a_n, \cdots]$ be the continued fraction of $\theta$.
By definition we have $C > 0$ such that
$$
\log a_n \le C \sqrt n
$$
for all $n \ge 1$. Let  $$P(z) = e^{2 \pi i \theta} z + \alpha_2 z^2
+ \cdots + \alpha_d z^d$$ with $\alpha_d \ne 0$. We want to show
that the Siegel disk of $P$ centered at the origin is a Jordan
domain with at least one critical point on its boundary.

For the above $C > 0$,  let $$\Theta_C^b \subset \Theta_C$$  denote  the subset consisting of all the bounded type irrational numbers  in $\Theta_C$.
Let $\theta_N \in \Theta_C^b$, $N \ge 1$, be a sequence  such that $\theta_N \to \theta$ as $N \to
\infty$.  Such sequence can be constructed in many ways.   To fix the idea let us take
$$
\theta_N =  [a_1, \cdots, a_N, 1, 1, 1, \cdots] .
$$

For each $N \ge 1$, let $$P_{N}(z) = e^{2 \pi i \theta_N} z +
\alpha_2 z^2 + \cdots + \alpha_d z^d.$$  Then $P_N$ converges to $P$
uniformly in any compact set of the complex plane.  It follows that
the critical sets of all $P_N$ are contained in a neighborhood of
that of $P$, and therefore contained in a compact set of the plane.
Let $D_N$ denote the Siegel disk of $P_N$ centered at the origin.
Since $\theta_N$ is of bounded type, by Shishikura's theorem
(\cite{Sh}, see also \cite{Zh1}), there is a critical point $c_N$ of
$P_N$,  and a $K_N
> 1$ depending only on $$\sup_{1 \le k \le N}\{a_k\},$$  such that
$\partial D_N$ is a $K_N$-quasi-circle and passes through $c_N$. By
taking a subsequence, we may assume that $c_N$ converges to some
critical point $c$ of $P$.  Note that $K_N \to \infty$ if
$\sup_{k\ge 1}\{a_k\} = \infty$. Because otherwise, by taking a
subsequence, $\partial D_N$ would converge to a quasi-circle passing
through $c$. This quasi-circle must be the boundary of the Siegel
disk of $P$ centered at the origin. But by a result of Petersen
\cite{P3}, the rotation number of such Siegel disk must be of
bounded type. This is a contradiction.

Let
\begin{equation}\label{furc}
 Q = c^{-1} P(cz) \hbox{ and } Q_N(z)= c_N^{-1}P(c_N z).
 \end{equation}
   Then the point $1$ is a critical point of both $Q$ and $Q_N$. Let us
still use $D_N$ to denote the Siegel disk of $Q_N$ centered at the
origin.  It follows that  $1 \in \partial D_N$ for all $N \ge 1$.

The main task of our proof is to show that the sequence of curves
$\partial D_N$ converge to a Jordan curve passing through the
critical point $1$.  Since the quasiconformal constant $K_N$ is unbounded,
we need to find an appropriate way to measure the oscillation of these curves so that
the oscillation of the curves can be uniformly controlled.
Before we proceed further let us introduce some notations first.

 Let  $0< \alpha < 1$ be a bounded type irrational number. We use
 $\bold{\mathcal{P}_\alpha^d}$ to denote the class of all the polynomial maps $f$ such that
$f(z) = e^{2 \pi i \alpha} z + a_2 z^2 + \cdots a_d z^d$ with $a_d \ne 0$ and $f'(1) = 0$.
Let $D$ denote the Siegel disk of $f$ centered at the origin.
Let $\bold{\mathcal{Q}_\alpha^d} \subset \mathcal{P}_\alpha^d$ be the subclass which contains all those $f$ such that $1 \in \partial D$.
 Let $\bold{\Sigma_\alpha^d} \subset \mathcal{Q}_\alpha^d$ be the subclass which contains all the $f$ such that each critical point of $f$ either belongs to the basin of some attracting periodic cycle of $f$ or belongs to $\partial D$.
 Let $\bold{\Pi_\alpha^d} \subset \Sigma_\alpha^d$ be the subclass which contains all the $f$ such that all the finite critical points of $f$ belongs to $\partial D$. By definition we have
 $$
 \Pi_\alpha^d \subset \Sigma_\alpha^d \subset \mathcal{Q}_\alpha^d \subset \mathcal{P}_\alpha^d.
 $$
For $f \in \mathcal{Q}_\alpha^d$, $\partial D$ is a quasi-circle and contains the critical point $1$.  We refer to
$$
\sigma_{k,m}(f) = f^k(1) - f^m(1), \:\:k > m \ge 0
$$
as the family of oscillation functions for $\partial D$.

\begin{Red-m}\label{reduction} Let $d \ge 2$ be an integer and $C > 0$.
Then there exist a pair of positive functions $\eta, \lambda: (0, 2] \to
\Bbb R^+$ satisfying  $$\lim_{\delta \to 0_+} \lambda (\delta) =
\lim_{\delta \to 0_+} \eta (\delta)  =  0$$ such that for any $$f \in \bigcup_{\alpha \in \Theta_C^b}\mathcal{Q}_{\alpha}^d,$$  any  pair of integers $k > m \ge 0$ and any pair of  positive numbers
$0< \delta' \le \delta$ satisfying $\delta' \le
 |e^{2 \pi ik\alpha}- e^{2 \pi im\alpha}| \le \delta$,  the
 inequality
 $$
 \eta(\delta') \le  |\sigma_{k,m}(f)| \le \lambda(\delta)
 $$
holds.

\end{Red-m}

Applying  the Reduced Main Theorem to $Q_N$ we get a uniform control of the oscillation of $\partial D_N$. This implies that the sequence of curves  $\partial D_N$ converge to a Jordan curve
and    the  Main Theorem  follows. The detailed argument will be presented in $\S3$.
The main part  of the paper is to prove the Reduced Main Theorem. The
proof is based on four key lemmas.

   For each $f \in \Pi_\alpha^d$, let $D_f$ be the Siegel disk of $f$ and $H: D_f \to \Delta$ be a conformal isomorphism such that $H(0) = 0$ and $f|D_f = H^{-1} \circ R_\alpha \circ H$ where $R_\alpha: z \mapsto e^{2 \pi i \alpha}z$ is the rigid rotation given by $\alpha$. Since $\partial D_f$ is a quasi-circle by Shishikura's theorem, $H$ can be homeomorphically extended to $\partial D_f$.  It follows that  all the finite critical points of $f$  are mapped by $H$ to points in $\Bbb T$.  In this sense we can  speak of the angle between any two critical points on $\partial D_f$. We will see $f$ is uniquely  determined  by the $d-2$ angles
between the critical point $1$ and all the other $d-2$  critical points (A topological characterization of  the maps in $\Sigma_\alpha^d$ will be given in $\S 6$).

\begin{kl1}\label{Key-1} Let $d \ge 2$ be an integer and $C > 0$.
Then
there exist a pair of positive functions $\eta_1, \lambda_1: (0, 2] \to
\Bbb R^+$ satisfying $$\lim_{\delta \to 0_+} \lambda_1 (\delta) =
\lim_{\delta \to 0_+} \eta_1 (\delta)  =  0,$$ such that
 for any $$f \in \bigcup_{\alpha \in \Theta_C^b}\Pi_\alpha^d,$$
 any pair of integers $k > m \ge 0$ and any pair of  positive numbers
$0< \delta' \le \delta$  satisfying $\delta' \le
 |e^{2 \pi ik\alpha}- e^{2 \pi im\alpha}| \le \delta$,  the
 inequality
 $$
 \eta_1(\delta') \le  |\sigma_{k,m}(f)| \le \lambda_1(\delta)
 $$
holds.
\end{kl1}

 The proof of  Key Lemma 1  is based on Lemmas~\ref{uniform-p}-
\ref{normal}.  The following is the outline of the proof.
 For any $$f \in \bigcup_{\alpha \in \Theta_C^b}
\Pi_\alpha^d,$$  there is a Blaschke
product $B_f$ of degree $(2d-1)$ which models $f$ (cf. $\S6$). In particular,
all the critical points of $B_f$, except $0$ and $\infty$, are
contained in $\Bbb T$. Let $R_\alpha: z \mapsto e^{2 \pi i \alpha} z$
be the rigid rotation given by $\alpha$ and $h_f: \Bbb T \to \Bbb T$ be the circle
homeomorphism such that
$$B_f|\Bbb T  = h_f^{-1} \circ R_{\alpha}\circ h_f.$$
Since  $\alpha$ is of  bounded type, by Herman's theorem, $h_f: \Bbb T \to \Bbb T$
is a quasisymmetric circle homeomorphism. Because the qc constants can not be uniformly controlled,  instead of  making a
usual quasiconformal extension of $h_f$  and then performing a qc surgery,
we will  construct a David extension $H_f: \Delta
\to \Delta$ of $h_f$  by adapting the idea in \cite{PZ} and then perform a trans-qc surgery.    Since $\alpha$ is of bounded type,  the map $H_f$
obtained in this way is necessarily quasiconformal. The key point
here is that we regard $H_f$ as a David homeomorphism when we
measure its distortion.
\begin{lem}\label{uniform-p}
There exist $M, \beta > 0$ and $0< \epsilon_0 < 1$ depending only
on $C$ and $d$ such that for  any $f \in \bigcup_{\alpha \in
\Theta_C^b} \Pi_\alpha^d$,  the conjugation map $h_f:
\Bbb T \to \Bbb T$ has a David extension $H_f: \Delta \to \Delta$
which fixes the origin and satisfies  the following. For any $0< \epsilon < \epsilon_0$, we
have
$$
area\{z\in \Delta\:|\: |\mu_{H_f}(z)| > 1 - \epsilon\} < M
e^{-\frac{\beta}{\epsilon}}
$$ where $\mu_{H_f}$ denotes the Beltrami coefficient of $H_f$ and $area(\cdot)$ denotes the area with respect to the Euclidean metric.
\end{lem}

The essential idea behind the proof of  Lemma~\ref{uniform-p}  is certain
 uniform saddle-node geometry satisfied by the family of  circle mappings $B_f|\Bbb T, f \in
\bigcup_{\alpha \in \Theta_C^b} \Pi_\alpha^d$ (cf.  Lemma~\ref{U-S-G}),
which is  a consequence of  Herman's
uniform estimate on the distortion of cross-ratios for compact family of holomorphic circle mappings (cf.  Lemma ~\ref{HC1}).

Now let $H_f: \Delta \to \Delta$ be the David homeomorphism in
Lemma~\ref{uniform-p}. Define
\begin{equation}\label{modified-B}
\widehat{B}_f(z) =
\begin{cases}

 B_f(z) & \text{for $z \in \widehat{\Bbb C} \setminus \Delta$}, \\

H_f^{-1} \circ R_{\alpha} \circ H_f (z) & \text{for $z \in \Delta$}.
\end{cases}
\end{equation}
Let $\mu_f$ denote the Beltrami differential on the whole plane which is obtained by
pulling back $\mu_{H_f}$ through the iteration of $\widehat{B}_f$.

\begin{lem}\label{uniform-w}
There exist $\tilde{M}, \tilde{\beta} > 0$ and $0<
\tilde{\epsilon}_0 < 1$ depending only on $C$ and $d$  such that for
any  $f \in \bigcup_{\alpha \in \Theta_C^b} \Pi_\alpha^d$,   we have
\begin{equation}\label{end-k1}
area\{z\in \Bbb C\:|\: |\mu_{f}(z)| > 1 - \epsilon\} < \tilde{M}
e^{-\frac{\tilde{\beta}}{\epsilon}}
\end{equation}for any $0< \epsilon < \tilde{\epsilon}_0$.
\end{lem}

Lemma~\ref{uniform-w} asserts the uniform integrability of the invariant Beltrami differentials for all Blaschke premodels $\widehat{B}_f$, $f \in \bigcup_{\alpha \in \Theta_C^b} \Pi_\alpha^d$.    As we mentioned before the main difficulty in performing a trans-qc surgery is to verify the integrability of
certain degenerate Beltrami differential. The key idea used in the proof of Lemma~\ref{uniform-w} is a method
developed  in \cite{Zh2} which allows us to obtain a uniform area estimate.

\begin{lem}[Tukia, \cite{Tu}]\label{normal}
Let $\mathcal{F}$ denote the class of all $(\tilde{M},
\tilde{\beta}, \tilde{\epsilon}_0)$-David homeomorphisms of the
plane to itself which fix $0$ and $1$.  Then there exist three
constants $\hat{M}, \hat{\beta} > 0$ and $0< \hat{\epsilon}_0 < 1$  depending only on $\tilde{M},
\tilde{\beta}$ and  $\tilde{\epsilon}_0$
such that any sequence in $\mathcal{F}$ has a subsequence which
converges uniformly  to a $(\hat{M}, \hat{\alpha}, \hat{\epsilon}_0)$-David
homeomorphism of the plane which fixes $0$ and $1$.
\end{lem}

Key Lemma 1 is a direct consequence of Lemmas~\ref{uniform-p}
-~\ref{normal} (cf.  $\S4$).  We would like to point out that in the case of cubic polynomial
maps,  the Reduced Main Theorem follows from  Key Lemma 1.  The following is the detailed argument.

Let $\alpha \in \Theta_C^b$.   For $f \in \mathcal{P}_\alpha^3$ let $c$ denote the other finite critical point of $f$. The space $\mathcal{P}_\alpha^3$ is parameterized by $c$.  With this parametrization, $\mathcal{P}_\alpha^3$ is homeomorphic to the punctured plane $\Bbb C^*$.  For each $c \in \Bbb C^*$, let $f_c$ denote the  corresponding cubic polynomial and $D_c$ denote the Siegel disk of $f_c$ centered at the origin.   By a simple calculation we have $$f_c(z) = \frac{e^{2 \pi i \alpha}}{3c} z^3 - \frac{e^{2 \pi i \alpha}}{2}(1 + \frac{1}{c})z^2 + e^{2 \pi i \alpha} z.$$  Since $f_c$ depends holomorphically on $c$ when $c$ varies in $\Bbb C^*$, $\sigma_{k,m}(f_c)$ is a holomorphic function in $\Bbb C^*$. By a result of Zakeri (cf. $\S 14$ of \cite{Za1}), there is a Jordan curve $\Gamma \subset \Bbb C^*$ such that
\begin{itemize}
\item[1.] $\Gamma$ separates $0$ and $\infty$, passes through $1$ and is invariant under $c \to 1/c$,
\item[2.] for all $c$ belonging to the interior of $\Gamma$ and not equal to $0$, $\partial D_c$  passes through the critical point $c$ only; for all $c$ belonging to the exterior of $\Gamma$ and not equal to $\infty$, $\partial D_c$  passes through the critical point $1$ only; for all $c$ belonging to $\Gamma$,  $\partial D_c$  passes through both $1$ and $c$.
\end{itemize}
 For the $\alpha$ given and $d = 3$, let  $\lambda_1$ and $\eta_1$ be the two positive functions guaranteed by Key Lemma 1. Note that $\Gamma$ corresponds to the class $\Pi_\alpha^3$ by the second assertion above. So for any pair of integers $k > m \ge 0$ and any pair of  positive numbers
$0< \delta' \le \delta$  satisfying $\delta' \le
 |e^{2 \pi ik\alpha}- e^{2 \pi im\alpha}| \le \delta$,   we have
 \begin{equation}\label{cubic-in}
 \eta_1(\delta') \le  |\sigma_{k,m}(f_c)| \le \lambda_1(\delta) \:\:\:\hbox{   for  all   }\: c \in \Gamma.
\end{equation}
Since $1$ belongs to $\partial D_c$ for $c$ belonging to the exterior of $\Gamma$,  $\sigma_{k,m}(f_c)$ does not vanish in the exterior of $\Gamma$. Noting that as $c \to \infty$, $f_c$ converges uniformly to a quadratic polynomial, it follows that $\sigma_{k,m}(f_c)$ has a removable singularity at infinity. This implies that $\sigma_{k,m}(f_c)$ is a holomorphic function in the exterior of $\Gamma$ and does not vanish. So both the maximal and minimal principles apply.  It follows that (\ref{cubic-in}) holds  for all $c$ belonging to the exterior of $\Gamma$.  The Reduced Main Theorem for cubic polynomials  follows by taking $\lambda = \lambda_1$ and $\eta = \eta_1$.

The argument above, however, does not work for polynomial maps of degree $d \ge 4$.
The following is a very rough explanation.  For each $\alpha \in \Theta_C^b$, a Siegel
polynomial map in $\Pi_\alpha^d$ is uniquely determined
by the $d-2$ angles between $1$ and all the other $d-2$ finite
critical points.  So the real dimension of the set of parameters  corresponding to the maps in the class $\Pi_\alpha^d$ is equal to the dimension of the set
 $$\underbrace{S^1 \times \cdots
\times S^1}_{(d-2) \:\:\rm{copies}}/G_{d-2}
$$ where  $G_{d-2}$ is the permutation group of order $d-2$,    which is equal to
 $d-2$.  But the whole  parameter space
$$\underbrace{{\Bbb C}^* \times \cdots \times \Bbb C^*}_{(d-2)\:\:{\rm copies}}\:$$  has complex dimension $d-2$ and real dimension $2d-4$. To bound a domain in the whole parameter space, the set must have real dimension at least  $2d-5$. But for $d \ge 4$, $d -2 < 2d-5$.  Hence for $d \ge 4$,  the  parameters  corresponding to the maps in the class $\Pi_\alpha^d$ can not bound any domain in the whole parameter space, and  the maximal and minimal principles can not be used directly.
  To solve this problem,  we will  introduce certain slices in the parameter space.
  Each slice is an algebraic  Riemann surface determined by a system of polynomial equations.
We will apply the maximal and minimal principles successively on a
chain of such slices,  and finally   prove that the oscillation of
the boundary of  a Siegel disk for an arbitrary polynomial map, in
certain sense, can be controlled by the oscillation of the boundary of the Siegel disk  for
some   Siegel polynomial map in the class $\Pi_\alpha^d$.  Since we have proved that the
later can be uniformly controlled, the Reduced Main Theorem follows.
The construction of these  Riemann
surfaces relies on the other three key lemmas.

Key Lemma 2 is based on a topological characterization of the maps in $\Sigma_\alpha^d$ which will be established in $\S6$.  It  is an extension of Thurston's characterization  for post-critically finite rational maps. Before we state the theorem, let us introduce some terminologies first.  We call an orientation preserved and  finitely branched   covering map $f: \widehat{\Bbb C} \to \widehat{\Bbb C}$  a $\emph{topological polynomial}$     if  $f^{-1}(\infty) = \{\infty\}$. Let $f$ be a topological polynomial.
Let  $\mathcal{O} = \{x_1, \cdots, x_p\}$   be a periodic cycle of $f$ with period $p$. We say $\mathcal{O}$ is a holomorphic
attracting cycle  if (1) $f$ is holomorphic in an open neighborhood $U$ of
$\mathcal{O}$, and (2) $|Df^p(x_1)| < 1$, and (3) $\mathcal{O}$ attracts at least one infinite critical orbit of $f$.

\begin{defi}\label{T-d}{\rm Let $0< \alpha < 1$ be a bounded type irrational number. Let
$\mathcal{T}_\alpha^d$ denote the class of all topological
polynomials of degree $d$ such that
\begin{itemize}
\item[1.] the point $1$ is a critical point of $f$,
\item[2.] $f|\Delta$ is the rigid rotation  given by $z \to e^{2 \pi i \alpha} z$,
\item[3.] any critical point of $f$   either is    attracted to
some holomorphic attracting cycle of $f$, or eventually lands on a periodic cycle containing some critical point, or belongs to $\Bbb T$.
\end{itemize}}\end{defi}
Let $P_f$ denote the closure of the union of all critical orbits of $f$.
\begin{defi}{\rm
We say a map $f \in \mathcal{T}_\alpha^d$ is CLH-equivalent to a map $g \in \Sigma_\alpha^d$ if there exist two homeomorphisms $\phi, \psi: \widehat{\Bbb C} \to \widehat{\Bbb C}$ such that
\begin{itemize} \item[1.] $\phi|\Delta = \psi|\Delta$ are holomorphic,
\item[2.] for each holomorphic attracting cycle $\mathcal{O}$ of $f$ if there is any, there is an open neighborhood $U$ of $\mathcal{O}$ such that $\phi|U = \psi |U$ are holomorphic,
\item[3.] $\phi$ is isotopic to $\psi$ rel $P_f \cup \cup_i \overline{D_i}$ where $D_i$ are  open neighborhoods of all holomorphic attracting cycles,
    \item[4.] $\phi \circ f = g \circ  \psi$.
\end{itemize}
}
\end{defi}

\begin{thm}\label{Thurston-Siegel-Ch}
A map $f \in \mathcal{T}_\alpha^d$ is CLH-equivalent to a map $g \in \Sigma_\alpha^d$ if and only if $f$ has no Thurston obstructions in the exterior of $\Delta$.  Such $g$ if exists, must be unique up to a linear conjugation.
\end{thm}
For the definition of $\emph{Thurston obstructions}$, cf. $\S6.1$.

Key Lemma 2 is closely related to Theorem~\ref{Thurston-Siegel-Ch}.
Suppose $g\in \Sigma_\alpha^d$  such that the boundary of the
Siegel disk contains more than one critical point.  Key Lemma 2
asserts that one  can always perturb  $g$  in
$\Sigma_\alpha^d$ so that after the perturbation  all
critical points in the boundary of the Siegel disk satisfy an orbit
relation.  For any polynomial $f$, let $\|f\|$ denote the maximal absolute value of all the coefficients of $f$. For any two polynomials $f$ and $g$, define  ${\rm dist}(f, g) = \|f-g\|$.

\begin{kl2}\label{Thurston-ri}
Let $g \in \Sigma_\alpha^d$. Suppose $g$ has two or
more distinct critical points on the boundary of the Siegel disk, say $c_1,
 \cdots, c_m$, where $m \ge 2$.  Suppose $c_1 = 1$.  Then for any $\epsilon > 0$, there is
a $\tilde{g} \in \Sigma_\alpha^d$ such that
\begin{itemize}
\item[1.] ${\rm dist}(g, \tilde{g})  < \epsilon$,
\item[2.]  $\tilde{g}$ has exactly $m$ distinct critical points on the boundary of the Siegel disk, say $\tilde{c}_i$, $1 \le i \le m$, such that $\tilde{c}_1 = 1$ and $|c_i - \tilde{c}_i| < \epsilon$ for all $1 \le i \le m$,
\item[3.]  there are positive integers $k_i$, $2 \le i \le m$, such that
$\tilde{g}^{k_i}(1) = \tilde{c}_i$ for $2 \le i  \le m$.
\end{itemize}
\end{kl2}
The proofs of  Theorem~\ref{Thurston-Siegel-Ch}   and   Key Lemma 2 will be given in $\S6$.

Let $f \in \mathcal{P}_\alpha^d$. Then  $f$  has $d-1$ critical
points (counting by multiplicities) and at least one of them is
contained in the boundary of the Siegel disk centered at the origin.
So $f$ has  at most $d-2$ attracting periodic cycles. Note that
$f$ is uniquely determined by the set of its critical points. More
precisely,  for each $(d-1)$-tuple $$X = (c_1, \cdots, c_{d-1}),
\:\:c_i \in \Bbb C^* \hbox{  for  }1 \le i \le d-2,  \hbox{ and
}c_{d-1} = 1,$$ there is a unique $f \in \mathcal{P}_{\alpha}^d$
such that $X$ is the critical set of $f$. By a simple calculation,
we have
\begin{equation}\label{form-1}
f(z)  =  \sum_{i=1}^{d} a_i z^i
\end{equation} with
\begin{equation}\label{form-2}a_i = e^{2 \pi i \alpha}
\cdot \bigg{(}\frac{(-1)^{i-1}}{i} \bigg{)} \cdot \frac{Q_{d-i}(c_1,
\cdots, c_{d-1})}{c_1 \cdots c_{d-1}}
\end{equation}
where $Q_{d-i}$ is the
degree-$(d-i)$ elementary polynomials of $c_1, \cdots, c_{d-1}$. Let
us denote such $f$ by $$f_{c_1, \cdots, c_{d-2}, 1} \hbox{   or  } f_X.$$

\begin{kl3}\label{First-implicit-lemma}
Let  $f \in \mathcal{P}_{\alpha}^d$  and $1 \le l \le d-3$. Suppose $f$ has $l$ attracting cycles with non-zero multipliers $t_1, \cdots, t_l$.   Then
 there exist a compact Riemann surface $S$   and meromorphic functions  $c_1, \cdots, c_{l+1}$  in $S$,  such that   $f$  can be embedded in the holomorphic family of polynomials maps $$h_t = f_{c_1(t), \cdots, c_{l+1}(t), c_{l+1}^0, \cdots, c_{d-2}^0, 1}, \:\: t \in S \setminus (Z \cup P),$$ where $Z$ and $P$ are respectively the set of the zeros and poles of $c_i$, $i = 1, \cdots, l+1$, and moreover,  each $h_t$, $t \in S\setminus (Z \cup P)$,  has $l$ attracting cycles which depend
holomorphically on $t$  and have constant multiplies $t_1, \cdots, t_l$.
\end{kl3}

\begin{kl4}\label{Second-implicit-lemma}
Let $f \in \Sigma_\alpha^d$ and  $0 \le l \le d-3$. Suppose $f$ has $l+1$ attracting cycles with
multiplies $t_1, \cdots,  t_{l+1}$,
each of which attracts exactly one of the critical points, and moreover, there are $d-l-3$ integers $k_i \ge 0$, such that $f^{k_i}(1) = a_{i}$ for $1
\le i \le d-l-3$ where  $a_1, \cdots, a_{d-l-3}, a_{d-l-2} =1$ are the critical points, counting by multiplicities, on the boundary of the Siegel disk.
 Then
there exist a  compact Riemann surface $S$
and meromorphic functions $c_1, \cdots, c_{d-2}$  in $S$,  such that $f$  can be embedded in the holomorphic family of polynomials maps $$h_t = f_{c_1(t), \cdots, c_{d-2}(t), 1}, \:\: t \in S \setminus (Z \cup P),$$ where $Z$ and $P$ are respectively the set of the zeros and poles of $c_i$, $i = 1, \cdots, d-2$, and moreover,  each $h_t$, $t \in S\setminus (Z \cup P)$,  has $l$ attracting cycles which depend
holomorphically on $t$  and have constant multiplies $t_1, \cdots, t_l$, and the boundary of the Siegel disk of $h_t$ centered at the origin contains $1$, and $h_t^{k_i}(1) = c_i(t)$ for $1 \le i \le d-l-3$  (In the case that $l = d-3$, there is no such orbit relations).
\end{kl4}

The proofs of  Key Lemmas 3 and 4 will be given in $\S5$.

Now let us prove the Reduced Main Theorem by assuming  Key Lemmas~1-4.

For $d = 2$, the point $1$ is the only finite critical point of $f$ and is contained in the boundary of the Siegel disk. Thus  $f \in \Pi_\alpha^2$ and  the Reduced Main Theorem follows from  Key Lemma 1 in this case.

Suppose $d \ge 3$  and assume that  the Reduced Main Theorem  holds  for polynomial maps of degrees less than $d$ : that is,  there exist a pair of positive functions $\eta_0, \lambda_0: (0, 2] \to
\Bbb R^+$ satisfying $$\lim_{\delta \to 0_+} \lambda_0 (\delta) =
\lim_{\delta \to 0_+} \eta_0 (\delta)  =  0,$$ such that for any $\alpha \in \Theta_C^b$, if $f \in \mathcal{Q}_\alpha^j$ with $2 \le j < d$, then
 for  any pair of integers $k > m \ge 0$ and any pair of  positive numbers
$0< \delta' \le \delta$  satisfying $\delta' \le
 |e^{2 \pi ik\alpha}- e^{2 \pi im\alpha}| \le \delta$,   the
 inequality
 $$
 \eta_0(\delta') \le  |\sigma_{k,m}(f)| \le \lambda_0(\delta)
 $$
holds.

  Now let $\alpha \in \Theta_C^b$ and  $f \in \mathcal{Q}_\alpha^d$.
 The  proof is divided into two steps.

In the first step, by the Key-Lemma 3 we will construct a finite chain of slices in the parameter space to connect $f$ to some   $g \in \Sigma_\alpha^d$
such that   \begin{itemize}\item[1.] the boundary of the Siegel disk of $g$ centered at the origin contains only the critical point $1$, \item[2.]  $g$ has
 $d-2$ periodic attracting cycles each of which attracts exactly one of the other finite critical points of $g$, \item[3.] the oscillation of the boundary of the Siegel disk of $f$ is controlled either by the oscillation of the boundary of the Siegel disk of $g$ or by the oscillation of the boundary  of the Siegel disk  of some polynomial map in $\mathcal{Q}_\alpha^j$ with $2 \le j < d$. \end{itemize}

In the second step,   by Key Lemmas 2 and 4 we will   construct a finite chain of slices in the parameter space to connect  $g$ to  some  $h \in \Pi_\alpha^d$.  In each of these slices we apply maximal and minimal principles to the oscillation functions. In this way we derive that  the oscillation of the boundary of the Siegel disk of  $g$  is  controlled either by the oscillation of the boundary of the Siegel disk  of some polynomial map in $\mathcal{Q}_\alpha^l$ with $2 \le l < d$,   or
 by the oscillation of the boundary of the Siegel disk of $h$.
 The Reduced Main Theorem  then follows  by induction and  Key Lemma 1.

Now let $k > m \ge 0$ be any two integers. Suppose $0< \delta' < \delta \le 2$ such that $\delta ' < |e^{2 k\pi i \alpha} - e^{2 m\pi i \alpha}| < \delta$. Let $\epsilon > 0$ be an arbitrary positive number.

Step I.    Assume that the number of the periodic attracting cycles of $f$ is less than $d-2$. Otherwise,
we  go to Step II directly.  Let us label the critical points
of $f$ as
$$c_1^0, \cdots, c_{d-2}^0, c_{d-1}^0 = 1.$$
We will repeat the following process at most $d-2$ times.
Each time we will get some polynomial map which has at least one more periodic attracting cycle.

Assume that $f$ has $l$ attracting periodic cycles with $0 \le l \le d-3$.
In the case $l = 0$, that is, $f$ has no periodic attracting cycles, we just embed $f$ into the one-parameter holomorphic family
$$
f_{c_1, c_2^0, \cdots, c_{d-2}^0, 1}, \:\:\:c_1 \in \Bbb C^{*}.$$ For the expression of $f_{c_1, c_2^0, \cdots, c_{d-2}^0, 1}$, see (\ref{form-1}-\ref{form-2}).
Otherwise, we have $1 \le l \le d-3$.  If $f$ has super-attracting cycles, by doing quasi-conformal deformation in the immediate basins of the super-attracting cycles,  we can get
$\tilde{f} \in \mathcal{Q}_\alpha^d$ which can be arbitrarily close to $f$ such that all the $l$ attracting cycles of $\tilde{f}$ have non-zero multipliers, and moreover,
 \begin{equation}\label{perturbation-k3}|\sigma_{k,m}(f)| - \epsilon < |\sigma_{k,m}(\tilde{f})| < |\sigma_{k,m}(f)| + \epsilon.\end{equation} To simplify the notation, let us still use $f$ to denote $\tilde{f}$.  By  Key Lemma 3,   we can
embed $f$ into a holomorphic family of polynomial maps
$$
h_t = f_{c_1(t), \cdots, c_{l+1}(t), c_{l+2}^0, \cdots, c_{d-2}^0, 1},  \:\:\: t \in S \setminus (Z \cup P)
$$ where $Z$ and $P$ denote respectively  the set of zeros and poles of the meromorphic functions $c_i$, $1 \le i \le l+1$. Suppose
$h_{t_0} = f_{c_1^0, \cdots, c_{l+1}^0, c_{l+2}^0, \cdots, c_{d-2}^0, 1}$ for some $t_0 \in  S \setminus (Z \cup P)$.

Note that  $\sigma_{k,m}(h_{t})$ is  holomorphic  in $S \setminus (Z \cup P)$.  By Lemma~\ref{compact-class} points in $P$ are removable singularities of $\sigma_{k,m}(h_t)$.

In the first case, $h_{t_0}$ is $J$-stable at $t_0$.
Let $U\subset S\setminus (Z\cup P)$  be the component containing $t_0$  in which $J_{h_{t}}$ moves
holomorphically.  This implies that the critical point $1$ always stays on the boundary of the Siegel disk  for all $t \in U$ and thus  $\sigma_{k,m}(h_{t})$ does not vanish in $U$.  Take a sequence  $t_n$ in $U$ such that
$|\sigma_{k,m}(h_{t_n})|$ converges to  $\sup_{t\in U} |\sigma_{k,m}(h_{t})|$.  By taking a subsequence we may
assume that the sequence $t_n \to t^* \in \partial U$ (the same argument works for the
infimum for  $\sigma_{k,m}(h_{t})$ does not vanish in $U$).  Since
 for all the parameters in $U$, the boundary of the Siegel disk centered at the origin passes through the critical point $1$, by the second assertion of Lemma~\ref{plmd}, the critical points of $h_{t}$ are uniformly bounded away from the origin for all the parameters in $U$. Thus $Z \cap \partial U = \emptyset$.   There are two  subcases.

Subcase I. $t^* \in P$.   By Lemma~\ref{compact-class}
there is some  $g \in \mathcal{Q}_\alpha^j$ with $2 \le j < d$ such that
 \begin{equation}\label{es1}
  |\sigma_{k,m}(f)|  =  |\sigma_{k,m}(h_{t_0})|  \le \lim_{n\to \infty} |\sigma_{k,m}(h_{t_n})|  =   |\sigma_{k,m}(g)|  \le \lambda_0(\delta).
 \end{equation}
the last inequality comes from our induction assumption.

Subcase II. $t^* \in S \setminus (Z \cup P)$.   Then
  $h_{t}$ is not $J$-stable at $t^*$. Thus
 by   Theorem 4.2 of \cite{McM1},
one can take a $\hat{t} \in S \setminus (Z \cup P)$, which can be arbitrarily close to
$t^*$, such that $h_{\hat{t}}$ has at least one
more periodic attracting cycle  than $h_{t^*}$.  We can choose such $\hat{t}$ so that the new periodic cycles have non-zero multipliers.    Let $D_{t^*}$ and $D_{\hat{t}}$
denote respectively the Siegel disks of $h_{t^*}$ and $h_{\hat{t}}$ centered at the origin. Note that it is possible that
$1 \notin \partial D_{\hat{t}}$.  In this case, $\partial D_{\hat{t}}$  contains some other critical point $c$.

 Since $1 \in \partial D_{t^*}$ we have ${\rm diam}(D_{t^*}) \ge 1$. Since the boundary of the Siegel disk moves continuously by Lemma~\ref{plmd},
by taking $\hat{t}$ close enough to $t^*$ we can make sure that ${\rm diam}(D_{\hat{t}}) > 1$. Now from the second assertion of Lemma~\ref{plmd} we have some $L > 1$ depending only on $d$ such that
\begin{equation}\label{im-e}
L \ge  {\rm diam}(D_{\hat{t}}) \ge |c| \ge  {\rm diam}(D_{\hat{t}})/L > 1/L.
\end{equation}

By taking $\hat{t}$ close enough to $t^{*}$,
  we can make the  $c$
 arbitrarily  close to $\partial D_{t^*}$, and thus by taking an appropriate integer $p \ge 0$,  we can make   $h_{\hat{t}}^{p}(c)$ arbitrarily close to $1$.   So for the given $\epsilon > 0$, by taking $\hat{t}$ close enough to $t^{*}$,
 and letting $k' = k + p$ and $m' = m + p$, we have
$$
|\sigma_{k,m}(h_{{t^*}})| = |h_{t^*}^{k}(1)-h_{t^*}^{m}(1)| < |h_{\hat{t}}^{k'}(c) - h_{\hat{t}}^{m'}(c)| + \epsilon.
$$
Let
$$
g(z) = c^{-1} h_{\hat{t}}(cz).
$$
Then $g \in \mathcal{Q}_\alpha^d$.  Since $|c| \le L$ by (\ref{im-e}), we have
$$
|h_{\hat{t}}^{k'}(c) - h_{\hat{t}}^{m'}(c)|  \le L \cdot |\sigma_{k',m'}(g)|.
$$ This means that  $$|\sigma_{k,m}(h_{{t^*}})| <  L \cdot |\sigma_{k',m'}(g)| + \epsilon.$$
We finally get
\begin{equation}\label{es2}
|\sigma_{k,m}(f)| \le \lim_{n\to \infty} |\sigma_{k,m}(h_{t_n})| =  |\sigma_{k,m}(h_{{t^*}})| < L \cdot |\sigma_{k',m'}(g)| + \epsilon.
\end{equation}

\begin{rmk}\label{min-pri}{\rm  To get the lower bound of $|\sigma_{k,m}(f)|$, besides applying the minimal principle instead of the maximal principle to $\sigma_{k,m}(h_t)$ in $U$,
we need the   inequality $|c| > 1/L$ implied by (\ref{im-e}). In this case, (\ref{es2}) becomes $|\sigma_{k,m}(f)| > |\sigma_{k',m'}(g)|/L - \epsilon$. }
\end{rmk}

Note that $k' - m' = k - m$ and thus $\delta' < |e^{2 k'\pi i \alpha} - e^{2 m'\pi i \alpha}| = |e^{2 k\pi i \alpha} - e^{2 m\pi i \alpha}| < \delta$.
We can thus  repeat the above process on $g$. Since each time the number of the attracting periodic cycles is increased by at least one and the number of periodic attracting cycles is not more than $d-2$ (one of the critical point is contained in the boundary of the Siegel disk),  we have the following   two possibilities.  The first possibility is that we  get an inequality like (\ref{es1})  after not more than $d-2$ steps,   and therefore,
 \begin{equation}\label{choice-1}
 |\sigma_{k,m}(f)| \le L\cdot (L\cdot(\cdots(L \cdot  \lambda_0 (\delta) +  \epsilon) +\cdots)+ \epsilon )+ \epsilon,
 \end{equation} where the number of the recursive  steps is not greater than $d-2$. The second possibility is that  we finally get a map in $\Sigma_\alpha^d$ which has $d-2$ attracting periodic cycles in $\Bbb C$, each of which
attracts a finite critical point, and  a pair of integers $k' > m' \ge 0$ with $k' - m' = k- m$,  such that
\begin{equation}\label{choice-2}
|\sigma_{k,m}(f)| \le L\cdot (L\cdot(\cdots(L \cdot  |\sigma_{k',m'}(g)| +  \epsilon) +\cdots)+ \epsilon )+ \epsilon,
\end{equation} where the number of the recursive  steps is not greater than $d-2$.

If the first possibility occurs, that is,  we get (\ref{choice-1}), then we skip  Step II and draw the conclusion. If the second possibility occurs,  we go to Step II.

$\bold{Step\:\:II}$.  Recall that $g$ has $d-2$ attracting cycles in $\Bbb C$, and exactly one critical point, $1$, on the boundary of the Siegel disk centered at the origin, and satisfies (\ref{choice-2}). We will repeat the following process by induction.

Suppose for some integer $0 \le l \le d-3$,
$g$ has $l+1$ attracting cycles and $d - l -2$ critical points on the boundary of the Siegel disk which satisfy orbit relations as given in Key Lemma 4.  By  Key Lemma 4, there is a compact
Riemann surface $S$ and two finite subsets $Z$ and $P$ of $S$, such that  $g$ can be embedded into the holomorphic family
$$
h_{t} = f_{c_1, \cdots, c_{d-3}, c_{d-2}, 1}, \:\:t \in S \setminus (Z \cup P),
$$ where  all $c_i$  are meromorphic functions in $S$
with  $Z$ and $P$ being  respectively the sets of zeros and the poles of
$c_i$, $1 \le i \le d-2$, and moreover,
\begin{itemize}
\item[1.] each $h_{t}$ has $l$ attracting
cycles with constant multipliers $t_1, \cdots, t_{l}$,
\item[2.] all the orbit relations among the critical points on the boundary of the Siegel disk,  are preserved, that is,
 $h_{t}^{k_i}(1) = c_i$, $1 \le i \le d-l-3$. \end{itemize}    Note that we start from
  $l= d-3$ and the point $1$ is the only critical point  on the boundary of the Siegel disk. So in the beginning  there are no orbit relations among critical points on the boundary of the Siegel disk.

As before,  $\sigma_{k',m'}(h_{t})$ is a holomorphic function defined in $S \setminus (Z \cup P)$.  Let  $\Sigma \subset S \setminus (Z \cup P)$ be the subset which contains all those points for which the boundary of the Siegel disk of $h_{t}$ contains  exactly those $d-l-2$  critical points, $c_1, \cdots, c_{d-l-3}$, $1$,  which satisfy the orbit relations.  We claim that
  $\Sigma$ is an open set. Let us prove the claim. Suppose $t_0 \in \Sigma$. Let $D_{t}$ denote the Siegel disk of $h_{t}$ centered at the origin.
   Then all the critical points of $h_{t_0}$, which do not belong to $\partial D_{t_0}$,  have a positive distance from $\partial D_{t_0}$.
   By Lemma~\ref{continuous-moving},
   in a small open neighborhood of $t_0$, all these critical points are still bounded away from $\partial D_{t_0}$.  It suffices to show that, in a small open neighborhood of $t_0$,  the critical points, which satisfy the orbit relations, still belong to $\partial D_{t}$.    Let us prove this by contradiction.
  Note that except $1$,  there are exactly  $d- l -3$ critical points, $c_1, \cdots, c_{d-3}$ on the boundary of the Siegel disk, and $d-l-3$ integers, $1\le k_1 < \cdots <  k_{d-l-3}$ such that $h_t^{k_i}(1) = c_i$, and moreover, there are $l$ attracting cycles with multipliers $t_1, \cdots, t_{l}$.  If the claim were not true, then there would be a sequence $t^n \to t_0$ such that  for $h_{t^n}$, at least one of the $d-l-2$ critical points, $1$, $c_1^n, \cdots, c_{d-l-3}^n$, does not belong to $\partial D_{t^n}$. For the convenience, let us denote $c_0^n =1$.
   By taking a subsequence, we may assume that there is an
   $0 \le i \le d-l-3$ such that $c_i^n \notin \partial D_{t^n}$,  and  $c_j^n \in \partial D_{t^n}$, for all $ i+1 \le j \le d-l-3$.  Consider the map
   $$
   F_{t} = h_{t}^{k_{i+1} - k_i}.
   $$
The map $F_{t_0}$ maps $c_i$ to $c_{i+1}$ and the local degree of $F_{t_0}$ at $c_{i}$ is equal to that of $h_{t_0}$ at $c_{i}$, which is two.
This means there is a pair of Jordan domains $A$ and $B$ containing $c_i$ and $c_{i+1}$ respectively such that
$F_{t_0}: A \to B$ is a branched covering map of degree two.  Then for all $t^n$ close to $t_0$ enough, there is a pair of Jordan domains $A_n$ and $B_n$  with $A_n \to A$ and $B_n \to B$  such that
$F_{t^n}: A_n \to B_n$  is a branched covering map of degree two.   But on the other hand, when $t^n$ is close to $t_0$,
$c_{i+1}^n$ is close to $c_{i+1}$ and thus belongs to $B_n$ and $c_{i}^n$ is close to $c_i$ and thus belongs to $A_n$. Let $z_n \in \partial D_{t^n}$ be the point  such that $F_{t^n}(z_n) = c_{i+1}^n$. As $t^n \to t_0$,  by Lemma~\ref{continuous-moving} we  have $\partial D_{t^n} \to \partial D_{t_0}$ and thus  $z_n \to c_i$. Thus as $n$ is large enough, $z_n$ belongs to $A_n$ also. This implies that the preimage of $c_{i+1}^n$ in $A_n$   under the map $F_{t^n}$, counting by multiplicities, is at least three. This is a contradiction. Thus $\Sigma$ is an open set and the claim has been proved.

Now let $t_0 \in \Sigma$ such that $g = h_{t_0}$. Let $U$ denote the component of $\Sigma$ which contains the point $t_0$.  Since for all $t \in U$,
the point $1$ belongs to the boundary of the Siegel disk on which $h_{t}$ is qc-conjugate to the irrational rotation $R_{\alpha}$, $\sigma_{k',m'}(h_{t})$ does not vanish in $U$. So both the maximal and minimal principles  apply to $\sigma_{k',m'}(h_{t})$ in $U$.  Take a sequence $t^n$ in $U$  such that
$\sigma_{k',m'}(h_{t^n})$ converges to its supremum (The same argument works for the infimum).  By taking a subsequence, we may assume that $t^n$ converges to a point $t^* \in \partial U$.

Since for all $t\in U$,  the boundary of the Siegel disk of $h_{t}$ centered at the origin passes through the critical point $1$,  by the second assertion of  Lemma~\ref{plmd}, all critical points of $h_{t}$ are uniformly bounded away from the origin. This implies that $Z \cap \partial U = \emptyset$. We thus have the following two subcases.

Subcase I. $t^* \in P$.
  By Lemma~\ref{compact-class}
there is some  $h \in \mathcal{Q}_\alpha^j$ with $2 \le j < d$  such that
 \begin{equation}\label{qch1}
  |\sigma_{k',m'}({g})| \le \lim_{n\to \infty} |\sigma_{k',m'}(h_{t^n})|  =  | \sigma_{k',m'}(h)|  \le \lambda_0(\delta).
 \end{equation}

Subcase II.  $t^* \in S \setminus (Z \cup P)$.  We thus have
\begin{equation}\label{nbry}
|\sigma_{k',m'}({g})| \le |\sigma_{k',m'}(h_{t^{*}})|.
\end{equation}
We claim that $h_{t^*}$ has exactly one more critical point on the boundary of the Siegel disk, that is, there are $d - l -1$ critical points on $\partial D_{t^*}$.    To see this, first note that  all the $d-l-2$
critical points  $c_{1}^n, \cdots, c_{d-l-3}^n$, $1$ belong to $\partial D_{t^n}$ for all $n$. Since $t^n \to t^*$ and the boundary of the Siegel disk moves continuously by Lemma~\ref{continuous-moving}, it follows that
all the $d-l-2$ critical points $c_{1}^*, \cdots, c_{d-l-3}^*$, $1$ must belong to $\partial D_{t^*}$  also. If there is no more critical point on $\partial D_{t^*}$, then $t^* \in \Sigma$ by the claim we proved above. Because $\Sigma$ is open, there exists an open disk neighborhood of $t^*$, say $V \subset \Sigma$.   This contradicts the fact that $t^* \in \partial U$  and that $U$ is a connected component of $\Sigma$.  So there must be at least one more critical point in $\partial D_{t^*}$. Since $h_{t^*}$ has $l$ attracting cycles which attract at least $l$ critical points, it follows that $\partial D_{t^*}$ contains exactly one more critical point. The claim has been proved.

Now let us summarize the above two subcases in Step II.   If the Subcase I happens,  the step II will be ended.   If the Subcase II happens, we will have a polynomial map, say $\hat{g} \in \Sigma_\alpha^d$, having one more critical point on the boundary of the Siegel disk and $l$ attracting cycles each of which attracts exactly one of the other critical points, and moreover, $\hat{g}$ satisfies (\ref{nbry}), that is,
\begin{equation}\label{ess-in}
|\sigma_{k',m'}({g})|  \le  |\sigma_{k',m'}(\hat{g})|.
\end{equation}
Now applying   Key Lemma 2 to  $\hat{g}$, we  get   $\tilde{g}\in \Sigma_{\alpha}^d$, which can be  arbitrarily close to
 $\hat{g}$  such that all the critical points of $\tilde{g}$, which belong  to the boundary of the Siegel disk, satisfy  orbit relations,  that is, $\tilde{g}^{k_i}(1) = \tilde{c}_i$ for $1 \le i \le d-l-2$; and moreover, $\tilde{g}$ has $l$ attracting cycles with the same multipliers as those of $\hat{g}$.
Because $\tilde{g}$ can be arbitrarily close to $\hat{g}$, we may assume that $|\sigma_{k',m'}(\hat{g})| < |\sigma_{k',m'}(\tilde{g})| + \epsilon$, and thus by (\ref{ess-in}) we have
\begin{equation}\label{k2}
|\sigma_{k',m'}(g)| < |\sigma_{k',m'}(\tilde{g})| + \epsilon.
\end{equation}

Note that $\tilde{g}$ has $l$ attracting cycles and $d - l-1$ critical points, including the critical point $1$, on the boundary of the Siegel disk.
Now we  repeat the above process for the polynomial map $\tilde{g}$ from the beginning of the Step II. Since each time  the number of the critical points on the boundary of the Siegel disk is increased by one,
after at most $d-2$ steps, we can either have the Subcase I and get an inequality like (\ref{qch1}),  and  therefore by (\ref{k2})  and the induction hypothesis  we have
  \begin{equation}\label{choice-f1}
 |\sigma_{k',m'}(g)| <  |\lambda_0 (\delta)| + (d-2) \epsilon,
 \end{equation}
or after $(d-2)$ steps we finally get a
 polynomial map $h \in \Pi_\alpha^d$ such that
 \begin{equation}\label{choice-f2}
|\sigma_{k',m'}(g)| <   |\sigma_{k',m'}(h)| +  (d-2) \epsilon \le   \lambda_1(\delta)  +  (d-2) \epsilon.
\end{equation}
Note that we use the fact that  $k' - m' = k-m$  and therfore  $$|e^{2 \pi ik'\alpha} - e^{2 \pi im'\alpha}| = |e^{2 \pi ik \alpha} - e^{2 \pi im\alpha}| < \delta$$ and thus
$|\sigma_{k',m'}(h)| \le \lambda_1(\delta)$ by   Key Lemma 1.

From (\ref{perturbation-k3}), (\ref{choice-1}), (\ref{choice-2}), (\ref{choice-f1})  and (\ref{choice-f2}) we have in all the cases
\begin{equation}\label{choice-im}
|\sigma_{k,m}(f)| \le
 L\cdot (L\cdot(\cdots(L \cdot  \max\{\lambda_0(\delta),  \lambda_1(\delta)\} +  \epsilon) +\cdots)+ \epsilon )+\epsilon,
\end{equation} where the number of the recursive steps is $2d$ times. Since $\epsilon > 0$ is arbitrary, by letting $\epsilon \to 0$, we have
$$
|\sigma_{k,m}(f)| \le L^{2d} \cdot \max\{\lambda_0(\delta), \lambda_1(\delta)\}.
$$

Using the same argument and replacing the maximal principle by the minimal principle (cf.  Remark~\ref{min-pri}),  we get
$$
|\sigma_{k,m}(f)| \ge L^{-2d} \cdot \min\{\eta_0(\delta), \eta_1(\delta)\}.
$$

Now define $\lambda, \eta:(0, 2] \to (0, +\infty)$ by setting
$$
\lambda(x) = L^{2d} \cdot \max\{\lambda_0(x), \lambda_1(x)\}  \:\:\hbox{  and  }\:\: \eta(x) =  L^{-2d} \cdot \min\{\eta_0(x), \eta_1(x)\}.
$$
The  Reduced Main Theorem follows.

\section{Proof of the Main Theorem}

Let $F: \Bbb C \to \Bbb C$ be a continuous map.  For any pair of
integers $k
> m \ge 0$, let
\begin{equation}\label{oscillation}
\sigma_{k,m}(F) = F^k(1) - F^m(1).
\end{equation}
Suppose $\lambda, \eta: (0, 2] \to \Bbb R^+$ are  a pair of positive
functions  such that
\begin{equation}\label{oscillation-control}
\lim_{\delta \to 0_+} \lambda (\delta) = \lim_{\delta \to 0_+} \eta
(\delta)  =  0.
\end{equation}

\begin{lem}\label{fundamental principle} Let $0< \theta < 1$ be an irrational number.
Suppose for any pair of integers $k > m \ge 0$ and any pair of
positive numbers $0< \delta' < \delta$  satisfying
$$
\delta' <  |e^{2k\pi i \theta} - e^{ 2m\pi i\theta}| <  \delta,
$$ we have
$$
\eta(\delta') \le |\sigma_{k,m}(F)| \le \lambda(\delta).
$$
Then
$$
\Gamma = \overline{\{F^k(1)\}_{k=0}^\infty}
$$  is a Jordan curve. Moreover, $F: \Gamma \to \Gamma$ is a
topological circle homeomorphism with rotation number $\theta$.
\end{lem}
\begin{proof}
 Let $\Bbb T$ denote the unit circle. Then $X = \{e^{2k \pi i \theta}\}_{k\ge 0}$ is a dense subset of $\Bbb T$.
 Define a map $\phi: X \to \Bbb C$ by setting $\phi(e^{2 k \pi i\theta}) = F^k(1)$ for every $k \ge 0$. Then
 $\phi: X \to \Bbb C$ is uniformly continuous by assumption. Thus
 $\phi$ can be uniquely extended to a continuous map from $\Bbb T$
 to $\Bbb C$. Let us still denote the map by $\phi$. We claim that
 $\phi$ is injective.

Let us prove it by contradiction. Assume that $\phi(x) = \phi(y)$
for some $x \ne y \in \Bbb T$. Let $\delta' = \frac{1}{2}|x - y|$.
Since $X$ is dense in $\Bbb T$ and $\phi: \Bbb T \to \Bbb C$ is
continuous,  we have two integers $k
> m \ge 0$ such that \begin{itemize}\item[(1)] $|e^{2k \pi i\theta} - x| < |x -
y|/4$, \item[(2)] $|e^{2m \pi i\theta} - y| < |x - y|/4$, \item[(3)]
$|\phi(e^{2 k \pi i \theta}) - \phi(x)| < \eta(\delta')/2$,
\item[(4)] $|\phi(e^{2 m \pi i \theta}) - \phi(y)| < \eta(\delta')/2$.
\end{itemize}  From (1) and (2) we have $|e^{2k\pi i\theta} -
e^{2m \pi i\theta}| > |x - y|/2 = \delta'$.  From the assumption, we
have $|\phi(e^{2k \pi i \theta}) - \phi(e^{2m \pi i \theta})| >
\eta(\delta')$. But from (3), (4) and $\phi(x) = \phi(y)$,  we have
$|\phi(e^{2k \pi i \theta}) - \phi(e^{2m \pi i \theta})| <
\eta(\delta')$. This is a contradiction. This implies that $\phi:
\Bbb T \to \Bbb C$ is injective.  Thus
$$\overline{\{F^k(1)\}_{k=0}^\infty} = \overline{\phi(X)}= \phi(\Bbb
T)$$ is a Jordan curve. This proves the first assertion.

To prove the second assertion, note that
$$
F \circ \phi  = \phi \circ R_{\theta}
$$
holds on $X$. Since $X$ is dense on $\Bbb T$   and  $F \circ \phi,\: \phi \circ R_\theta: \Bbb T \to \Bbb C$ are both continuous, the above equation holds on $\Bbb T$.  Since $\phi$ is injective on $\Bbb T$,  it
follows that
$$
\phi^{-1} \circ F \circ \phi   = R_{\theta}
$$
holds on $\Bbb T$.  This proves the second assertion.

\end{proof}

Let us now prove the Main Theorem. Let $Q_N$ and $Q$ be the Siegel polynomial maps in (\ref{furc}).  Let $D_N$ and $D$
 be respectively the Siegel disks of $Q_N$ and $Q$ centered at the origin.  Let $\lambda, \eta: (0, 2] \to \Bbb R^+$ be the pair of positive functions in the Reduced Main Theorem. Since $Q_N $ converges to $Q$  uniformly in any compact set of the plane, it follows for any pair of integers
$k > m \ge 0$, $\sigma_{k,m}(Q_N)$ converges to $\sigma_{k,m}(Q)$.  Since $\theta_N \to \theta$, thus for any  pair of integers $k > m \ge 0$, if
$$
\delta' <  |e^{2k\pi i \theta} - e^{ 2m\pi i\theta}| <  \delta
$$  for some $0< \delta' < \delta \le 2$, then
$$
\delta' <  |e^{2k\pi i \theta_N} - e^{ 2m\pi i\theta_N}| <  \delta
$$ for all $N$ large enough.  By the Reduced Main Theorem, we thus have \begin{equation}\label{osi-1}\eta(\delta') \le |\sigma_{k,m}(Q_N)| \le \lambda(\delta)\end{equation} for all $N$ large enough.
Since $\sigma_{k,m}(Q_N)$ converges to $\sigma_{k,m}(Q)$ as $N \to \infty$, we have
\begin{equation}\label{osi-2}
\eta(\delta') \le |\sigma_{k,m}(Q)| \le \lambda(\delta).
\end{equation}
By the first assertion of  Lemma~\ref{fundamental principle}, $\Gamma = \overline{\{Q^k(1)\}_{k\ge 0}}$ is a Jordan curve which contains the critical point $1$. By the second assertion of Lemma~\ref{fundamental principle}, $Q:
\Gamma \to \Gamma$ is topologically conjugate to the rigid rotation $R_\theta$. This implies that $\Gamma$ bounds a Siegel disk of rotation number $\theta$. It remains to show $\Gamma = \partial D$,  i.e., the Siegel disk bounded by $\Gamma$ is the one which is centered at  origin. It suffices to prove that the origin is contained in the interior of $\Gamma$. The proof is as follows.

First  note that $Q(0) = 0$  and   $0 \notin \Gamma$.  Let  $\epsilon > 0$ such that $|z| > \epsilon$ for all $z \in \Gamma$. Now for a large integer  $L$ we parameterize $\partial D_N$ and $\Gamma$ as  $\phi_N: \Bbb T \to \partial D_N$ and $\phi: \Bbb T \to \Gamma$ respectively such that
 $\phi_N(e^{2 k \pi i \theta_N}) = Q_N^k(1)$ and $\phi(e^{2 k \pi i \theta}) = Q^k(1)$ for $0 \le k \le L$.
Since $\partial D_N = \overline{\{Q_N^k(1)\}_{k\ge 0}}$ and $\Gamma = \overline{\{Q^k(1)\}_{k\ge 0}}$,    from   (\ref{osi-1}), (\ref{osi-2}),  $Q_N \to Q$  and $\lim \lambda(\delta) \to 0$ as $\delta \to 0+$, by taking $L$ large enough we can make sure that
\begin{equation}\label{control-osi}
|\phi_N(t) - \phi(t)| < \frac{\epsilon}{2}, \: \forall t \in \Bbb T
\end{equation}
holds for all $N$ large enough.  This implies that the homotopy $$H_s(t) = s \phi_N(t) + (1- s)\phi(t), 0 \le s \le 1$$ between $\phi_N$ and $\phi$ does not cross the origin.   It follows that the winding number of $\Gamma$ around the origin is equal to the winding number of $\partial D_N$ around the origin, which is equal to $1$. Thus the origin belongs to the interior of $\Gamma$. This implies  $\Gamma = \partial D$. The  first assertion of the Main Theorem follows.

Now let  $\theta \in \Theta_C$ for some $C > 0$  and $d \ge 3$ be fixed. Let us prove the boundary of the Siegel disks depend continuously on the polynomial maps. To see this,  let $$P(z) = e^{2 \pi i \theta}(z) + a_1 z + \cdots + a_d z^d$$ and $$P_N(z) = e^{2 \pi i \theta} z + a_1^N z + \cdots a_d^N z^d$$ such that for every $1 \le i \le d$,
$a_i^N \to a_i$ as $N \to \infty$.
Let $D_N$ and $D$ denote respectively the Siegel disks of $P_N$ and $P$ which are centered at origin.   It suffices to prove that $\partial D_N$ and $D$ can be parameterized as $\phi_N: \Bbb T \to \partial D_N$ and $\phi: \Bbb T \to \partial D$ respectively such that
$\phi_N(t)$ converges to $\phi(t)$ uniformly in $t\in \Bbb T$. Note that all the critical sets of $P_N$ are contained in a neighborhood of that of $P$ and is thus bounded. Suppose $\partial D_N$ contains a critical point $c_N$ of $P_N$.  By  taking a subsequence we may assume that $c_N \to c$ where $c$ is a critical point of $P$. Now consider polynomial maps $Q_N$ and $Q$ defined by
 $$
 Q_N(z) = c_N^{-1}P_N(c_N z)  \hbox{  and  } Q(z) = c^{-1} P(cz).
 $$
 Then $Q_N \to Q$ as $N \to \infty$.  Since $c_N \to c \ne 0$,  we may assume that $P_N = Q_N$ and $P= Q$. Let $\eta, \lambda: [0, 2) \to \Bbb R^+$ be the two functions in the Reduced Main Theorem. Then (\ref{osi-1}) and  (\ref{osi-2}) still hold for $Q_N$ and $Q$ in this case. Now  for any $\epsilon > 0$, using the same argument as in the proof of (\ref{control-osi}) we can  parameterize $\partial D_N$ and $D$ respectively as
  $\phi_N: \Bbb T \to \partial D_N$ and $\phi: \Bbb T \to \partial D$ such that
$$
|\phi_N(t) - \phi(t)| < \frac{\epsilon}{2}, \: \forall t \in \Bbb T
$$  holds for all $N$ large enough.  This completes the proof of the Main Theorem.

\section{Proof of  Key Lemma 1}

\subsection{Uniform real bounds} In this subsection we introduce Herman-Swiatek's real bounds on critical circle mappings which will be essentially used in this work.  Our presentation follows \cite{P2}. For more details in this aspect,  the reader may refer to \cite{Ch2}, \cite{dFdM},\cite{He} and \cite{P2}.

Let us identify the unit circle with $\Bbb T = \Bbb R/\Bbb Z$ and give
$\Bbb T$ the induced orientation. Let $f: \Bbb T \to \Bbb T$ be a
homeomorphism. Let $(a, b, c, d)$ be a quadruple with $a < b < c < d
< a+1$. Define the cross-ratio
$$
[a, b, c, d] = \frac{b-a}{c - a}\frac{d-c}{d-b}
$$
and the cross-ratio distortion
$$
D(a, b, c, d, f) = \frac{[f(a), f(b), f(c), f(d)]}{[a, b, c, d]}.
$$

Let $d \ge 2$ be an integer. Let
\begin{equation}\label{e-H}
\mathcal{H}_d = \big{\{}g(z)= \lambda z^d \prod_{i=1}^{d-1} \frac{1
- \overline{a}_i z}{z - a_i}, 0< |a_i| < 1, g|_{\Bbb T} \hbox{  is a
homeomorphism}\}.
\end{equation}

\begin{lem}[Herman, \cite{He}, see also \cite{Ch2}] \label{HC1}There is a $0< C(d) < \infty$ depending only on $d$ such that
for any $g \in \mathcal{H}_d$, any
integer $m \ge 1$ and any finite family of quadruples $(a_i, b_i, c_i, d_i)$, $1 \le i \le n$,   if
$$\sup_{x\in \Bbb T} \#\{(a_i, d_i)\:|\: x \in (a_i, d_i)\} \le m,$$ then
$$
\prod_{i=1}^n D(a_i, b_i, c_i, d_i, f) < C(d)^m.
$$
\end{lem}

\begin{lem}[Uniform power law] \label{HC2}
There exist constants $\nu, C' > 1$ such that for any $g \in
\mathcal{H}_d$, if $c \in \Bbb T$ is a critical point of $g$, then
for any $x, y \in \Bbb T$ with $|x - c| \le |y - c|$, we have
$$
\bigg{|}\frac{g(x) - g(c)}{g(y) - g(c)}\bigg{|} \le C' \cdot
\bigg{|}\frac{x -c}{y-c}\bigg{|}^{\nu}.
$$
\end{lem}

\begin{proof}
Since $\mathcal{H}_d$ is compact, it suffices to prove the lemma in the case that $y$ and $x$ are both close to $c$.
 Let $I$  and $J $ denote respectively the smaller arc intervals which connect $y$ to $c$, and $x$ to $c$. Then $|J| \le  |I|$ by assumption. By Lemma 15 of \cite{He},   there exists an open neighborhood  of $\Bbb T$  depending only on $d$  on which
$\mathcal{H}_d$  is  a normal family.   Thus there exist  $0< m < M$ and an $0< r < 1$
which depend only on $d$ such that for any $g \in \mathcal{H}_d$
 there exists  an open neighborhood $U$ of $\Bbb T$ which contains $\{z\:|\: 1-r < |z|< 1 +r \}$ and a
 holomorphic function $\phi$ defined on $U$    such that
$$
g'(z) = \phi(z) \cdot  (z - c)^l \cdot \prod_{i\in \Lambda} (z -
c_i)^{l_i} \cdot \prod_{i\in \Theta} (z - c_i)^{l_i}
$$
where
\begin{itemize}
\item[1.] $m \le |\phi(z)| \le M$ for all $z \in U$,
\item[2.] $\{c\} \cup \{c_i, i \in \Lambda \cup \Theta\} = \{z \in U\:|\: g'(z) = 0\}$,

\item[3.]
${\rm dist}(c_i, c) \ge 2 |I|$ for all $i \in \Lambda$ and   ${\rm dist}(c_i, c) < 2 |I|$ for all $i \in \Theta$.
\end{itemize}

Note that  except $0$ and
$\infty$ $g$ has  $2(d-1)$ critical points, counting by multiplicities.  So
 there exist constants $0< \kappa(d), \eta(d) < 1$ depending only on $d$
and a sub-interval $I_1 \subset I$  such that
\begin{itemize}
\item[i.] $|I_1| > \eta(d) \cdot |I|$,
\item[ii.] ${\rm dist}(c, I_1) > \eta(d) \cdot |I|$,
\item[iii.] for any $i \in \Theta$, ${\rm dist}(c_i, I_1) > \kappa(d) \cdot |I|$.
\end{itemize}

For any $i \in \Theta$ and $z \in J$, it is clear that  ${\rm dist}(c_i, z) \le {\rm dist}(c_i, c) +  {\rm dist}(c, z)$.  Since
${\rm dist}(c_i, c) < 2|I|$ and ${\rm dist}(c, z) \le |J| \le |I|$, we have   \begin{equation}\label{tha}\sup_{z \in
J}{\rm dist}(c_i, z)  < 3|I|.\end{equation}

For any
$i \in \Lambda$ and $z \in J$,  it is clear that ${\rm dist}(c_i, z) \le  {\rm dist}(c_i, I_1) + {\rm dist}(z, I_1) + |I_1|$.  This, together with ${\rm dist}(z, I_1) \le {\rm dist}(z, c)+{\rm dist}(c, I_1) \le |J| + |I| \le 2|I|$  and $${\rm dist}(c_i, I_1) \ge {\rm dist}(c_i, c) - \max_{z\in I_1} {\rm dist}(c, z) > {\rm dist}(c_i, c) - |I| \ge |I|,$$  impliest that for $i \in \Lambda$,
\begin{equation}\label{lmd}
\sup_{z \in J}{\rm dist}(c_i, z) < {\rm dist}(c_i, I_1) + \sup_{z \in J}{\rm dist}(z, I_1) +|I_1|< {\rm dist}(c_i, I_1) + 3 |I| < 4\:{\rm dist}(c_i, I_1).
\end{equation}

Now from the intermediate value theorem we have $\xi \in I_1$ and
$\zeta \in J$ such that
$$
|g(I)| \ge |g(I_1)| = |g'(\xi)| \cdot |I_1| = |\phi(\xi)| \cdot
|\xi - c|^l \cdot \prod_{i\in \Lambda} |\xi - c_i|^{l_i} \cdot
\prod_{i\in \Theta} |\xi - c_i|^{l_i}\cdot |I_1|,
$$
and
$$
|g(J)|  = |g'(\zeta)| \cdot |J| = |\phi(\zeta)| \cdot  |\zeta - c|^l
\cdot \prod_{i\in \Lambda} |\zeta - c_i|^{l_i} \cdot \prod_{i\in
\Theta} |\zeta - c_i|^{l_i}\cdot |J|.
$$

Since $g$ is of degree of $2d-1$
we have  $\sum_{i \in \Lambda} l_i \le 2d-2$ and $\sum_{i\in \Theta} l_i \le 2d-2$.    This, together with (i-iii) and (\ref{tha}-\ref{lmd}), implies
$$
\frac{|g(J)|}{|g(I)|} \le \frac{M}{m} \cdot
\frac{|J|^{l+1}}{\eta(d)^{l+1} \cdot |I|^{l+1} }\cdot 4^{2 d -2} \cdot
\bigg{(}\frac{3}{\kappa(d)}\bigg{)}^{2d-2}.
$$
Since $2 \le l \le 2d-2$ the lemma follows by taking  $\nu = 3$
and
$$
C' = \frac{M \cdot 4^{4d-4}} {m \cdot \eta(d)^{2d-1} \cdot
\kappa(d)^{2d-2}}.
$$

\end{proof}

Let  $B \in \mathcal{H}_{d}$ such that the rotation number of $B|\Bbb T: \Bbb T \to \Bbb T$
is an irrational number and the point $1$ is a critical point of $B$.
It is necessary that the local degree of $B$ at $1$ is odd and not
less than $3$.  Let $p_n/q_n, n\ge 0$, denote the convergents of the rotation number.
 Let $x_i$ denote the point
in $\Bbb T$ such that $B^{i}(x_i) = 1$. Let
$$
I_n = [1, x_{q_n}] \hbox{   and   } I_{n+1} = [1, x_{q_{n+1}}].
$$
Let $I_{n}^i, i\ge 0$, denote the subintervals of $\Bbb T$ such that
$B^i(I_{n}^i) = I_n$. Then the collections of the intervals
\begin{equation}\label{dnp}
\Pi_n(B) = \{I_{n}^i, \:\:0 \le i \le q_{n+1} -1;\:\:\: I_{n+1}^i,\:\: 0 \le i \le q_n -1\}
\end{equation}
form a partition of $\Bbb T$. We call $\Pi_{n}(B)$ the $\emph{dynamical partition}$ of
level $n$.  For $K > 1$ and $I, J \subset \Bbb T$, we say $I$ and $J$ are $K$-$\emph{commensurable}$ if $|J|/K < |I| < K |J|$.

\begin{thm}[Herman-Swiatek's uniform real bounds]\label{real bounds}
There is a $1 < C(d) < \infty$ depending only on $d$ such that for
any $B \in \mathcal{H}_{d}$ with an irrational  rotation number  and a
critical point at $1$, we have
\begin{itemize}

\item[1.]  for any $x \in \Bbb T$ and all $n \ge 1$ the two intervals $[x, B^{q_n}(x)]$ and $[x, B^{-q_n}(x)]$ are
$C(d)$-commensurable,
\item[2.] for all $n \ge 1$,
any two adjacent intervals in the dynamical partition of  level $n$
are $C(d)$-commensurable.
\end{itemize}
\end{thm}
\begin{proof}
The first assertion is implied by  Proposition 3.3 of \cite{P2}. The second assertion is implied by Proposition 3.3,  and Theorem 3.5 of \cite{P2}. We only need to notice that the
constants  guaranteed by Proposition 3.3  and Theorem 3.5 of \cite{P2}   depend only on  the constants $C, C'$ and $\nu$ in the two
assumptions of  the Hypothesis 1 of \cite{P2}. By Lemmas~\ref{HC1} and
~\ref{HC2},   we can take these three constants depending only on $d$ so that the Hypothesis 1 of \cite{P2} is satisfied.
\end{proof}

\begin{rmk}\label{lr-1}{\rm Since each interval in $\Pi_{n+2}$ is a proper sub-interval of some interval in $\Pi_n$, from Theorem~\ref{real bounds}, it follows that there exist $\lambda(d) > 0$ and $0< \delta(d) < 1$ depending only on $d$ such that   for any $B \in \mathcal{H}_{d}$ with irrational  rotation number  and a critical point at $1$, we have $|I| < \lambda(d) \cdot \delta(d)^n$  for all $n \ge 1$ and all  intervals $I$ in $\Pi_n(B)$. Since any interval $[x, B^{q_n}(x)]$ is contained in the union of two adjacent intervals in $\Pi_{n-1}(B)$, by taking $\lambda(d) > 0$ larger, we may always have  $|[x, B^{q_n}(x)]| < \lambda(d) \cdot \delta(d)^n$.  }
\end{rmk}

\subsection{ The family $\mathcal{S}_d^\alpha$} Let $0< \alpha< 1$ be an irrational number and $d \ge 2$ be an integer. Let $\Gamma$ be a Jordan curve. Suppose
 $f: \Gamma \to \Gamma$ is a  homeomorphism which is conjugate to the rigid rotation $R_\alpha: \Bbb T \to \Bbb T$, that is, there exists a homeomorphism
 $\phi: \Gamma \to \Bbb T$  such that
$$
f =\phi^{-1} \circ  R_\alpha \circ \phi.
$$
For any two points $x, y \in \Gamma$ we define the $\emph{dynamical angle}$ between  $x$ and
$y$  to be the angle between $\phi(x)$ and $\phi(y)$ anticlockwise.

\begin{lem}\label{basic-m}
For any tuple $(\alpha_1, \cdots,
\alpha_{d-1})$ with $0\le \alpha_i < 2 \pi$ and $\sum_{i=1}^{d-1}
\alpha_i = 2 \pi$, there exists a $B \in \mathcal{H}_d$ such that \begin{itemize}\item[1.]
there are exactly $d-1$ critical points $c_1 =1, c_2, \cdots, c_{d-1}$ in $\Bbb T$, ordered anticlockwise and  counted by multiplicities, \item[2.]
   the dynamical angle from $c_i$ to $c_{i+1}$ anticlockwise  is $\alpha_i$ for $1 \le i \le d-1$ (we identify $c_1$ with $c_{d}$), \item[3.]   the
rotation number of $B|\Bbb T: \Bbb T \to \Bbb T$ is $\alpha$. \end{itemize}
\end{lem}
\begin{proof}
We will construct  such $B$  in the proof of Theorem~\ref{Thurston-Siegel-Ch}, cf. $\S6$.
 \end{proof}
Let $\mathcal{S}_d^\alpha \subset \mathcal{H}_d$ denote the class of all
the Blaschke products  guaranteed by Lemma~\ref{basic-m}.  In $\S6$, we shall see that for any $\alpha \in \Theta_C^b$, each $f \in \Pi_\alpha^d$ is uniquely determined by the group of angles $(\alpha_1, \cdots, \alpha_{d-1})$ formed by the $d-1$ critical points on the boundary of the Siegel disk. Thus for any  $f \in \Pi_\alpha^d$,  we can find  a $B \in \mathcal{S}_d^\alpha$ such that $f$ is obtained  through a qc surgery on $B$.

\subsection{Uniform saddle node geometry} Most of the arguments and ideas in this subsection are adapted from  \cite{dFdM} and \cite{PZ}. The difference is that in \cite{dFdM} and \cite{PZ}, the critical circle mappings have only one critical point in $\Bbb T$, and  here  the maps $B \in  \mathcal{S}_d^\alpha$ may have several critical points in $\Bbb T$.

Let $\alpha \in \Theta_C^b$ and $B \in \mathcal{S}_d^\alpha$. Then there is a circle homeomorphism $h: \Bbb T \to \Bbb T$ such that $$B|\Bbb T = h^{-1} \circ R_{\alpha} \circ h.$$
 The aim of this subsection is to show that  the conjugation map  $h$  exhibits a  uniform saddle node geometry which depends only on $d$ and $C$.

Let us recall some notations first.  For $i \ge 0$,  let  $x_i \in \Bbb T$  denote
the point such that $B^i(x_i) = 1$. For $n \ge 0$, let  $p_n/q_n$  denote the $n$-th convergent of $\alpha$  and $\Pi_n(B)$  denote the collection of intervals in the dynamical partition of level $n$,  cf.  (\ref{dnp}). Now for  each
$n \ge 0$,  define
$$
\mathcal{Q}_n = \{x_i\:|\: 0 \le i < q_n\}.
$$
Then $\mathcal{Q}_0 = \{1\}$.  The following proposition is summarized from $\S6.1$  of \cite{PZ}. Note that in \cite{PZ} the
circle homeomorphism is induced by the Douady-Ghys Blaschke model which contains exactly one (double) critical point at $1$.  Since Proposition~\ref{crp} only involves the combinatorial information about the rotation number and is  independent of the number of the critical points in $\Bbb T$, it still holds for $B \in \mathcal{S}_d^\alpha$.

\begin{pro}[cf. \S6.1 of \cite{PZ}]\label{crp}{\rm  Let $0 \le j < k < q_n$. Then
 $x_j$ and $x_k$  are adjacent in $\mathcal{Q}_n$ if and only if $k = j + q_{n-1}$
and $0 \le j <  q_{n} - q_{n-1}$, or $k = j + q_n - q_{n-1}$ and $0
\le j < q_{n-1}$. In the former case we have
\begin{equation}\label{ref-a}
[x_k, x_j] \cap \mathcal{Q}_{n+1} = \{x_k, x_{k+q_n}, x_{k+2q_{n}},
\cdots, x_{k+{(a_{n+1}-1)}q_{n}}, x_j\},
\end{equation} and in the later case we have
\begin{equation}\label{ref-b}
[x_j, x_k] \cap \mathcal{Q}_{n+1} = \{x_j, x_{j+q_{n}},
x_{j+2q_{n}}, \cdots, x_{j+a_{n+1}q_{n}}, x_k\}.
\end{equation}
Moreover,  each interval in $\Bbb T\setminus \mathcal{Q}_n$  either is a single  interval in $\Pi_{n-1}(B)$, or is the union of two adjacent intervals in $\Pi_{n-1}(B)$.
In particular, any two adjacent intervals in $\Bbb T \setminus \mathcal{Q}_n$ are $K$-commensurable with $K > 1$ being some constant depending only on $d$.
}
\end{pro}
The last assertion of Proposition~\ref{crp} is implied by the second assertion of  Lemma~\ref{real bounds}  that any two adjacent intervals in $\Pi_{n-1}(B)$ are $C(d)$-commensurable with $C(d) > 1$ being some constant depending only on $d$.

When $d = 2$, $B$ is exactly the Douady-Ghys Blaschke model
considered in \cite{PZ}. In this case the point $1$ is the unique
critical point of $B$ in $\Bbb T$. It is  clear  that in this case for every interval component $I$ of $\Bbb T \setminus \mathcal{Q}_n$,  the map
$$
B^{q_n}: I \to B^{q_n}(I)
$$ is a diffeomorphism.

  For $d > 2$,
any $B \in S_{d}^\alpha$ has   $(d-1)$  critical points in $\Bbb T$, counting by multiplicities.  If all these critical points collide into one critical point at $1$, then for any interval component $I$ of $\Bbb T \setminus \mathcal{Q}_n$, the map
 $B^{q_n}: I \to B^{q_n}(I)$ is still a diffeomorphism.  Otherwise, there are $1 \le d' \le d-2$ distinct critical points other than $1$.
Let us denote them by $c_i, 1 \le i \le d'$ and denote $1$ by $c_0$.
For  any integer $k \ge 0$,
let $c^k_i \in \Bbb T$ denote the point such that $B^k (c^k_i) =
c_i$.  Let
$$
\mathcal{Q}_n^i  = \{c_i^k\:|\:  \:\:0 \le k < q_n\}, \: 0 \le i \le
d', \:n \ge 1.
$$ Then $\mathcal{Q}_n^0 = \mathcal{Q}_n = \{x_i\:|\: 0\le i < q_n\}$.

\begin{lem}\label{sub-int} Let $n \ge 1$. Then for each $1 \le i \le d'$, we have  \begin{itemize} \item[1.]
in the case that $k = j + q_{n-1}$ and $0 \le j <  q_{n} - q_{n-1}$,
the interval $(x_k, x_j)$ contains at most one point in
$\mathcal{Q}_n^i$, \item[2.]
 in the case that $k = j + q_n - q_{n-1}$ and $0 \le j < q_{n-1}$, the interval $(x_j, x_k)$ contains at most two points in $\mathcal{Q}_n^i$. \end{itemize}
\end{lem}
\begin{proof}  Fix an $1 \le i \le d'$. By replacing $\mathcal{Q}_n$ by $\mathcal{Q}_n^i$ in
 Proposition~\ref{crp}, it follows that each component of $\Bbb T \setminus \mathcal{Q}_n^i$ either has the form $(c^m_i, c^l_i)$, where
$m = l + q_{n-1}$ and $0 \le l <  q_n - q_{n-1}$, or has the form
$(c^l_i, c^m_i)$ where $m = l + q_{n} - q_{n-1}$ and $0 \le l <
q_{n-1}$.

Suppose $k = j + q_{n-1}$ and $0 \le j < q_n - q_{n-1}$.  By the property of closest returns, $(x_k,
x_j)$  can not  contain any interval either  of the form $[c_i^m, c_i^l]$
with $m = l + q_{n-1}$ and $0 \le l <  q_n - q_{n-1}$, or  of the form $[c_i^l, c_i^m]$
with $m = l + q_n - q_{n-1}$ and $0 \le l <  q_{n-1}$.   It follows  that $(x_k, x_j)$ contains at most one point in $\mathcal{Q}_n^i$. This  proves the
first assertion.

Suppose  $k = j + q_n - q_{n-1}$ and $0 \le j < q_{n-1}$. Again by the property of closest returns,
 the interval $(x_j, x_k)$ can  contain at most one interval of the form $[c^m_i, c^l_i]$  with
$m = l + q_{n-1}$ and $0 \le l \le q_n - q_{n-1}$ and can not contain any interval with the form $[c_i^l, c_i^m]$
with $m = l + q_n - q_{n-1}$ and $0 \le l <  q_{n-1}$.  It follows  that $(x_j, x_k)$ can contain at most two
points in $\mathcal{Q}_n^i$. This proves the second assertion.
\end{proof}

Let $I$ be an interval component of $\Bbb T \setminus \mathcal{Q}_n$.  By Proposition~\ref{crp} the points in $\mathcal{Q}_{n+1}$ divide $I$ into finitely many sub-intervals  and  any two such sub-intervals are $K$-commensurable for some  constant $K > 1$ depending only on $d$. When the number of such sub-intervals are large, however,  the ones which lie in the middle position are very small compared to the ones near the end.   This is the so called ``saddle-node" geometry.  More precisely,
\begin{lem}[cf. Theorem 6.6 of \cite{PZ}]\label{DGS}
Suppose $d = 2$, that is, $B|\Bbb T$ has only one double critical point at $1$.  Then there is a universal constant $K > 1$ such that for any interval component $I$ of $\Bbb T \setminus \mathcal{Q}_n$, if the points in $\mathcal{Q}_{n+1}$ divide $I$ into sub-intervals
$$
I_1, \cdots, I_m,
$$
then we have
$$
\frac{1}{K} \cdot \frac{|I|}{\min\{k, m-k+1\}^2} < |I_k| < K \cdot \frac{|I|}{\min\{k, m-k+1\}^2}
$$
\end{lem}

Now suppose $d > 2$ and $\alpha \in \Theta_C^b$. Let  $B \in \mathcal{S}_d^\alpha$. Suppose $B|\Bbb T$ has at least one critical point other than $1$, that is, $d' \ge 1$.  In the case that all the critical points of $B|\Bbb T$  collapse into one single point at $1$, that is, $d' = 0$, the above lemma still holds for some constant $K > 1$ depending only on $d$.  This can be derived by taking the limit in the following lemma.
 \begin{lem}[Uniform Saddle Node Geometry]\label{U-S-G}
There is a $K > 1$ depending only on $d$ such that the following holds.   Suppose $B \in \mathcal{S}_d^\alpha$   such that $B$ has  $d' \ge 1$ distinct critical points in $\Bbb T$ other than the critical point $1$.   Then for any   component $I$ of $\Bbb T \setminus \mathcal{Q}_n$, if  $J$ is a component  of
$$
I \setminus \bigcup_{i=1}^{d'} \mathcal{Q}_n^i
$$
which  contains at least one interval component of $I \setminus \mathcal{Q}_{n+1}$,
we have
\begin{equation}\label{fas}
\frac{1}{K} \cdot |I| < |J| \le |I|.
\end{equation}  Moreover, if $$J_1, \cdots, J_m$$ denote
all the interval components of $I \setminus \mathcal{Q}_{n+1}$ contained in $J$, labeled by order,
then we have
\begin{equation}\label{sas}
\frac{1}{K} \cdot \frac{|J|}{\min\{k, m-k+1\}^2} < |J_k| < K \cdot \frac{|J|}{\min\{k, m-k+1\}^2}.
\end{equation}
\end{lem}

As in  the proof of Lemma~\ref{DGS} (cf. \cite{dFdM}), the basic tool used in our  proof of Lemma~\ref{U-S-G} is Yoccoz's almost parabolic lemma.
  Before we state this lemma let us introduce a terminology first.
  Let $n$  be a positive integer and  $I_1, \cdots, I_n$ be
consecutive intervals on the line or circle. According to \cite{dFdM},  by an $\emph{almost parabolic
map}$ of length $n$ and  fundamental domains $I_1, \cdots, I_n$,   we
mean a negative-Schwarzian diffeomorphism $$f: I_1 \cup \cdots I_n
\to I_2 \cup \cdots \cup I_{n+1}$$ such that $f(I_j) = I_{j+1}$. The
basic geometric estimate on almost parabolic maps is

\begin{lem}[Yoccoz's almost parabolic lemma, cf. \cite{dFdM}]\label{Y-L}  Suppose that $I = \bigcup_{i=1}^{n} I_i$
 and  $f : I \to  f(I)$ is an almost parabolic map of
length $n$  with  fundamental domains $I_j, 1 \le j \le n$.  If $|I_1|
\ge \sigma \cdot |I|$ and $|I_n| \ge \sigma \cdot |I|$ for some
$\sigma > 0$, then
$$
\frac{1}{C_\sigma } \frac{|I|}{\min\{j, n-j+1\}^2} \le |I_j| \le
{C_\sigma} \frac{|I|}{\min\{j, n-j+1\}^2}
$$
where $C_\sigma > 1$ is a constant depending only on $\sigma$.
\end{lem}

Let us begin the proof of Lemma~\ref{U-S-G}. The first assertion of Lemma~\ref{U-S-G} is implied by the following lemma.  Recall that  $\mathcal{Q}_n^0 = \mathcal{Q}_n$.

\begin{lem}\label{ubc} Let $I$ and $J$ be the intervals in Lemma~\ref{U-S-G}.  Then
 for any interval component $S$ of $I \setminus \mathcal{Q}_{n+1}$, if
 $$
 \overline{S} \cap \bigcup_{i=0}^{d'}
\mathcal{Q}_{n}^i \ne \emptyset,$$ then $|S| > |I|/K$ where $K > 1$ is some constant depending only on $d$.
\end{lem}
\begin{proof}
The argument is  standard. By Proposition~\ref{crp} we have two cases: either $I = (x_k, x_j)$ where $k = j + q_{n-1}$
and $0 \le j <  q_{n} - q_{n-1}$, or  $I = (x_j, x_k)$ where $k = j + q_n - q_{n-1}$ and $0
\le j < q_{n-1}$.   In the first case, by (\ref{ref-a}) $S$ either has the form $(x_{k+lq_n}, x_{k+(l+1)q_n})$ for some $0 \le l \le a_{n+1}-2$, or has the form $(x_{k+(a_{n+1}-1)q_n}, x_j)$.  In the second case, by (\ref{ref-b}) $S$ either has the form $(x_{j+lq_n}, x_{j+(l+1)q_n})$ for some $0 \le l \le a_{n+1}-1$ or has the form $(x_{j+a_{n+1} q_n}, x_k)$. Since the proofs of all these four subcases are similar to each other, let us only deal with the first subcase. With very minor changes of the argument, the reader shall easily prove the remaining three subcases.

Now let us suppose $I = (x_k, x_j)$ where $k = j + q_{n-1}$
and $0 \le j <  q_{n} - q_{n-1}$ and $S = (x_{k+lq_n}, x_{k+(l+1)q_n})$ for some $0 \le l \le a_{n+1}-2$.
Since $\overline{S} \cap \bigcup_{i=0}^{d'}
\mathcal{Q}_{n}^i \ne \emptyset$, there is a least integer $0 \le t \le q_n-1$  and a point $x \in \overline{S}$ such that $B^t(x) = c_i$ for some $0 \le i \le d'$.  For $k \ge 0$, recall that $c_i^k$ denote the point in $\Bbb T$ such that $B^k(c_i^k) = c_i$.

Consider the following two group of intervals
$${\rm I.}\:\: [c_i, c_i^{q_n}], [c_i^{q_n}, c_i^{q_n - q_{n-1}}], [c_i^{q_n - q_{n-1}}, c_i^{q_n - 2q_{n-1}}]$$
and
$${\rm II.}\:\: [c_i^{2q_{n-1}}, c_i^{q_{n-1}}], [c_i^{q_{n-1}}, c_i], [c_i, c_i^{q_n}].$$
In the case that $q_n = q_{n-2} + q_{n-1}$, we replace the last interval in the first group by $[c_i^{q_{n-2}}, c_{i}^{q_{n-2} + q_n}]$.
These intervals belong to the collection of the intervals of the dynamical partition of level $n-1$ with respect to the critical point $c_i$, and moreover, they are adjacent to each other.  By Theorem~\ref{real bounds}, these  intervals are $C(d)$-commensurable with $C(d)> 1$ being  some constant depending only on $d$. Thus the cross ratios of both $(c_i, c_i^{q_n}, c_i^{q_n - q_{n-1}}, c_i^{q_n - 2q_{n-1}})$ and $(c_i^{2q_{n-1}}, c_i^{q_{n-1}}, c_i, c_i^{q_n})$ have a lower bound $\kappa(d) > 0$ depending only on $d$. Pull back the two group of intervals by $B^{-t}$. We get the following two group of intervals
$$ {\rm I'.}\:\: [x, B^{-q_n}(x)], [B^{-q_n}(x), B^{-q_n +q_{n-1}}(x)], [B^{-q_n +q_{n-1}}(x), B^{-q_n +2q_{n-1}}(x)]$$ and
$${\rm II'.}\:\:[B^{-2q_{n-1}}(x), B^{-q_{n-1}}(x)], [B^{-q_{n-1}}(x), x], [x, B^{-q_n}(x)].$$  Since $ 0\le t < q_n$, the pull backs of each interval in I and II by $B^i, i = 0, 1, \cdots, t$, are disjoint.  Thus the intersection multiplicity  of the pull backs of each of the two groups  is not greater than $3$.
Now by Lemma~\ref{HC1},  it follows that the cross ratios of both $$(x, B^{-q_n}(x), B^{-q_n +q_{n-1}}(x), B^{-q_n +2q_{n-1}}(x))$$ and $$(B^{-2q_{n-1}}(x), B^{-q_{n-1}}(x),  x,  B^{-q_n}(x))$$ have a positive lower bound $\eta(d) > 0$ with $\eta(d) > 0$ being some constant depending only on $d$. This then implies that
\begin{equation}\label{inc}
|[x, B^{-q_n}(x)]|> \lambda(d) \cdot|[B^{-q_n}(x), B^{-q_n +q_{n-1}}(x)]|\hbox{  and  } |[x, B^{-q_n}(x)]| > \lambda(d) \cdot | [B^{-q_{n-1}}(x), x]|
\end{equation} with $\lambda(d) > 0$ being some constant depending only on $d$.

By assumption we have  $I = (x_k, x_j) = (B^{-q_{n-1}}(x_j), x_j)$ and $x \in \overline{S} \subset \overline{I}$. Thus  we have $$I \subset   [B^{-q_{n-1}}(x), x]\cup[x, B^{-q_n}(x)] \cup [B^{-q_n}(x), B^{-q_n +q_{n-1}}(x)]. $$  This, together with (\ref{inc}), implies   that
\begin{equation}\label{b-1}
|[x, B^{-q_n}(x)]| > \frac{\lambda(d)}{2+\lambda(d)} \cdot |I|.
\end{equation} By assumption $\overline{S} = [x_{k+lq_n}, x_{k+(l+1)q_n}]$.
By the first assertion of   Theorem~\ref{real bounds} we have
\begin{equation}\label{check-1}
[x_{k+(l+1)q_n}, x_{k+(l+2)q_n}]| < C(d) \cdot |[x_{k+lq_n}, x_{k+(l+1)q_n}]| = C(d) \cdot |S|
 \end{equation}where $C(d) > 1$ is some constant depending only on $d$. Since $x \in \overline{S} = [x_{k+lq_n}, x_{k+(l+1)q_n}]$,  we have
\begin{equation}\label{b-2}
[x, B^{-q_n}(x)] \subset \overline{S} \cup [x_{k+(l+1)q_n}, x_{k+(l+2)q_n}].
\end{equation}
From (\ref{b-1})-(\ref{b-2}) we have
$$
|S| > \frac{1}{1 + C(d)} \cdot |[x, B^{-q_n}(x)]| > \frac{1}{1 + C(d)} \cdot \frac{\lambda(d)}{2+\lambda(d)} \cdot |I|.
$$
\end{proof}

Now let us prove the second assertion of  Lemma~\ref{U-S-G}.
Let $J_1, \cdots, J_m$ be the intervals in Lemma~\ref{U-S-G}. Since any two adjacent interval components in $\Bbb T \setminus \mathcal{Q}_{n+1}$ are $K$-commensurable for some $1 < K < \infty$ depending only on $d$,  it suffices to assume that  $m \ge 4$ and prove $J_3, \cdots, J_{m-1}$ satisfies the uniform saddle node geometry described by  (\ref{sas}).
 Let us consider the diffeomorphism $$B^{q_n}: J_3 \cup \cdots \cup J_{m-1} \to J_2 \cup \cdots \cup J_{m-2}.$$
From Lemma~\ref{Y-L},   we need only to check two conditions. The first one is to show that the two boundary sub-intervals, that is, $J_3$ and $J_{m-1}$, are uniformly commensurable with the whole interval $J_3 \cup \cdots \cup J_{m-1}$. The second one is to show that $B^{q_n}$ has negative Schwarze derivative on $J_3\cup \cdots J_{m-1}$.  Since $J \supset J_3 \cup \cdots \cup J_{m-1} \supset J_3 \cup J_{m-1}$, the following lemma implies the first condition.

\begin{lem}\label{u-b-c}
There exists a $\sigma > 0$ depending only on $d$ such that $|J_3| > \sigma\cdot |J|$ and $|J_{m-1}| > \sigma \cdot |J|$.
\end{lem}
\begin{proof}
Let us prove the first inequality only. The second one can be proved by the same argument.  Since $J_1$,  $J_2$ and $J_3$ are interval components in $\Bbb T \setminus \mathcal{Q}_{n+1}$ and adjacent to each other, $J_3$ is $K$-commensurable with $J_1$ for some $1 < K < \infty$ depending only on $d$. It suffices to prove that $|J_1| > \sigma|J|$ for some $\sigma > 0$ depending only on $d$. There are two cases. In the first case $J_1$ has a common boundary point with $J$. Then the boundary point must be a point in $\bigcup_{i=0}^{d'} \mathcal{Q}_n^i$. By Lemma~\ref{ubc}, we have $|J_1| > |I|/K > |J|/K$ for some $K > 1$ depending only on $d$. In the second case, $J_1$ is adjacent to an interval component of $I \setminus \mathcal{Q}_{n+1}$, say $S$,  which contains a boundary point of $J$. Again this boundary point must be a point in $\bigcup_{i=0}^{d'} \mathcal{Q}_n^i$.  Then we get $|J_1| > |S|/K$ by Proposition~\ref{crp}  and  get $|S| > |I|/K > |J|/ K$ by Lemma~\ref{ubc}  where  $K> 1$ is some constant depending only on $d$. This implies that $|J_1|  > |J|/K^2$.  The same argument can be used to prove that $|J_{m-1}| > \sigma \cdot |J|$ for some $\sigma > 0$ depending only on $d$. This proves the lemma.
\end{proof}

It remains to prove that $B^{q_n}$ has negative Schwarz derivative on $J_3 \cup \cdots \cup J_{m-1}$.
Here when we talk about the Schwarz derivatives of the iterations of $B$,  we regard $\Bbb T$ as $\Bbb R/\Bbb Z$ and $B: \Bbb T \to \Bbb T$ as its lift $\widetilde{B}: \Bbb R \to \Bbb R$ and regard the intervals $J_i$ in $\Bbb T$ as its lift $\tilde{J}_i$ in $\Bbb R$. In this way, $B$ is real analytic in a strip neighborhood of $\Bbb R$, and moreover, $B(x + 1) = B(x) + 1$. To simplify the notation we still use $B$ and $J_i$ to denote these objects.

 \begin{lem}\label{uniform-dr}
 There is an $M > 1$ depending only on $d$ such that for any $x$ and $y$ in  $J_3 \cup \cdots \cup J_{m-1}$ and all $1 \le k \le q_n$, we have
 $$
 M^{-1} < \frac{DB^{k}(x)}{DB^{k}(y)} < M.
 $$
 \end{lem}
\begin{proof} It is known that the map
$B^{q_n}: J_2 \cup \cdots \cup J_{m} \to J_1 \cup \cdots \cup J_{m-1}$  is a  diffeomorphism. That is to say, $J_1 \cup \cdots \cup J_{m-1}$ contains no critical values of $B^{q_n}$ in its interior.  Since $J_1$ and $J_{m-1}$ are $K$-commensurable with $
J_1 \cup \cdots \cup J_{m-1}$ with  $1< K <\infty$ being a constant depending  only on $d$ (cf. the proof of Lemma~\ref{u-b-c}),  there is a Jordan domain $U$ in the punctured plane $\Bbb C \setminus \{0\}$  such that $U \cap \Bbb T = J_1 \cup \cdots J_{m-1}$ and the modulus of the
annulus $U \setminus \overline{J_2 \cup \cdots \cup J_{m-2}}$ has a positive lower bound depending only on $d$.  Note that  $U$ does not intersect the critical values of  of $B^{q_n}$. So $B^{-q_n}$ can be holomorphically extended to a univalent function on $U$ which maps
$J_1 \cup \cdots \cup J_{m-1}$ to $J_2\cup\cdots \cup J_m$.   Let $V$ be the component of
$B^{-q_n}(U)$ which contains $J_2\cup \cdots \cup J_m$. Then the modulus of $V \setminus \overline{J_3 \cup \cdots \cup J_{m-1}}$ is equal to that of $U \setminus \overline{J_2\cup \cdots\cup J_{m-2}}$ and thus has a positive lower bound depending only on $d$.  It is clear that for every $1 \le k \le q_n$, the map  $B^k$ is univalent  in $V$.   The lemma then follows from  Koebe's distortion theorem.
\end{proof}

\begin{lem}\label{neg-s} There is an $N \ge 1$ such that for any $\alpha \in \Theta_C^b$, any $B \in \mathcal{S}_d^\alpha$ and every $n \ge N$,  if  $J_i$, $1 \le i \le m$, are the intervals in Lemma~\ref{U-S-G}, then $S(B^{q_n})(z) < 0$ for all
 $z \in J_3 \cup \cdots \cup J_{m-1}$, where $S(\cdot)$ denotes the Schwarz derivative.

\end{lem}
\begin{proof}  Let $\mathcal{H}_d$ be the family of Blaschke products defined in (\ref{e-H}). Let $\mathcal{S}_d \subset \mathcal{H}_d$ be the subfamily which contains all $B \in \mathcal{H}_d$ such that all the critical points of $B$, except $0$ and $\infty$, are contained in $\Bbb T$. By the compactness property of $\mathcal{H}_d$ (cf. $\S 15$ of \cite{Ch}), $\mathcal{S}_d$ is compact in the sense that there exits an annular neighborhood $H$ of $\Bbb T$, such that any $B$ in $\mathcal{S}_d$ is holomorphic in $H$, and moreover, for any sequence $\{B_n\}$ in $\mathcal{S}_d$, there is a subsequence $\{B_{n'}\}$ and a $B \in \mathcal{S}_d$ such that $B_{n'}$ converge to $B$ uniformly in any compact set of $H$.

  Recall that  each $B \in \mathcal{H}_d$ can be regarded as a holomorphic function defined in a strip neighborhood of $\Bbb R$  such that $B: \Bbb R \to \Bbb R$ is a homeomorphism and  $B(x+1) = B(x) + 1$.
By the compactness property of $\mathcal{S}_d$, there exists a $\xi > 0$ depending only on $d$  such that every $B \in \mathcal{S}_d$ is holomorphic in
$S = \{x+ iy\:| -\xi < y < \xi\}$. In particular, $B'(x)$ is a periodic function with period $1$ and all the zeros of $B'$ are contained in the real line.
Since $B'$ is periodic and  has at most $d-1$ distinct zeros in each interval $[x, x+1)$,  there is a $0< \kappa < 1$ depending only on $d$  such that for each $B \in \mathcal{S}_d$, we can find a $t \in [0, 1)$ such that all zeros of $B'$ in $[t-\kappa, t+1+\kappa]$ belong to $(t, t +1)$.  By symmetry the order of $B'$ at each zero is even and is not less than two.   Let $c_1, \cdots, c_{d-1}$ denote all the zeros of $B'$ in $(t, t+1)$, counting by multiplicities,
 such that the order of $B'$ at each $c_i$ is exactly two.     Let $U = \{x + iy\:|\: t-\kappa < x < t + 1+\kappa, -\xi < y< \xi\}$.  For $z \in U$,  let
 $$
 B'(z) = g(z) \cdot \prod_{1 \le i \le d-1} (z - c_i)^2.
 $$ Then $g$ is a holomorphic function defined in $U$.  Let $V  = \{x + iy\:|\: t-\kappa/2 < x < t + 1+\kappa/2, -\xi/2 < y< \xi/2\}$. Since $\mathcal{S}_d$ is compact, it follows that there is a $0< \eta <1$ depending only on $d$ such that for all $z \in V$, we have
\begin{itemize}
\item[i.] $|g(z)| \ge  \eta$,
\item[ii.] $|g'(z)| < 1/\eta$,
\item[iii.] $|g''(z)|< 1/\eta$.
\end{itemize}

Let  $x \in (t-\kappa/2, t + 1+ \kappa/2)$ and $x \ne c_i$, $1 \le i \le d-1$.  Let $$P(x) = \prod_{1 \le i \le d-1}  (x - c_i).$$  Then
$
B'(x) =  P^2(x) \cdot g(x)
$.
By direct calculations we have
$$
B''(x) = 2 P(x) P'(x) g(x) + P^2(x) g'(x)$$$$ = 2 P^2(x) \Sigma(x)  g(x) +
P^2(x) g'(x) = P^2(x)(2 \Sigma(x) g(x) + g'(x))
$$
where $$\Sigma(x) = \frac{P'(x)}{P(x)} = \sum_{1 \le i \le d-1} \frac{1}{x - c_i},$$ and
$$
B'''(x) = 2P^2(x) \Sigma(x)(2 \Sigma(x) g(x) + g'(x)) + P^2(x) (-2 \sigma(x) g(x) + 2 \Sigma(x)
g'(x) + g''(x))
$$
where $$\sigma(x) = -\Sigma ' (x) = \sum_{1 \le i \le d-1} \frac{1}{|x - c_i|^2}.$$ Then
$$
\frac{B'''(x)}{B'(x)} = 4 \Sigma^2(x) + 4\Sigma(x) \frac{g'(x)}{g(x)} - 2 \sigma(x)  +
\frac{g''(x)}{g(x)}
$$ and
$$
\frac{B''(x)}{B'(x)} = 2 \Sigma(x) + \frac{g'(x)}{g(x)}.
$$

From
$$
S(B) (x) = \frac{B'''(x)}{B'(x)} - \frac{3}{2}\bigg{(}\frac{B''(x)}{B'(x)}\bigg{)}^2,$$ we finally have
\begin{equation}\label{swz} S(B)(x)= -2
\Sigma^2(x) - 2 \Sigma(x) \frac{g'(x)}{g(x)} - 2 \sigma(x) +\frac{g''(x)}{g(x)} -
\frac{3}{2} \bigg{(}\frac{g'(x)}{g(x)}\bigg{)}^2.
\end{equation}

Let $\Omega_B = \{c_i, 1 \le i \le d-1\}$. Recall that $x \in (t-\kappa/2, t + 1+ \kappa/2)$ and $x \notin \Omega_B$. Let    $$\delta = \min_{1 \le i \le d-1}|x - c_i| = {\rm dist}(x, \Omega_B).$$  We clearly have  \begin{itemize} \item[1.]
$-2 \Sigma^2(x) < 0$,  \item[2.] $- 2 \sigma(x)
\le  -\frac{2}{\delta^2}$.\end{itemize}  From  (i), (ii) and (iii) and the fact that $0< \eta < 1$ depends only on $d$, it follows that there is an $0< L < \infty$ depending only on $d$ such that
 \begin{itemize}  \item[3.] $|-2 \Sigma(x)\frac{g'(x)}{g(x)}| \le  \frac{L}{\delta}$,   \item[4.] $|\frac{g''(x)}{g(x)} -
\frac{3}{2} \big{(}\frac{g'(x)}{g(x)}\big{)}^2| < L$.  \end{itemize}

From (\ref{swz}) and the above  properties (1-4) it follows that there is an $0< \epsilon <1$ depending only on $d$  such that whenever $\delta = {\rm dist}(x, \Omega_B) < \epsilon$, one has
\begin{equation}\label{pde}
S(B)(x) < -\frac{1}{{\rm dist}^2(x, \Omega_B)}.
\end{equation}

Note that until now we have been assuming $B \in \mathcal{S}_d$ only.  Now let us assume that $\alpha \in \Theta_C^b$ and  $B \in \mathcal{S}_d^\alpha$. As before, let $p_n/q_n$ denote the $n$-th convergent of $\alpha$.
Let $L = J_3 \cup \cdots \cup J_{m-1}$ and for $i \ge 0$ let $L_i = B^i(L)$.    Let $x \in L $ be an arbitrary point. Let us consider the sum
\begin{equation}\label{sum-usg}
S(B^{q_n})(x) \cdot  |L|^2  =   \sum_{j=0}^{q_n -1} SB(B^j (x)) (DB^j(x))^2 \cdot |L|^2 .
\end{equation}
By Lemma~\ref{uniform-dr}, it follows that
\begin{equation}\label{usg-1} K_1^{-1} \cdot |L_j|^2  < (DB^j(x))^2 \cdot |L|^2  < K_1 \cdot |L_j|^2\end{equation} where $K_1 > 1$ is some constant depending only on $d$.

Let $U_0 = \{x \in \Bbb T \:|\: {\rm dist}(x, \Omega_B) < \epsilon\}$ and $V_0 = \{x \in \Bbb T \:|\: {\rm dist}(x, \Omega_B) > \epsilon/2\}$. By Remark~\ref{lr-1} there is an $N >0$ depending only on $d$ and $\epsilon$  such that for any $n \ge  N$, any $B \in \mathcal{S}_d^{\alpha}$, the length of any interval $[x, B^{q_{n-1}}x]$ is less than $\epsilon/4$. Since $\epsilon > 0$ depends only on $d$, such $N$ eventually depends only on $d$.
By Proposition~\ref{crp} any component of $\Bbb T \setminus \mathcal{Q}_n$ is contained in some interval with the form $[x, B^{q_{n-1}}(x)]$ or $[B^{q_n}(x), x] \cup [x, B^{q_{n-1}}(x)]$ for some $x \in \Bbb T$.    Since $L$ is contained in some component of $\Bbb T \setminus \mathcal{Q}_n$, it follows that $L$ and thus all $L_j$  are  contained in intervlas with the same form.  This implies that
 for $n > N$,  each $L_j$ has length less than $\epsilon/2$.  Now from the definition of $U_0$ and $V_0$, it follows  that each $L_j$, $0 \le j < q_n$,   is either contained in $U_0$ or contained in $V_0$.    We split the sum in (\ref{sum-usg}) into $\Sigma_1$ and
$\Sigma_2$:  $\Sigma_1$  is taken over  all the terms such that $L_j$ is
contained in  $U_0$,  and $\Sigma_2$ is taken over all the other terms.

By (\ref{pde}) all the terms in $\Sigma_1$ are negative. Recall that in Lemma~\ref{U-S-G} the two boundary points of $J$ belong to $\bigcup_{i=0}^{d'} \mathcal{Q}_n^i$.    By the definition of $J$, $J_1 = [z, B^{-q_n}(z)]$ for some $z \in \Bbb T$.  Let $J_{0}= [B^{q_n}(z), z]$.
 Then there exists a $0 \le j < q_n$ and a critical point $c_i$ of $B$
 such that either $B^j(z) = c_i$ or $B^{j}(J_{0})$ contains $c_i$.      In either of the two cases, by Theorem~\ref{real bounds}, both $B^j(J_{0})$ and $B^j(J_1)$  are $K(d)$-commensurable with $[c_i, B^{q_n}(c_i)]$, which is then  $K(d)$-commensurable with both $[c_i, B^{q_{n-1}}(c_i)]$ and $[c_i, B^{q_{n-1}-q_n}(c_i)]$.  By Proposition~\ref{crp}    $L_j$ is contained either in $[c_i, B^{q_{n-1}}(c_i)]$  or  $[c_i, B^{q_{n-1}-q_n}(c_i)]$. We thus have  $|L_j| \le K(d) \cdot |B^j(J_0)|$. On the other hand, By Theorem~\ref{real bounds},
 $B^j(J_1)$ is $K(d)$-commensurable with $B^j(J_2)$, and  $B^j(J_2)$ is $K(d)$-commensurable with $B^j(J_3)$. Here $1 < K(d) < \infty$ is some constant depending only on $d$. Since $B^j(J_3) \subset L_j$, we finally have
 $$
 |B^{j}(J_0)| \asymp |B^j(J_1)| \asymp |B^j(J_2)| \asymp|L_j|
 $$ where the implicit constants depend only on $d$.
 So for any $x \in  L$,
 $$
 {\rm dist}(B^j(x), \Omega_B) \le {\rm dist}(B^j(x), c_i) \le |B^{j}(J_0)| + |B^j(J_1)| + |B^j(J_2)| + |L_j|  < K_2 \cdot    |L_j|$$
 where  $K_2 > 1$ is some constant depending only on $d$.  By Theorem~\ref{real bounds}  and by taking $N$ larger if necessary,   we may assume that $K_2 \cdot    |L_j| < \epsilon$.      Thus by (\ref{pde}) we have for such $j$,
 $$
 S(B) (B^j(x)) < -\frac{1}{K_2^2 \cdot |L_j|^2}.$$  Since all the terms in $\Sigma_1$ are negative,  it follows from (\ref{sum-usg}) and (\ref{usg-1}) that
\begin{equation}\label{sum-1}
\Sigma_1 < -\frac{1}{K_1  K_2^2}
\end{equation} provided that $n \ge N$.

On the other hand, for all $x \in V_0$, by the compactness property of $\mathcal{S}_d$, $|SB(x)| < M$ for some $0< M < \infty$ depending only on $d$ and $\epsilon$. Since $\epsilon$ depends on $d$, such $M$ depends eventually on $d$.  Since  all $L_j$, $ 0\le j < q_n$,  are disjoint, we have $$\sum_{j=0}^{q_n -1} |L_j| \le  2\pi.$$
By taking $N$ larger if necessary, we can make sure that $|L_j| < (2\pi  M K_1^2   K_2^2 )^{-1}$ provided that $n \ge N$. Then for all $n \ge N$, from (\ref{usg-1}) we have
\begin{equation}\label{sum-2}
\Sigma_2 < M   \sum_{j=0}^{q_n-1}  K_1  |L_j|^2     <  M  K_1  (2\pi  M K_1^2   K_2^2 )^{-1} \cdot \sum_{j=0}^{q_n-1}  |L_j| <   \frac{1}{K_1 K_2^2}.
\end{equation}
Lemma~\ref{neg-s} now follows from (\ref{sum-1}) and (\ref{sum-2}).
\end{proof}

Now apply Lemma~\ref{Y-L} to the diffeomorphism $B^{q_n}: J_3 \cup \cdots \cup J_{m-1} \to J_2 \cup \cdots \cup J_{m-2}$.
By  Lemmas~\ref{u-b-c} and ~\ref{neg-s} it follows that the  two conditions in Lemma~\ref{Y-L} are satisfied. The second assertion of
Lemma~\ref{U-S-G} now follows from Lemma~\ref{Y-L}. This completes the proof of Lemma~\ref{U-S-G}.

\subsection{Constructing qc homeomorphisms between polygons}
In this subsection we will introduce the key construction in the proof of Lemma~\ref{uniform-p}.  The basic idea comes from  \cite{PZ}, but due to the presence of more than one critical point in $\Bbb T$,  we need to deal with some new difficulty in the construction, cf. Lemma~\ref{g-f-s}.

Let $\alpha \in \Theta_C$ and
 $B \in \mathcal{S}_{d}^{\alpha}$.  As in  \cite{PZ}  we will  give two ways to divide $\Delta$ into countably many polygons, one for the circle homeomorphism $B|\Bbb T: \Bbb T \to \Bbb T$  and the other for the rigid rotation $R_\alpha: \Bbb T \to \Bbb T$.  For each pair of corresponding polygons, we construct a qc homeomorphism between them so that the restriction of the homeomorphism to each edge of the polygon is linear.  We then glue all these qc homeomorphisms along the edges of the polygons and get a desired David homeomorphism $H: \Delta \to \Delta$.
  Compared with the situation in \cite{PZ}, a slight difference arises here. For the Douady-Ghys' Blaschke model $G$ used in \cite{PZ},
the bottom side of each polygon is a polyline satisfying the
saddle-node condition, while in our case, the bottom side of each polygon  consists of several pieces of polylines each of which satisfies  the saddle-node condition.   The  idea here is to find an appropriate way  to divide each polygon into finitely many subpolygons  so that the bottom side of   each subpolygon is a polyline satisfying the saddle-node condition. The following lemma, which is essentially  Lemma 6.5 in \cite{PZ}, is the fundamental block in this construction.
\begin{lem}[Yoccoz, cf. Theorem 6.5,  \cite{PZ}] \label{fundamental square}
Let $P$ and $Q$ be two unit squares. Let
$$
X = \{x_1, x_2, \cdots, x_m\} \hbox{  and  } Y = \{y_1, y_2, \cdots,
y_m\}
$$ be two partitions of the two bottom sides of $P$ and $Q$, respectively. Then $P$ and $Q$ become into two polygons by
adding the points in $X$ and $Y$ to the set of vertices of $P$ and $Q$ respectively.  Suppose
the partition $X$ satisfies the $C_0$-bounded saddle-node condition for some
$C_0 > 1$, that is,
$$
\frac{1}{C_0} \frac{|x_1-x_{m}|}{\min\{i, m-i\}^2} \le |x_i -
x_{i+1}| \le {C_0} \frac{|x_1 - x_m|}{\min\{i, m-i\}^2},\:\: 1 \le i
\le m-1,
$$ and $Y$  satisfies the $C_1$-bounded linear condition for some $C_1 > 1$, that is,
$$
\frac{1}{C_1 } \cdot \frac{|y_1 - y_m|}{m} \le |y_i - y_{i+1}| \le
{C_1} \cdot \frac{|y_1 - y_m|}{m}, \:\: 1 \le i \le m-1.
$$ Then there is a $K$-qc homeomorphism $F: P \to Q$ such that when restricted to the corresponding edges,
$F$ is  linear  and
$$
K < \lambda \cdot (1 + (\log{m})^2)
$$
where $\lambda > 1$ is a constant depending only on $C_0$ and $C_1$.
\end{lem}

The following lemma, which is a generalized version of  Lemma~\ref{fundamental square}, is the key of the proof of Lemma~\ref{uniform-p}.
\begin{lem}\label{g-f-s}
Let $P$ and $Q$ be two unit squares. Let
$$
X = \{x_1, x_2, \cdots, x_m\} \hbox{  and  } Y = \{y_1, y_2, \cdots,
y_m\}
$$ be two partitions of the two bottom sides of $P$ and $Q$, respectively. Then $P$ and $Q$ become into two polygons by
adding the points in $X$ and $Y$ to the set of vertices of $P$ and $Q$ respectively. Let $l \ge 1$.

Suppose the
partition $X$ consists of $l$ pieces all of which satisfy the $C_0$-bounded saddle-node condition and are $C_0$-commensurable with each other, with $C_0 > 1$ being some constant,  that is,
\begin{itemize}
\item[1.] there exist
$$
1=m_0 < m_1 < m_2 < \cdots < m_{l-1} < m_l = m
$$ such that $m_{i+1}-m_i \ge 2$ for all $0 \le i \le l-1$, and
$$
|x_1 - x_{m}|/C_0 \le |x_{m_{i-1}} - x_{m_i}|\le  |x_1 - x_{m}| ,\:\:
1 \le i \le l,
$$   and  \item[2.] for $0 \le i \le m-1$  and  $m_i \le  j < m_{i+1}$,
$$
\frac{1}{C_0} \cdot \frac{|x_{m_i}-x_{m_{i+1}}|}{\min\{j-m_i+1,
m_{i+1}-j\}^2} \le |x_j - x_{j+1}| \le C_0 \cdot \frac{|x_{m_i} -
x_{m_{i+1}}|}{\min\{j-m_i +1, m_{i+1}-j\}^2}.
$$ \end{itemize}

Suppose in addition that  $Y$  satisfies the  $C_1$-bounded linear condition for some $C_1 > 1$, that is,
$$
\frac{1}{C_1} \cdot  \frac{|y_1 - y_m|}{m} \le |y_j - y_{j+1}| \le
C_1 \cdot \frac{|y_1 - y_m|}{m}, \:\: 1 \le j \le m-1.
$$ Then there is a $K$-qc homeomorphism $F: P \to Q$ such that when restricted to the corresponding edges,
$F$ is a linear map and
$$
K < \lambda \cdot (1 + (\log{m})^2)
$$
where $\lambda > 1$ is a constant depending only on $C_0$, $C_1$ and $l$.
\end{lem}

\begin{proof} Let $A$ and $B$ denote the two vertices of the bottom side of $P$.
Let $A'$ and $B'$ denote the two corresponding vertices of the bottom side of $Q$.

If $l = 1$, then the lemma is implied by Lemma~\ref{fundamental
square}.

Suppose $l \ge 2$ and the lemma holds for $l-1$. Let us prove  the lemma  for $l$.
   Let us assume that $m_1 - m_0 \ge m_l - m_{l-1} (\ge 2)$. The case that  $m_1 - m_0 < m_l - m_{l-1}$ can be treated in a similar  way.
   Let $$n  =  \bigg{[}\frac{m_l - m_0}{m_l - m_{l-1}}\bigg{]}+16$$ where $[\:\cdot\:]$ denote the integer part of a number. Then $n \ge 18$. Since $m_l-m_{l-1} \ge 2$, we get $n < m/2 + 16$.

Claim:
 There exist  $K_1, C_2  > 1$ depending only on
$C_0$, and  $K_2, C_3 > 1$ depending only on $C_1$,  and  two group of points  $x_1', \cdots, x_n'$ and $y_1',  \cdots,  y_n'$,
such that  (see Figure 1 for an illustration) \begin{itemize} \item[1.] $x_1' = x_1 = A$,
  $x'_{n-3} = x_{m_{l-1}}$,   $x'_n = x_{m_l} = B$,
  \item[2.] $x'_j$ lies in the interior of $P$ for all $2 \le j \le n-4$  and $j =n-2, n-1$,
  \item[3.] $y_1' = y_1 = A'$,
  $y'_{n-3} = y_{m_{l-1}}$,  $y'_n = y_{m_l} = B'$, \item[4.] $y'_j$ lies in the interior of $Q$ for all $2 \le j \le n-4$ and $j =n-2,  n-1$, \end{itemize} and moreover,   let $L$ and $L'$  be respectively the two polylines connecting $x_1', \cdots, x_n'$, and $y_1', \cdots, y_n'$ in order, then
  \begin{itemize}

  \item[5.]  Let $L_1 $ be the part of $L$ connecting  $A(=x_1')$ and $x_{m_{l-1}}(=x_{n-3}')$ and $L_2$ be the remaining part of $L$.  There exist  a polyline  $S$ between $L_1$ and the straight segment $[A,   x_{m_{l-1}}]$ which consists of three straight segments and connects  $A$ and $x_{m_{l-1}}$ such that the following properties hold. $P$ is divided by $L$ and $S$ into four polygons $P_{1}$, $P_2$, $P_3$, $P_4$, where $P_1$ is the top one, $P_2$ is the one at the right-lower corner, $P_3$ is the one bounded by $S$  and $L_1$, $P_4$ is the one bounded by $S$ and $[A, x_{m_{l-1}}]$.  Moreover,  for each $i = 1, 2, 3$, $4$, there is a $K_1$-qc homeomorphism $\phi_i$  mapping $P_i$ to a polygon which is the  standard unit square with the bottom side consisting of either a single polyline or $(l-1)$ polylines   satisfying   the $C_2$-bounded saddle-node condition.   More precisely,  for $P_1$, $L$ is mapped to the bottom side which satisfies the $C_2$-bounded saddle-node condition;  for $P_2$, $[x_{m_{l-1}},   B]$ is mapped to the bottom side which satisfies the $C_2$-bounded saddle-node condition;  for $P_3$,  $L_1$ is mapped to the bottom side which satisfies $C_2$-bounded saddle-node condition;  for $P_4$, $[A, x_{m_{l-1}}]$ is mapped to the bottom side which consists of $(l-1)$ polylines all of which  satisfy $C_2$-bounded saddle-node condition and are $C_2$-commensurable with each other.  For each $P_i$, $1 \le i \le 4$, the map $\phi_i$ is linear on each edge of $P_i$ and maps each edge of $P_i$ to the corresponding edge of the polygon.

   \item[6.]  Let $L'_1 $ be the part of $L'$  connecting $A'(=y_1')$ and $y_{m_{l-1}}(=y_{n-3}')$  and $L_2'$ be the remaining part of $L'$.   There exists  a  polyline $S'$   between $L'_1$ and the straight segment $[A' ,  y_{m_{l-1}}]$ which consists of three straight segments and   connects $A'$ and $y_{m_{l-1}}$ such that  the following properties hold. $Q$ is divided by $S'$ and $L'$  into four polygons $Q_{1}$, $Q_2$, $Q_3$, $Q_4$, where $Q_1$ is the top one, $Q_2$ is the one at the right-lower corner, $Q_3$ is the one bounded by $S'$ and  $L_1'$, $Q_4$ is the one bounded by $S'$ and $[A', y_{m_{l-1}}]$.  Moreover, for each $i$, there is a $K_2$-qc homeomorphism $\psi_i$  mapping $Q_i$  to a polygon which is the standard square with the bottom  side   satisfying $C_3$-bounded linear condition.  More precisely,  for $Q_1$, $L'$ is mapped to the bottom side satisfying $C_3$-bounded linear condition, for $Q_2$, $[y_{m_{l-1}},  B']$ is mapped to the bottom side satisfying $C_3$-bounded linear condition, for $Q_3$,  $L_1'$ is mapped to the bottom side satisfying $C_3$-bounded linear condition,  for $Q_4$, $[A', y_{m_{l-1}}]$ is mapped to the bottom side which  satisfies  $C_3$-bounded linear condition.  For each $Q_i$, $1 \le i \le 4$, the map $\psi_i$  is linear on each edge of $Q_i$ and maps each edge of $Q_i$ to the corresponding edge of the polygon.
  \end{itemize}

Let us first prove Lemma~\ref{g-f-s} by assuming the Claim. It suffices to prove the lemma in the case that $m$ is large.  Since $n < m/2 + 16$, we may assume that $m > 32$ such that $n < m$.   Our proof is by induction on $l$. When $l=1$, the lemma follows by  Lemma~\ref{fundamental
square}. We assume that the lemma holds when the number of the polylines in the bottom side of the unit square is not more than $l$.  Then from (5) and (6) in the Claim and by the induction assumption,    there exists a $C_4 > 1$ depending only on $C_0, C_1, C_2, C_3, l$ and thus only on $C_0$, $C_1$ and $l$ (since $C_2$ and $C_3$ depend respectively on $C_0$ and $C_1$) such that for each $i =1, 2, 3, 4$, there is a $K_0$-qc homeomorphism $\sigma_i: \phi_i(P_i) \to \psi_i(Q_i)$ which satisfies  the following properties. \begin{itemize}\item[1.]   $K_0 < C_4 \cdot (1 + (\log{m})^2)$, \item[2.]  each edge of $\phi_i(P_i)$ is linearly mapped to the corresponding edge of $\psi_i(Q_i)$. \end{itemize}   Note that $P_i$ is mapped to $Q_i$ by $\psi_i^{-1} \circ \sigma_i \circ \phi_i$ such that each edge of $P_i$ is linearly mapped to the corresponding edge of $Q_i$.  Since   $\phi_i$ is $K_1$-qc and $\psi_i$ is $K_2$-qc,   thus $\psi_i^{-1} \circ \sigma_i \circ \phi_i$ is $K_1 \cdot K_2 \cdot K_0$-qc.   We can now glue the maps $\psi_i^{-1} \circ \sigma_i \circ \phi_i$ along the edges of $P_i$ and get a $K$-qc map $F: P \to Q$ with   $K= K_1K_2K_0 < K_1 K_2  C_4 \cdot (1 + (\log{m})^2)$. Since $K_1$, $K_2$ and $C_4$ depend eventually on $C_0$, $C_1$ and $l$,   the lemma thus follows  by assuming the Claim.

Now let us prove the Claim.  Let us first describe  how to choose the points $x_i'$, $i = 1, \cdots, n$. Let $a  = |[A, x_{m_{l-1}}]|$ and $b = |[x_{m_{l-1}}, B]|$. Let $x_1' = x_1 = A$,  $x_{n-3}' = x_{m_{l-1}}$ and $x_n' = B$.  Let $x_2'$ be the point in the interior of $P$ such that  $|[x_1', x_2']|=a/2$, and the angle formed by $[x_1', x_2']$ and $[A, B]$ is $\pi/3$.  Let $x_{n-4}'$ be the point in the interior of $P$ such that  $|[x_{n-4}', x_{n-3}']||=a/2$ and the angle formed by $[x_{n-4}', x_{n-3}']|$ and $[x_{m_{l-1}}, A]$ is equal to $\pi/3$. Then $|[x_2', x_{n-4}']|=a/2$. There is an obvious way to insert points $x_3', \cdots, x_{n-5}'$ in the interior of $[x_2', x_{n-4}']$ so that the horizontal polyline
$$
[x_2', \cdots, x_{n-4}']
$$
satisfies $\Lambda$-bounded saddle-node condition with $\Lambda > 1$ being some universal constant.

Now let $x_{n-2}'$ be the point in the interior of $P$ such that  $|[x_{n-3}', x_{n-2}']|=b/2$ and the angle formed by $[x_{n-3}', x_{n-2}']$ and $[x_{n-3}', B]$ is $\pi/3$. Let $x_{n-1}'$ be the point in the interior of $P$ such that $|[x_n', x_{n-1}']|=b/2$ and the angle formed by $[x_{n}', x_{n-1}']$ and $[B, A]$ is equal to $\pi/3$.  Then the length of the horizontal straight segment $[x_{n-2}', x_{n-1}']$ is $b/2$.

\begin{figure}
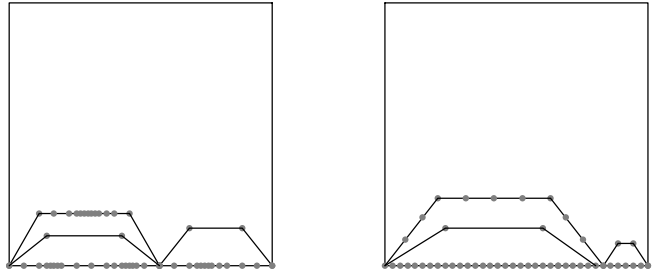

\bigskip
\begin{center}
\centertexdraw { \drawdim cm \linewd 0.02 \move(-2 1)

\move(-0.5  0) \lvec(-4  0) \lvec(-4 -3.5) \lvec(-0.5 -3.5)\lvec(-0.5 0)

\move(-2 -3.5) \fcir f:0.1 r:0.04
\lvec(-1.6 -3)
\fcir f:0.5 r:0.04
\lvec(-0.9 -3)
\fcir f:0.5 r:0.04
\lvec(-0.5 -3.5)
 \fcir f:0.5 r:0.04
\move(-2 -3.5)
\fcir f:0.5 r:0.04
\lvec(-2.4 -2.8)
\fcir f:0.5 r:0.04
\lvec(-3.6 -2.8)
\fcir f:0.5 r:0.04
\lvec(-4 -3.5)
\fcir f:0.5 r:0.04

\move(-3.4 -2.8)
\fcir f:0.5 r:0.04
\move(-3.2 -2.8)
\fcir f:0.5 r:0.04
\move(-3.1 -2.8)
\fcir f:0.5 r:0.04
\move(-3.05 -2.8)
\fcir f:0.5 r:0.04
\move(-3 -2.8)
\fcir f:0.5 r:0.04
\move(-2.95 -2.8)
\fcir f:0.5 r:0.04
\move(-2.9 -2.8)
\fcir f:0.5 r:0.04
\move(-2.85 -2.8)
\fcir f:0.5 r:0.04
\move(-2.8 -2.8)
\fcir f:0.5 r:0.04   \move(-2.7 -2.8)
\fcir f:0.5 r:0.04  \move(-2.6 -2.8)
\fcir f:0.5 r:0.04

\move(-3.8 -3.5)
\fcir f:0.5 r:0.04
\move(-3.6 -3.5)
\fcir f:0.5 r:0.04
\move(-3.5 -3.5)
\fcir f:0.5 r:0.04
\move(-3.45 -3.5)
\fcir f:0.5 r:0.04
\move(-3.4 -3.5)
\fcir f:0.5 r:0.04  \move(-3.35 -3.5)
\fcir f:0.5 r:0.04 \move(-3.3 -3.5)
\fcir f:0.5 r:0.04 \move(-3.1 -3.5)
\fcir f:0.5 r:0.04 \move(-2.9 -3.5)
\fcir f:0.5 r:0.04

\move(-2.7 -3.5)
\fcir f:0.5 r:0.04  \move(-2.6 -3.5)
\fcir f:0.5 r:0.04 \move(-2.5 -3.5)
\fcir f:0.5 r:0.04  \move(-2.45 -3.5)
\fcir f:0.5 r:0.04  \move(-2.4 -3.5)
\fcir f:0.5 r:0.04  \move(-2.35 -3.5)
\fcir f:0.5 r:0.04  \move(-2.3 -3.5)
\fcir f:0.5 r:0.04  \move(-2.2 -3.5)
\fcir f:0.5 r:0.04  \move(-2 -3.5)
\fcir f:0.5 r:0.04

\move(-1.8 -3.5)
\fcir f:0.5 r:0.04  \move(-1.6 -3.5)
\fcir f:0.5 r:0.04  \move(-1.5 -3.5)
\fcir f:0.5 r:0.04  \move(-1.45 -3.5)
\fcir f:0.5 r:0.04  \move(-1.4 -3.5)
\fcir f:0.5 r:0.04  \move(-1.35 -3.5)
\fcir f:0.5 r:0.04  \move(-1.3 -3.5)
\fcir f:0.5 r:0.04  \move(-1.2 -3.5)
\fcir f:0.5 r:0.04  \move(-1.1 -3.5)
\fcir f:0.5 r:0.04  \move(-0.9 -3.5)
\fcir f:0.5 r:0.04  \move(-0.7 -3.5)
\fcir f:0.5 r:0.04  \move(-0.5 -3.5)
\fcir f:0.5 r:0.04

\move(-4  -3.5)
\fcir f:0.5 r:0.04
\lvec(-3.5 -3.1)
\fcir f:0.5 r:0.04
\lvec(-2.5 -3.1)
\fcir f:0.5 r:0.04
\lvec(-2 -3.5)
\fcir f:0.5 r:0.04

\move(1  0) \lvec(4.5  0) \lvec(4.5 -3.5) \lvec(1 -3.5)\lvec(1 0)

\move(4.5 -3.5)\fcir f:0.5 r:0.04
 \lvec(4.3 -3.2)\fcir f:0.5 r:0.04
 \lvec(4.1 -3.2)\fcir f:0.5 r:0.04
 \lvec(3.9 -3.5) \fcir f:0.5 r:0.04
 \lvec(3.2 -2.6) \fcir f:0.5 r:0.04
 \lvec(1.7 -2.6)\fcir f:0.5 r:0.04
 \lvec(1 -3.5) \fcir f:0.5 r:0.04
\move(1.1 -3.5)\fcir f:0.5 r:0.04
\move(1.2 -3.5)\fcir f:0.5 r:0.04
\move(1.3 -3.5)\fcir f:0.5 r:0.04
\move(1.4 -3.5)\fcir f:0.5 r:0.04
\move(1.5 -3.5)\fcir f:0.5 r:0.04
\move(1.6 -3.5)\fcir f:0.5 r:0.04
\move(1.7 -3.5)\fcir f:0.5 r:0.04
\move(1.8 -3.5)\fcir f:0.5 r:0.04
\move(1.9 -3.5)\fcir f:0.5 r:0.04
\move(2 -3.5)\fcir f:0.5 r:0.04
\move(2.1 -3.5)\fcir f:0.5 r:0.04
\move(2.2 -3.5)\fcir f:0.5 r:0.04
\move(2.3 -3.5)\fcir f:0.5 r:0.04

\move(2.4 -3.5)\fcir f:0.5 r:0.04
\move(2.5 -3.5)\fcir f:0.5 r:0.04
\move(2.6 -3.5)\fcir f:0.5 r:0.04
\move(2.7 -3.5)\fcir f:0.5 r:0.04
\move(2.8 -3.5)\fcir f:0.5 r:0.04
\move(2.9 -3.5)\fcir f:0.5 r:0.04
\move(3 -3.5)\fcir f:0.5 r:0.04
\move(3.1 -3.5)\fcir f:0.5 r:0.04

\move(3.2 -3.5)\fcir f:0.5 r:0.04
\move(3.3 -3.5)\fcir f:0.5 r:0.04
\move(3.3 -3.5)\fcir f:0.5 r:0.04
\move(3.4 -3.5)\fcir f:0.5 r:0.04
\move(3.5 -3.5)\fcir f:0.5 r:0.04
\move(3.6 -3.5)\fcir f:0.5 r:0.04
\move(3.7 -3.5)\fcir f:0.5 r:0.04
\move(3.8 -3.5)\fcir f:0.5 r:0.04
\move(3.9 -3.5)\fcir f:0.5 r:0.04
\move(4 -3.5)\fcir f:0.5 r:0.04

\move(4.1 -3.5)\fcir f:0.5 r:0.04
\move(4.2 -3.5)\fcir f:0.5 r:0.04
\move(4.3 -3.5)\fcir f:0.5 r:0.04
\move(4.4 -3.5)\fcir f:0.5 r:0.04
\move(4.5 -3.5)\fcir f:0.5 r:0.04

\move(1.5 -2.85)\fcir f:0.5 r:0.04
\move(1.27 -3.15)\fcir f:0.5 r:0.04

\move(3.4 -2.85)\fcir f:0.5 r:0.04
\move(3.63 -3.15)\fcir f:0.5 r:0.04

\move(2.45 -2.6)\fcir f:0.5 r:0.04
\move(2.075 -2.6)\fcir f:0.5 r:0.04
\move(2.825 -2.6)\fcir f:0.5 r:0.04

\move(1 -3.5) \lvec(1.8 -3)\fcir f:0.5 r:0.04  \lvec(3.1 -3) \fcir f:0.5 r:0.04 \lvec(3.8 -3.5)
 }
\end{center}
\caption{ Divide $P$ and $Q$ into four polygons with one side satisfying saddle-node and linear geometry respectively }
\end{figure}
Now let us describe how to construct the polyline $S$. Let $s_1 = A$ and $s_4 = x_{m_{l-1}}$.
Let $s_2$ be the point in the interior of $P$ such that  $|[s_1, s_2]|=a/3$ and the angle formed by $[s_1, s_2]$ and $[A, B]$ is $\pi/4$. Let $s_3$ be the point in the interior of $P$ such that  $|[s_4, s_3]|=a/3$ and the angle formed by $[s_4, s_3]$ and $[x_{m_{l-1}}, A]$ is $\pi/4$. Let $S$ be the polyline which connects $s_1, s_2, s_3$ and $s_4$ in order.

Let $L_1$ be the part of $L$ connecting $x_1'$ and $x_{m_{l-1}}$ and $L_2$ be the remaining part of $L$. Then
from the construction we see that $L$ and $S$ divide $P$ into four polygons $P_i$, $1 \le i \le 4$:
 $P_1$ is the top one; $P_2$ is the one bounded by $L_2$ and $[x_{m_{l-1}}, B]$; $P_3$  is the one bounded by $L_1$ and $S$; and $P_4$ is the one bounded
 by $S$ and $[A, x_{m_{l-1}}]$. Each of these polygons has four sides, the three of which are straight segments and the last one is a polyline.  From the construction it is also clear that for $P_1$, $P_2$ and $P_3$,   the polyline side satisfies  the $C_2'$-bounded saddle-node condition, and for $P_4$, the polyline side  consists of $l-1$ polylines  all of which satisfy  the $C_2'$-bounded saddle-node condition and are $C_2'$-commensurable with the whole polyline side,  where   $C_2' > 1$  is some constant depending only on $C_0$.     Again from the construction,  the geometry of each $P_i$ is bounded and relies only on $C_0$.  This implies the existence of the constants  $K_1>0$ and $C_2> 1$ depending only on $C_0 $ and the desired $K_1$-qc homeomorphisms $\phi_i$, $i=1, 2, 3, 4$.

 Let us now describe  how to choose the points $y_i'$, $i = 1, \cdots, n$.
  Let $a'  = |[A', y_{m_{l-1}}]|$ and $b' = |[y_{m_{l-1}}, B']|$.
   Let $y_1' = y_1 = A'$,  $y_{n-3}' = y_{m_{l-1}}$ and $y_n' = B'$.  Let $n_1 = [n/3]$ and $n_2 = [2n/3]$.  Let $y_{n_1}'$ be the point in the interior of $Q$ such that  $|[y_1', y_{n_1}']|=a'/2$, and the angle formed by $[y_1', y_{n_1}']$ and $[A', B']$ is $\pi/3$.  Let $y_{n_2}'$ be the point in the interior of $Q$ such that the length of $[y_{n_2}', y_{n-3}']|$ is equal to $a'/2$ and the angle formed by $[y_{n_2}', y_{n-3}']|$ and $[y_{m_{l-1}}, A']$ is equal to $\pi/3$. Then the straight segment $[y_{n_1}', y_{n_2}']$ has length $a' /2$. Now we insert points $y_2', \cdots, y_{n_1 -1}'$ in the interior of $[y_1', y_{n_1}']$, and  insert points  $y_{n_1 +1}', \cdots, y_{n_2 -1}'$  in the interior of $[y_{n_1}', y_{n_2}']$, and insert points $y_{n_2 +1}', \cdots, y_{n-4}'$ in the interior of $[y_{n_2}', y_{n-3}']$, so that
   $$
   |[y_i',  y_{i+1}']| = |[y_1', y_{n_1}']|/(n_1-1), \:\: 1 \le i \le n_1 -1, $$
   $$|[y_i',  y_{i+1}']| = |[y_{n_1}', y_{n_2}']|/(n_2-n_1), \:\: n_1 \le i \le n_2 -1,$$ and  $$|[y_i',  y_{i+1}']| = |[y_{n_2}', y_{n-3}']|/(n-3-n_2), \:\: n_2 \le i \le n -4.$$

Now let $y_{n-2}'$ be the point in the interior of $Q$ such that  $|[y_{n-3}', y_{n-2}']|= b'/2$ and the angle formed by $[y_{n-3}', y_{n-2}']$ and $[y_{n-3}', B']$ is $\pi/3$. Let $y_{n-1}'$ be the point in the interior of $Q$ such that  $|[y_n', y_{n-1}']|= b'/2$ and the angle formed by $[y_{n}', y_{n-1}']$ and $[B', A']$ is equal to $\pi/3$.  Then the length of the horizontal straight segment $[y_{n-2}', y_{n-1}']$ is $b'/2$.

The construction of $S'$ is very similar to that of $S$.  Let $s'_1 = A'$ and $s'_4 = y_{m_{l-1}}$.
Let $s'_2$ be the point in the interior of $Q$ such that  $|[s'_1, s'_2]|= a'/3$ and the angle formed by $[s'_1, s'_2]$ and $[A', B']$ is $\pi/4$. Let $s_3'$ be the point in the interior of $Q$ such that the length of $[s'_3, s'_4]$ is equal to $a'/3$ and the angle formed by $[s'_4, s'_3]$ and $[y_{m_{l-1}}, A']$ is $\pi/4$. Let $S'$ be the polyline which connects $s'_1, s'_2, s'_3$ and $s'_4$  in order.

Let $L'_1$ bet the part of $L'$ connecting $y_1'(=A')$ and $y_{n-3}'(=y_{m_{l-1}})$ and $L'_2$ be the remaining part of $L'$. Then
from the construction we see that $L'$ and $S'$ divide $Q$ into four polygons $Q_i$, $1 \le i \le 4$:
 $Q_1$ is the top one; $Q_2$ is the one bounded by $L_2'$ and $[y_{m_{l-1}}, B']$; $Q_3$  is the one bounded by $L'_1$ and $S'$; and $Q_4$ is the one bounded
 by $S'$ and $[A', y_{m_{l-1}}]$. Each of these polygons has four sides, the three of which are straight segments and the last one is a polyline satisfying $C_3'$-bounded linear condition with $C_3' >1$ being some constant depending only on $C_1$.    From the  construction, it follows that the geometry of each $Q_i$ is bounded relies only on $C_1$.  This implies the existence of constants $K_2> 0$ and $C_3 > 1$  depending only on $C_1$  and the desired  $K_2$-qc homeomorphisms.

 This proves the Claim and  the proof of the lemma is completed.
\end{proof}

\subsection{Proof of Lemma~\ref{uniform-p}}
Let  $\alpha \in \Theta_C^b$ and  $f \in \Pi_\alpha^d$. Let $B_f \in \mathcal{S}_d^\alpha$ be the Blaschke
product which models $f$  and  $R_\alpha$ denote the rigid rotation given by $z \mapsto e^{2 \pi i \alpha} z$.  Let $h_f: \Bbb T \to \Bbb T$ be the circle homeomorphism such that $B_f|\Bbb T = h_f^{-1} \circ R_\alpha \circ h_f$ and $h_f(1) = 1$.  The aim of this subsection is to construct a David extension $H_f: \Delta \to \Delta$ of $h_f$ which satisfies the uniform integrability condition described in Lemma~\ref{uniform-p}. The idea is the same as the one used in \cite{PZ}. Namely, we will construct two decompositions of the unit disk into polygons, one for the circle homeomorphism $B|\Bbb T: \Bbb T \to \Bbb T$, and the other for the rigid rotation $R_\alpha: \Bbb T \to \Bbb T$. We then construct a qc homeomorphism from each polygon in the first decomposition to the corresponding polygon in the second decomposition so that when restricted to each edge of the polygon, the map is linear. The $H_f$ is then obtained by gluing all these qc homeomorphisms along the edges of the polygons. To get the uniform integrability of $\mu_{H_f}$, we will replace Theorem 6.5 in \cite{PZ} by Lemma~\ref{g-f-s}. Since the idea is generally the same as in $\S6$ of \cite{PZ}, let us merely provide the outline of the proof in the following.

 First recall that $x_i \in \Bbb T$  denotes the point such that $f^i(x_i) = 1$, and
$\mathcal{Q}_n  = \{x_{i} \: \big{|}\: 0\le i <  q_{n}\}$.  Let $x_i'\in \Bbb T$ denote the point such that $R_{\alpha}^i(x_i') =
1$ and $\mathcal{Q}_n' =  \{x_{i}' \: \big{|}\: 0\le i <  q_{n}\}$.
By Remark~\ref{lr-1} there is an integer
\begin{equation}\label{constant-n} N_{0} \ge 1\end{equation}  depending only on $d$  such that for all $n \ge N_{0}$, if $x_i$ and $x_j$ are two adjacent points in $\mathcal{Q}_{n}$ and $x_i'$ and $x_j'$ are two adjacent points in $\mathcal{Q}_{n}'$, then $d(x_{i}, x_{j})< 1$ and $d(x_{i}', x_{j}')< 1$.

For each $x_{i} \in
\mathcal{Q}_{n}$, let $y_{i}$ be the point on the radial segment
$[0, x_{i}]$ such that
$$
|y_{i}-x_{i}| = d(x_{r}, x_{l})/2
$$
where $x_{r}$ and $x_{l}$ denote the two  points immediately to the
right and left of $x_{i}$ in $\mathcal{Q}_{n}$.

\begin{defi}[Yoccoz's cells]\label{b-s}{\rm
Let $x_{i}$ and $x_{j}$ be any two adjacent points in
$\mathcal{Q}_{n}$. Connect $y_{i}$ and $y_{j}$ by a straight
segment. Then the three straight segments $[x_{i}, y_{i}]$, $[y_{i},
y_{j}]$, $[x_{j}, y_{j}]$,  and the  arc segment  $[x_{i}, x_{j}]$
bound a domain in $\Delta$. We call the closure of this domain  a
$\emph{cell of level}$ $n$. The segment $[y_i, y_j]$ is called the top side of the cell. }
\end{defi}
From the last assertion of Proposition~\ref{crp}, it is not difficult to see
\begin{lem}[cf. Lemma 6.1  of \cite{Zh2} or  Lemma 6.3 of \cite{PZ}]\label{geometry-of-cells}
There exist  $C(d) > 1$ and $0< \gamma(d) < \sigma(d) < \pi$ depending only on $d$ such that for any cell $E$ with level $n \ge N_0$, the diameters of the four sides of the cell $E$ are $C(d)$-commensurable with each other, and moreover, the angles formed by the top side and its two radial sides  belong to  $[\gamma(d), \sigma(d)]$.
\end{lem}

Let $E$ be a cell of level $n$. Let $E_1, \cdots, E_m$ be all the cells of level $(n+1)$ which are contained in $E$. Then $$ E \setminus \bigcup_{i=1}^m E_i$$ is either empty, or a triangle, or a polygon. In fact, let $x_l, x_i, x_j, x_r$ be four adjacent points in $\mathcal{Q}_n$ and $E$ be the cell determined by  $x_i$ and $x_j$. In the case that the four points are still adjacent in $\mathcal{Q}_{n+1}$, $E$ is also a cell of level $n+1$ and the above set is empty. If only $x_l$, $x_i$ and $x_j$ or $x_i, x_j$ and $x_r$ are adjacent in $\mathcal{Q}_{n+1}$, then $E$ contains only one cell of level $n+1$ which have three common vertices with $E$. In this case, the above set is a triangle. Otherwise, the above set is a $k$-polygon with $k \ge 4$.  The boundary of each such polygon is the union of four sides: one side is the top side of $E$ which is still called the top side, two sides are the two radial edges of the polygon, and are called radial sides, and the remaining side is   a polyline which is the union of  the top sides of all the cells of level $(n+1)$  contained in $E$. We call the last side the bottom side of the polygon.
By  Lemmas~\ref{sub-int}, ~\ref{U-S-G}  and Lemma~\ref{geometry-of-cells}, we have  $K(d), C(d) > 1$ depending only on $d$ such that for any such polygon, there is a
$K(d)$-qc homeomorphism $\xi$ which maps the polygon homeomorphically onto the standard polygon  $P$ described in Lemma~\ref{g-f-s}, and moreover, when restricted to each edge of the polygon, $\xi$ is linear.

 Now replacing $\mathcal{Q}_{n}$ by $\mathcal{Q}_{n}'$, and  using the same construction as above, we can construct cells and polygons for $R_\alpha$.   From Proposition~\ref{crp} it follows that there exists a universal constant
$K > 1$ such that for any such $k$-polygon with $k \ge 4$, there is a $K$-qc homeomorphism $\sigma$ which maps the polygon homeomorphically onto the standard polygon  $Q$ described in Lemma~\ref{g-f-s},  and moreover, when restricted to each edge of the polygon, $\sigma$ is linear.

Let us now construct the David extension $H_f: \Delta \to \Delta$.
   Suppose $A$ is a polygon in the first decomposition of level $n \ge N_0$ and $B$ is the corresponding polygon in the second decomposition. If both $A$ and $B$ are triangles, then from Lemma~\ref{real bounds}, both of them have bounded geometry in the sense that all the three edges of each triangle are universally commensurable.  So there is a universal $K > 1$ and $K$-qc homeomorphism $\phi$ which maps $A$ homeomorphically onto $B$, and moreover, when restricted to each edge of $A$, $\phi$ is linear.   Otherwise both $A$ and $B$ are $k$-polygons with $k \ge 4$.   Let $\xi$ and $\sigma$ be the
   qc-homeomorphisms described as above. Then $\xi(P)$ and $\sigma(Q)$ are respectively the standard polygons $P$ and $Q$ satisfying the conditions in Lemma~\ref{g-f-s} such that all the involved constants depends only on $d$.    By Lemma~\ref{g-f-s} there is a constant $C(d) > 1$ depending only on $d$ and  a qc map $\tau : \xi(P) \to \sigma(Q)$ such that $\tau$ maps each edge of $\xi(P)$  linearly to the corresponding edge of $\sigma(Q)$ and the qc constant of $\tau$  is bounded by $$ C (d) \cdot (1 + (\log a_{n+1})^2) <\lambda(d, C) \cdot n$$ where $\lambda(d, C) > 0$ is some constant depending only on $d$ and $C$. Here we uses the arithmetic condition that $\log a_n \le C \sqrt n$.   Now define $$\phi = \sigma^{-1} \circ \tau \circ \xi.$$ Since the qc constants of $\xi$ and $\sigma$ are bounded by some constant depending only on $d$, by increasing $\lambda(d, C)$ if necessary, we may assume that the qc constant of
   $\phi$ is bounded by $\lambda(d, C) \cdot n$.

Let $\Phi$ be  the union of all cells of level $N_0$ in the first decomposition. Then  $P_0 = \Delta \setminus \Phi$ is a polygon which contains the origin in its interior whose boundary is the union of  the top sides of all cells of level $N_0$ for the first decomposition.   Similarly, let $\Psi$ be  the union of all cells of level $N_0$ in the second decomposition. Then  $Q_0 = \Delta \setminus \Psi$ is a polygon which contains the origin in its interior whose boundary is the union of  the top sides of all cells of level $N_0$ for the second decomposition.  Then both $P_0$ and $Q_0$ are $k$-polygons with $k$ having an upper bound depending only on $N_0$ and $C$ (thus eventually depending only on $d$ and $C$). Connect the origin to each vertex of $P_0$ and $Q_0$ by a straight segment. Then $P_0$ and $Q_0$ are decomposed into $k$ triangles.
Note that $\bigcup_{\alpha \in \Theta_C}\mathcal{S}_d^\alpha$ is compact in the following sense: there is an open neighborhood $U$ of $\Bbb T$ such that for any sequence $\{B_n\} \subset \bigcup_{\alpha \in \Theta_C}\mathcal{S}_d^\alpha$, there is a subsequence $\{B_{n'}\}$ and a $B_0 \in \bigcup_{\alpha \in \Theta_C}\mathcal{S}_d^\alpha$ such that $B_{n'} \to B_0$ uniformly in any compact set of $U$. From the compactness of $\bigcup_{\alpha \in \Theta_C}\mathcal{S}_d^\alpha$,   it follows that there exists a $\eta(d, C) > 1$ depending only on $d$ and $C$ such that
all the three edges of each triangle of $P_0$ nd $Q_0$  are $\eta(d, c)$-commensurable.  Thus  by increasing $\lambda(d, C)$ if necessary,  for each triangle  of $P_0$,  there is a $\lambda(d, C)$-qc homeomorphism $\psi$  which maps it to the corresponding triangle of $Q_0$, and moreover, when restricted to each edge of the triangle, the map is linear. Now by gluing all these maps along the edges of the triangles of $P_0$ we get a $\lambda(d, C)$-qc homeomorphism $\psi: P_0 \to Q_0$.   Now we can  define $H_f: \Delta \to \Delta$ by gluing  $\phi$ and $\psi$ along  all the edges of  $P_0$.

  Let us now prove the existence of the constants $M, \alpha > 0$ and $0< \epsilon_0 < 1$ so that Lemma~\ref{uniform-p} holds.
  Let
\begin{equation}\label{immc}
  \epsilon_0 = \frac{2}{1 + \lambda(d, C) \cdot N_0}.
\end{equation}   For any  $0< \epsilon < \epsilon_0$,
  let $n > 0$ be the least integer such that $\epsilon > \frac{2}{1 + \lambda(d, C) \cdot n}$. Thus  $n > \frac{1}{\lambda(d, C)} \cdot (\frac{2}{\epsilon} -1) >  \frac{1}{   \lambda(d, C) \epsilon }$.  Since $0< \epsilon < \epsilon_0$, from (\ref{immc}) we have $n \ge N_0 + 1$.      By the minimal property of $n$ it follows that $\epsilon \le \frac{2}{1 + \lambda(d, C) \cdot (n-1)}$.  This implies
  $$
  \{z \in \Delta \:|\: |\mu_{H_f}(z)| > 1- \epsilon\} \subseteq \{z\in \Delta\:|\:  |\mu_{H_f}(z)| > \frac{\lambda(d, C)(n-1)- 1}{\lambda(d, C)(n-1) +1}\}.
  $$
On the other hand, from the construction of $H_f$, it follows that for $n \ge N_0 + 1$,  the dilatation of $H_f$ in the complement of the union of all the cells of level $n$ is less than $\lambda(d, C) \cdot (n-1)$. So the above set is contained in the union of all the cells of level $n$.  By Theorem~\ref{real bounds},  there exist $C_1(d) > 1$ and  $0< \delta(d) < 1$ depending only on $d$ such that the area of the union of all the cells of level $n$ is less $C_1(d)\cdot  \delta^n(d)$.  Thus the area of the above set is bounded by $C_1(d) \delta^n(d)$.
Since $$C_1(d) \delta(d)^n = C_1 (d) \cdot e^{-n \ln \frac{1}{\delta(d)}} < C_1(d)\cdot  e ^{- \frac{1}{ \lambda(d, C)\epsilon} \cdot \ln \frac{1}{\delta(d)} },$$ Lemma~\ref{uniform-p} follows by taking $M = C_1(d)$ and $\alpha = \frac{1}{\lambda(d, C)} \cdot \ln \frac{1}{\delta(d)}$.

\subsection{Proof of Lemma~\ref{uniform-w}} To simplify the notations,  from now on let us just denote $H_f$ and $B_f$ by $H$ and $B$ respectively.  Let $\mu_H$ be the Beltrami differential in $\Delta$ which is given by $H$.
 Let $\mu$ denote  the Beltrami differential on the plane which is the pull back of
$\mu_H$ by the iterations of $\widehat{B}$.

For $n \ge N_0$, let  $Y_n$  be   the union of all the Yoccoz's cells of level
$n$. Then  the outer boundary component  of $Y_n$ is $\Bbb T$,
and the  inner boundary component of $Y_n$  is the union of finitely
many straight segments, and moreover,
$$
Y_{N_{0}}  \supset Y_{N_{0}+1} \supset \cdots \supset Y_n \supset
Y_{n+1} \supset \cdots .
$$
  From the proof of Lemma~\ref{uniform-p}, it follows that
 \begin{lem}\label{basic-fact-1}  Let $C > 0$ and $d \ge 2$. Then there exists a
constant $1 < \lambda(d, C)< \infty$ depending only on $d$ and $C$  such that for any $\alpha \in \Theta_C$, any  $B \in
\mathcal{S}_d^\alpha$ and any $n \ge N_0$,  the dilatation of $H$  in $\Delta\setminus
Y_{n}$ is not greater than $\lambda(d, C) \cdot n$.
\end{lem}

Define $$ X = \{z \in {\Bbb C} \setminus
\overline{\Delta}\:\big{|}\: B^{k}(z) \in \Delta \hbox{ for some
integer } k \ge 1\}.
$$
For each $z \in X$, let $k_{z} \ge 1$ be the least positive integer
such that $B^{k_{z}}(z) \in \Delta$.    Define
$$
X_{n} = \{z \in X \:\big{|}\: B^{k_{z}}(z) \in Y_{n}\}.
$$

\begin{lem}\label{main lemma} Let $d \ge 2$ be an integer and $C > 0$. Let $\alpha \in \Theta_C$. Then
there exist $C_1(d, C) > 0$, $0< \epsilon_1(d, C) < 1$,  $0< \delta_1(d, C) < 1$ and an
integer $N_{1}(d, C) \ge N_{0}$ depending only on $d$ and $C$  such that for all $B \in \mathcal{S}_d^\alpha$,
\begin{equation}\label{pz-cci}
area(X_{n+2}) \le C_1(d, C) \cdot \epsilon_1(d, C)^{n} + \delta_1(d, C) \cdot \:
area(X_{n}), \:\:\:\forall \:n >  N_1(d, C).
\end{equation}
\end{lem}
\begin{pro}\label{pr2}{\rm
Lemma~\ref{main lemma} implies Lemma~\ref{uniform-w}.}
\end{pro}
\begin{proof} The argument  is completely the same as the one used in the proof of Proposition 8.1 in \cite{Zh2}.  The reader may refer  to \cite{Zh2} for the details.

\end{proof}

The remaining part of this section is  devoted to the proof of Lemma~\ref{main lemma}.
The idea of the proof is adapted from \cite{Zh2}.   Before we present the proof, let us introduce some notations and terminologies first.  For $z \in {\Bbb
C}$ and $r
> 0$, let $B_{r}(z)$ denote the Euclidean disk with radius $r$ and
center $z$.
\begin{defi}[$K$-bounded geometry]\label{bgp} Let $K > 1$ and $(U, V)$ be a
pair of sets in $\Bbb C$ such that $V \subset U$. We say $(U, V)$
has $K$-bounded geometry if there exist $x \in V$ and $r > 0$ such
that
$$
B_{r}(x) \subset V \subset U \subset B_{Kr}(x).
$$
\end{defi}  The following lemma is a variant of Vitali's covering lemma. For a proof, see \cite{Zh2}.
\begin{lem}[ cf.  Lemma 2.1 of \cite{Zh2}]\label{covering lemma}
Let $K > 1$ and $L = 8K + 9$.  Let  $\{(U_{i}, V_{i})\}_{i \in \Lambda}$ be a finite family of pairs of measurable sets in $\Bbb C$.  Suppose all $(U_i, V_i)$
have $K$-bounded geometry, namely, for each $i \in \Lambda$, there
exist $x_i \in V_i$ and $r_i
> 0$ satisfying   \begin{equation}\label{m-qq}
B_{r_i}(x_i) \subset V_i \subset U_i \subset
B_{Kr_i}(x_i).\end{equation} Then  there is a subfamily $\sigma_{0}$ of
$\Lambda$ such that all $B_{r_{j}}(x_{j}), j \in \sigma_{0}$, are
disjoint, and moreover,
$$
\bigcup_{i \in \Lambda} U_{i} \subset \bigcup_{j \in \sigma_{0}}
B_{Lr_{j}}(x_{j}).
$$
In particular, we have
$$
m \big{(}\bigcup_{i \in \Lambda} U_{i}\big{)} \le L^2 \cdot
m\big{(}\bigcup_{i \in \Lambda} V_{i}\big{)}
$$ where $m(\cdot)$ denotes the area with respect to the Euclidean
metric.
\end{lem}

Recall $\Delta$ and $\Bbb T$ denote the unit disk and unit circle respectively.  Let ${\rm diam}(\cdot)$ and ${\rm dist}(\:,\: )$ denote the diameter and distance with respect to the Euclidean metric.
Let $\Omega = \Bbb C \setminus \overline{\Delta} = \{z \in \Bbb C\:|\: |z| > 1\}$. Then $\Omega$ is a hyperbolic Riemann surface.

\begin{defi}\label{admiss}{\rm Let $1 < K  < \infty$ and $z \in
X_{n+2}$.  We say $z$ is associated to a $K$-admissible pair $(U,
V)$
 if $V \subset U \subset \Omega$ are two
open topological disks such that $z \in U$ and
\begin{itemize}
\item[(1)] $V \subset X_{n}\setminus X_{n+2}$,
\item[(2)] the pair $(U, V)$ has $K$-bounded geometry,
\item[(3)] there is a Jordan domain $W$  such that  $\overline{U} \subset W \subset \Omega$ and  ${\rm mod}(W \setminus \overline{U}) >
1/K$.
\end{itemize}}
\end{defi}

Let $I \subset \Bbb T$ be an open interval. Let $\Bbb C^{*} = \Bbb C
\setminus \{0\}$ be the punctured plane. Set
\begin{equation}\label{slt}
\Omega_{I} = {\Bbb C}^{*} \setminus ({\Bbb T} \setminus I).
\end{equation}Then $\Omega_I$ is a hyperbolic Riemann surface.  For $d > 0$,
the hyperbolic neighborhood of $I$ is defined by
\begin{equation}\label{dhs}
\Omega_{d}(I) = \{z \in \Omega_{I}\:\big{|}\: d_{\Omega_{I}}(z, I) <
d\} \end{equation} where $d_{\Omega_{I}}(\:, \:)$ denotes the
hyperbolic distance in $\Omega_{I}$. The next lemma says when $I$ is small enough, $\Omega_d(I)$ is very much like the hyperbolic neighborhood in the slit plane.
\begin{lem}[cf.  Lemma 2.2 of \cite{Zh2}]\label{hn-ps}
Let $d > 0$ be given. Then for any $\epsilon > 0$ the following three assertions hold provided that $I$ is small enough:
\begin{itemize}
\item[1.] $\partial \Omega_{d}(I) = \gamma_{int} \cup \gamma_{out}$ where $\gamma_{int}$ and $\gamma_{out}$
are real analytic curve segments connecting the two end
points of $I$. Moreover,  $\gamma_{int}$ and $\gamma_{out}$ are symmetric about $\Bbb T$  such that $\gamma_{int}\setminus \partial I \subset
\Delta$ and $\gamma_{out}\setminus \partial I \subset \Bbb C
\setminus \overline{\Delta}$;

\item[2.] let $\sigma$ denote the exterior angles formed by $\gamma_{int}$ and $\Bbb T$,
$\gamma_{out}$ and $\Bbb T$, all of which are the same, then $ d =
\log \cot (\sigma/4)$,

\item[3.] Let $C_{int}$
 and $C_{out}$ be the pair of arc segments of Euclidean circles  connecting the two end points of $I$ such that the angles formed by
 $C_{int}$ and $\Bbb T$ and the angles formed by $C_{out}$ and $\Bbb
 T$ are all equal to $\sigma$, then
 $${\rm dist_H}(\gamma_{int}, C_{int})< \epsilon \cdot |I| \hbox{  and  } {\rm dist_H}(\gamma_{out}, C_{out}) < \epsilon \cdot
 |I|,$$
 where $dist_{H}$ denotes the distance between compact sets in the plane with respect to the Hausdorff
 metric.
\end{itemize}
\end{lem}

Note that  $\Omega_d(I)$ is divided by $I$ into two parts: one is in
the interior of $\Delta$ and the other one is in the exterior of
$\Delta$.  We only consider  the part which is in the exterior of
$\Delta$.  Let  $H_\sigma(I)$  denote this part.  That is,
\begin{equation}\label{hndd}
H_\sigma(I) = \{z \in \Omega_d (I)\:|\: |z| > 1\}
\end{equation}
 where $\sigma$ is
determined by the formula $d = \log \cot (\sigma/4)$.  Take  $\sigma = \pi/3$ and  let
$$
\Pi_{n-1}(B) = \{I_{n}^i, \:\:0 \le i \le q_{n-1} -1;\:\:\: I_{n-1}^i,\:\: 0 \le i \le q_n -1\}
$$
be the collection of intervals (cf. \ref{dnp}).
Define
\begin{equation}\label{ZN}
Z_{n}  =  \bigcup_{0\le i\le q_{n-1}-1} H_{\sigma}(I_{n}^{i}) \cup
\bigcup_{0\le i\le q_{n}-1} H_{\sigma}(I_{n-1}^{i}). \end{equation}
It is easy to see that $Z_n$ is the outer half of an open
neighborhood of $\Bbb T$. See Figure 2 for an illustration. As a
 consequence of  Theorem~\ref{real bounds} and the fact that ${\rm diam}(H_\sigma(I)) = O(|I|)$ (cf. Lemma~\ref{hn-ps}) ,   it follows  that there exist
$C(d) > 1$ and $0 < \epsilon(d) < 1$ depending only on $d$ such that for all  $n \ge 1$ and $B \in \mathcal{S}_d^\alpha$ with $\alpha \in \Theta_C$,
\begin{equation}\label{artes}
area(Z_n) < C(d) \cdot \epsilon(d)^n.
\end{equation}
\begin{figure}
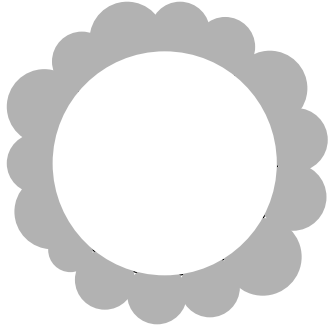

\bigskip
\begin{center}
\centertexdraw { \drawdim cm \linewd 0.02 \move(-2 1)

\move(-1.1 -1.75) \lcir  r:1.5

 \move(-2.8 -1.75) \fcir f:0.7 r:0.4
 \move(-2.7 -1.) \fcir f:0.7 r:0.5
  \move(-2.2 -0.4) \fcir f:0.7 r:0.4

  \move(-1.6 -0.1) \fcir f:0.7 r:0.5
  \move(-0.9 0) \fcir f:0.7 r:0.4
    \move(-0.3 -0.2) \fcir f:0.7 r:0.4
   \move(0.3 -0.75) \fcir f:0.7 r:0.5
  \move(0.64 -1.44) \fcir f:0.7 r:0.42
\move(0.6 -2.2) \fcir f:0.7 r:0.45

\move(0.2 -3) \fcir f:0.7 r:0.52

\move(-0.5 -3.4) \fcir f:0.7 r:0.4

\move(-1.2 -3.5) \fcir f:0.7 r:0.4

\move(-1.9 -3.3) \fcir f:0.7 r:0.4

\move(-2.6 -2.4) \fcir f:0.7 r:0.5

\move(-2.35 -2.9) \fcir f:0.7 r:0.3

\move(-1.1 -1.75) \fcir f:1 r:1.49

}
\end{center}
\vspace{0.2cm} \caption{The set $Z_n$}
\end{figure}

\begin{lem}\label{pre-lem} Let $C > 0$ and $d \ge 2$ be an integer. Then
there exist  $K > 1$  and $N_1 \ge N_0$ depending only on $d$ and $C$ such that for all $n \ge N_1$ and $B \in \mathcal{S}_d^\alpha$  with $\alpha \in \Theta_C$,  if
 $z \in X_{n+2}$,  then either $z \in Z_{n}$, or $z$ is associated to
some $K$-admissible pair $(U, V)$.
\end{lem}

\begin{pro} \label{pro-3}
{\rm Lemma~\ref{pre-lem} implies Lemma~\ref{main lemma}. }
\end{pro}
\begin{proof}
The argument is completely the same as the one used in the proof of Proposition 8.2 in \cite{Zh2}. The reader may refer to \cite{Zh2} for the details.
\end{proof}

The remaining part of the subsection is devoted to the proof of Lemma~\ref{pre-lem}.

\begin{lem}\label{Koe}
 Let $1 < L < \infty$.  Then there is a $1 < K < \infty$ depending only on $L$ such
that for any $B \in \mathcal{S}_d^\alpha$ with $\alpha \in \Theta_C$, any $z \in X_{n+2}$  and any integer $m \ge 1$, if $B^{i} (z) \in \Bbb
C \setminus \overline{\Delta}$ for all $1 \le i \le m$ and $\zeta =
B^{m}(z)$ is associated to some $L$-admissible pair,
then $z$ is associated to some $K$-admissible pair $(U, V)$.
\end{lem}
\begin{proof}
This is a direct consequence of  Koebe's distortion theorem. The argument is completely the same as the one used in the proof of Lemma 8.2 in \cite{Zh2}. The reader may refer to \cite{Zh2} for the details.
\end{proof}

Let us now begin the proof of Lemma~\ref{pre-lem}.  Let $C > 0$ and $d \ge 2$ be an integer.
Let $\alpha \in \Theta_C$ and $B \in \mathcal{S}_d^\alpha$.
  It suffices to prove that there exist $1 < K< \infty$ and $N_1 \ge N_0$
   depending only on $d$ and $C$  such that if  $z \in
X_{n+2} \setminus Z_n$ for some $n \ge N_1$, then $z$  is associated to some $K$-admissible pair $(U, V )$.

Recall that $k_z \ge 1$ is the least positive integer such
that $B^{k_z}(z) \in \Delta$. Since $z \in X_{n+2}$ we have
$B^{k_z}(z) \in Y_{n+2}$. Let us denote
$$z_l = B^l(z),\:\: 0  \le  l  \le k_z.$$
Since $z_0 = z \notin Z_n$, the set
$$\Pi= \{k \in \Bbb Z \:|\: 0 \le k < k_z \hbox{  and  } B^k(z)
\notin Z_n\}$$ is not empty. It is clear that $\Pi$ contains at most
$k_z$ elements and is thus a finite set. Let
$$k_0 = \max
_{k\in \Pi}\{k\}.$$ Then $0 \le k_0 \le k_z -1$.  Set
$$\zeta  = z_{k_0} \hbox{ and }\omega = B(\zeta) = z_{k_0+1}.$$
By the definition of $k_0$,  $\zeta \notin Z_n$, and moreover,
\begin{equation}\label{refl}
\omega \in Z_n \hbox{ if } k_0 < k_z - 1 \hbox{ and }\omega \in
Y_{n+2} \hbox{ if }k_0 = k_z - 1.\end{equation}  In the case that $\omega \in Z_n$,  there are $d$ pre-images of $\omega$ in the exterior of $\Bbb T$,  and in  the case that $\omega \in Y_{n+2}$,  there are $d-1$ pre-images of $\omega$ in the exterior of $\Bbb T$.  Since $\zeta$ belongs to the exterior of $\Bbb T$, thus  $\zeta$ is one of these pre-images. By Lemma~\ref{Koe}
it suffices to prove that $\zeta$ is associated to some
$L$-admissible pair $(U_1, V_1)$ for some uniform $1 < L < \infty$
depending only on $d$.

Let $I$ be the the arc interval in $\Pi_{n-1}(B)$ defined by (\ref{dnp}) such that either $\omega \in H_\sigma(I)$ or $\omega$ belongs to a  cell $E$  of level $n$  with $I \subset \partial E \cap \Bbb T$. By Proposition~\ref{crp} in the later case $\partial E\cap \Bbb T$ is either equal to $I$  or equal to the union of $I$ and  one of its adjacent intervals in $\Pi_{n-1}(B)$.

Now take a large constant $R = R(d) > 1$ and fix it. The dependence of $R$ on $d$ will be seen in the sequel.
Since there are at most  $d-1$ critical values  in $\Bbb T$,  for any $B \in \bigcup_{\alpha \in \Theta_C}\mathcal{S}_d^\alpha$,  there exist $\epsilon$ and $\delta$ satisfying
   \begin{itemize}
   \item[1.] $R^{-(d+1)} < \epsilon < \delta < 1$ and  $\delta/\epsilon > R$,
   \item[2.] if $v$ is a critical value of $B$ in $\Bbb T$,  then either
${\rm dist}(\omega, v) < \epsilon \cdot|I|$ or ${\rm dist}(\omega, v) > \delta \cdot|I|$.
 \end{itemize}

$\bold{General\:\:construction}$ By the geometry of the cells (cf. Lemma~\ref{geometry-of-cells}), we can construct a tuple of Jordan
domains  $V_0 \subset U_0 \subset W_0$ such that the following properties hold.
\begin{itemize}
\item[1.] $V_0 \subset E\setminus Y_{n+2}$  is a Euclidean disk,
\item[2.] $B_\omega(\delta\cdot|I|/2) \subset U_0$ and $W_0\setminus B_\omega(\delta\cdot|I|/2)$ contains no critical values of $B$,
\item[3.] $W_0 \cap \Bbb T$ is a connected arc segment,
\item[4.] ${\rm mod} (W_0 \setminus \overline{U_0}) \asymp 1$,
\item[5.]${\rm diam}(V_0) \asymp  {\rm diam}(U_0)\asymp {\rm diam}(W_0)\asymp |I|$.
\end{itemize}
The construction is very similar with the construction of the domains $B \subset A \subset \tilde{A}$ in the proof of Lemma 2.3 of \cite{Zh2}.  Now the proof is divided into two cases.

Case I. $B$ has a critical value $v$ in $\Bbb T$ such that ${\rm dist}(\omega, v) < \epsilon \cdot |I|$.

Suppose  $v \in \Bbb T$ is a critical value such that ${\rm dist}(\omega, v) > \delta \cdot |I|$.  Let $c \in \Bbb T$ be the critical point with $B(c) = v$ and $\Gamma$ be one of  the pre-circles (i.e., the pre-images of the unit circle) attached to $\Bbb T$ at $c$.   Then there is exactly one component of $B^{-1}(W_0)$, say $W_1$, which intersects $\Gamma$. It is clear that $B: W_1 \to W_0$ is a holomorphic isomorphism. Let $V_1 \subset U_1 \subset W_1$ be the domains such that $B(V_1) = V_0$ and $B(U_1) = V_0$.  Then $W_1$ contains a pre-image of $\omega$. If $\zeta$ is contained in $W_1$, then $\zeta$ is associated to the $L$-admissible pair $(U_1, V_1)$ with $L > 1$ being some constant depending only on $d$.

Suppose there are $1 \le l \le d-1$ critical points $c$, counting by multiplicities, such that ${\rm dist}(\omega, B(c)) > \delta \cdot |I|$.  Then there are $l$ pre-circles $\Gamma$ attached to $\Bbb T$ at these critical points. We have seen  that  each such  pre-circle $\Gamma$  corresponds to an $L$-admissible pair  $(U_1, V_1)$ such that $U_1$ contains a pre-image of $\omega$. Now suppose $\zeta$  is not any of these $l$ pre-images of $\omega$.  Let us   prove that $\zeta$  must belong to $Z_n$.

Note that,  if $\omega \in Z_n$, there are $d-l$ other pre-images of $\omega$ in the exterior of $\Bbb T$, and if $\omega \in Y_{n+2}$, there are $d-l-1$ other pre-images of $\omega$ in the exterior of $\Bbb T$.
Since $U_0$ contains all the  critical values $v$ in $\Bbb T$ with ${\rm dist}(\omega, v) < \epsilon \cdot |I|$,  there is a component $U_1$ of $B^{-1}(U_0)$   which intersects $\Bbb T$ such that the map $B: U_1 \to U_0$ is of degree $2d-2l-1$.  When  $\omega \in Z_n$, $U_1$ contains $d-l$ pre-images of $\omega$ which are in the exterior of $\Bbb T$, and when $\omega \in Y_{n+2}$, $U_1$ contains $d-l-1$ pre-images of $\omega$ which are in the exterior of $\Bbb T$.  This means all the remaining  pre-images of $\omega$, which we are concerned about,  are all contained in  $U_1$.  We need only to prove that these pre-images of $\omega$ must be contained in $Z_n$.   To see this,  let $I_r$ and $I_l$ be the two neighbor intervals of $E \cap \Bbb T$ in $\Pi_{n-1}(B)$. Let $S = I_r \cup (E\cap \Bbb T) \cup I_l$. Then $|S| \asymp |I|$. Then $S$ is the union of either three or four adjacent intervals in $\Pi_{n-1}(B)$.  Let $\partial S = \{p, q\}$. Let  $\omega^*$ denote the symmetric image of $\omega$ about $\Bbb T$ and  $\Pi$ denote  the set  of all the critical values.
 Consider the hyperbolic Riemann surface  $$X = \widehat{\Bbb C} \setminus (\Pi \cup \{\omega, \omega^*, p, q\}).$$   Because
$0< \epsilon < \epsilon/\delta <  R^{-1} =  R(d)^{-1}$, by taking $R$ large, we can make sure that there is a simple closed geodesic  $\gamma$ in $X$ which can be arbitrarily
 short  such that $\gamma$ encloses     $\omega$, $\omega^*$, and all those critical values with ${\rm dist}(\omega, v) < \epsilon \cdot |I|$. Since $\gamma$ is short,     $\gamma$ must intersects $S$. Since $B_\omega(\delta\cdot|I|/2) \subset U_0$ by the general construction, we have $\gamma \subset U_0$. Let $T \subset \Bbb T$ be the arc such that $B(T) = S$.  Let $\eta$ be the pre-image of $\gamma$ which intersects $T$.  Then $\eta$ is a simple closed geodesic in $Y = \widehat{\Bbb C} \setminus B^{-1} (\Pi \cup \{\omega, \omega^*, p, q\})$.  Since $\gamma \subset U_0$, $\eta \subset U_1$.  Since the covering degree of $B: \eta \to \gamma$ is not greater than $2d-1$, $\eta$ can be arbitrarily short provided that  $\gamma$ is short enough.  Thus  compared with $T$, the Euclidean diameter of $\eta$ can be arbitrarily small provided that $\gamma$ is short enough. Note that $T$ is contained in the union of at most five adjacent intervals in $\Pi_{n-1}(B)$.  By Theorem~\ref{real bounds} and the construction of $Z_n$, it follows that $Z_n \cup \overline{\Delta}$ contains a $\tau(d)\cdot |T|$-neighborhood of $T$ with $\tau(d) > 0$ depending only on $d$.  This implies that by taking $R$ large enough, we have $\eta \subset Z_n \cup \overline{\Delta}$.
  Note that the covering degree of $B: \eta \to \gamma$ is $2d-2l-1$,  it follows that all the other pre-images of $\omega$, which  belong to the exterior of $\Bbb T$,  are all contained in the interior of $\eta$.  So all the other   pre-images of $\omega$, which we are concerned about,  belong to $Z_n$.   This proves Lemma~\ref{pre-lem}  in  Case I.

Case II.  ${\rm dist}(\omega, v) > \delta |I|$ for all critical values $v$.  In this case we  may assume that $\omega \in Z_n$. This is because if $\omega \in Y_{n+2}$,  then $\omega$ has exactly $d-1$ pre-images which belong to the exterior of $\Bbb T$. As we have seen in the proof of Case I,  each of the $d-1$ pre-circles  corresponds to exactly one of the $d-1$ pre-images of $\omega$ which is associated to an $L$-admissible pair $(U_1, V_1)$ with $L >1$ being some constant depending only on $d$. Thus from now on we may assume that $\omega \in H_\sigma(I)$.

Subcase I of Case II:   $I$ contains no critical values and $I \ne I_{n}^{q_{n-1}-1}$.

Let $J \subset \Bbb T$ be the arc such that $B(J) = I$. Since $I$ contains no critical values and $B(1) \in I_{n-1}^{q_{n}-1}$,  we have  $I \ne I_{n-1}^{q_{n}-1}$. So $J$ is one of the intervals in the collection $\Pi_{n-1}(B)$.    Let $V_0 \subset U_0 \subset W_0$ be the tuple of Jordan domains described in the general construction. By a slight modification of the general construction, we may additionally assume
$$
U_0 \setminus \overline{\Delta} \subset H_\sigma(I).
$$  As we have seen before, each of the $d-1$ pre-circles intersects exactly one of the components of $B^{-1}(W_0)$, which contains a pre-image of $\omega$ and an $L$-admissible pair of domains for this pre-image. So if $\zeta$ is one of these $d-1$ pre-images of $\omega$,  the lemma has been proved.   Suppose it is not the case. Then $\zeta$ does not belong to any of these $d-1$ components.  Let $J \subset \Bbb T$ be the arc such that $B(J) = I$. Since $U_0 \cap I \ne \emptyset$, there is a component of $B^{-1}(U_0)$, say $U_1$, such that $U_1 \cap J \ne \emptyset$. This implies that $U_1$ does not intersect any of the $d-1$ pre-circles.  Thus   the last pre-image of $\omega$, which belongs to the exterior of $\Bbb T$, is contained in $U_1$.  Since $U_0 \setminus \overline{\Delta} \subset H_\sigma(I)$,
by Schwarz lemma, $U_1 \setminus \overline{\Delta}   \subset H_\sigma(J) \subset Z_n$.  It follows  that the last pre-images of $\omega$  is contained in $Z_n$.   This proves Lemma~\ref{pre-lem} in the Subcase I of Case II.

Subcase II of Case II: ${\rm dist}(\omega, \Bbb T) < \epsilon |I|$.
Let $V_0 \subset U_0 \subset W_0$ be the tuple of Jordan domains described in the general construction. Let $U_1$ be the component of $B^{-1}(U_0)$ which intersects $\Bbb T$. As in the Subcase I, it suffices to prove that the pre-image of $\omega$, which is contained in $U_1$. must be contained in $Z_n$.
To see this, let $\omega^*$ be  the symmetric image of $\omega$ with respect to  $\Bbb T$. Let $I_r$, $I_l$ be the two neighbor intervals of $I$ in $\Pi_{n-1}(B)$. Let $S = I_r \cup I \cup I_l$ and $\partial S = \{p, q\}$.  Then $|S| \asymp |I|$. Let $\pi$ denote the set of the critical values of $B$.  Since ${\rm dist}(\omega, \omega^*) < 2\epsilon |I|$,  ${\rm dist}(\omega, v) > \delta |I|$  and $\delta / \epsilon > R = R(d)$,
there is a short simple closed geodesic $\gamma$ in $$X = \widehat{\Bbb C} \setminus (\Pi \cup \{\omega, \omega^*, p, q\})$$ which  contains  $\omega$ and $\omega^*$ and no critical values  in its inside, provided that $R(d)$ is chosen large enough.   Then  $\gamma$ intersects $S$ and  can be arbitrarily short provided that $R$ is  large enough. Since $B_\omega(\delta\cdot|I|/2) \subset U_0$, we have $\gamma \subset U_0$.
  Let $T \subset \Bbb T$ be the arc such that $B(T) = S$.  Let  $\eta$ be the pre-image of $\gamma$ which intersects $T$. Then $\eta \subset U_1$. Then $\eta$  is a short simple closed geodesic in $Y = \widehat{\Bbb C} \setminus B^{-1}(\Pi \cup \{\omega, \omega^*, p, q\})$.  Since $\gamma$ encloses no critical values, the degree of $B: \eta \to \gamma$ is $1$. Hence     $\eta$ can be arbitrarily short provided that $R$ is large enough.   This implies that,  compared with $|T|$, the Euclidean diameter of $\eta$ can be arbitrarily small provided that $R$ is large enough. Since $T$ is contained in the union of at most four adjacent intervals in $\Pi_{n-1}(B)$,  by Theorem~\ref{real bounds} and the construction of $Z_n$, it follows that $Z_n \cup \overline{\Delta}$ contains a $\tau(d)\cdot |T|$-neighborhood of $T$ with $\tau(d) > 0$ depending only on $d$.  Thus by taking $R = R(d)$ large enough, we can make sure that    $\eta \subset Z_n \cup \overline{\Delta}$.   Since $\eta$ encloses a pre-image  of $\omega$, say $\zeta_0$,   it follows that $Z_n$ contains  $\zeta_0$. Since $\eta \subset U_1$, $\zeta_0 \in U_1$. This proves the lemma in the Subcase II of Case II.

Subcase III of Case II:  ${\rm dist}(\omega, \Bbb T) \ge  \epsilon |I|$ and
$I$ either contains at least one critical value or  $I = I_{n}^{q_{n-1}-1}$.  Since $I_{n-1}^{q_{n}-1}$ contains $B(1)$ and is adjacent to  $I = I_{n}^{q_{n-1}-1}$ in the collection $\Pi_{n-1}$,  either $I$ or one of its adjacent intervals in $\Pi_{n-1}$ contains  at least one critical value of $B$.   Since ${\rm dist}(\omega, v) > \delta |I|$ for all critical values $v$, we can construct two tuple of Jordan domains $V_0 \subset U_0 \subset W_0$ and $V_0' \subset U_0' \subset W_0'$ satisfying the properties (1), (3), (4) and (5) described in the General construction, and besides, the following three properties hold also:
\begin{itemize}
\item[1.] $V_0 = V_0'$,
\item[2.] both $W_0$ and $W_0'$  contain no critical values,
\item[3.]  $W_0 \cup W_0'$ is a topological annulus and separates  at least one critical value from $\infty$.
\end{itemize}

Let $c_i, 1 \le i \le d-1$ denote all the critical points in $\Bbb T$, counting by multiplicities and labeled by order. For $1 \le i \le d-1$, let $\Gamma_i$ denote  the pre-circle attached to $\Bbb T$ at $c_i$, and let $W_1^i$ and ${W_1^i}'$  denote respectively the component of $B^{-1}(W_0)$ and $B^{-1}(W_0')$ which intersects $\Gamma_i$, and let $\phi_i: W_0 \to W_1^i$ and $\phi_i': W_0 \to {W_1^i}'$ denote the corresponding inverse branch of $B^{-1}$. Let $F_0$ denote the bounded component of $\Bbb C \setminus \overline{W_0 \cup W_0'}$.  Suppose $B(c_i), 1 \le i \le l$, are all the critical values which are separated by $W_0 \cup W_0'$ from $\infty$, that is, contained in $F_0$.

For $l+1 \le i \le d-1$, since $F_0$ does not contain $B(c_i)$,
 $\phi_i$ and $\psi_i$ can be extended to the same univalent function defined on the Jordan domain $\overline{F_0} \cup W_0 \cup W_0'$. Thus $W_1^i$ and ${W_1^i}'$ contains the same pre-image of $\omega$. As we have seen before, each of these pre-image is associated to some $L$-admissible pair.

The situation is a little bit  subtle for $1 \le i \le l$. For these $\Gamma_i$, since $B(c_i)$ are separated by $W_0 \cup W_0'$, the  maps  $\phi_i: W_0 \to W_1^i$ and $\phi_i': W_0 \to {W_1^i}'$  represent different branches of $B^{-1}$.  Thus  $W_1^i$ and ${W_1^i}'$ contain two different pre-images of $\omega$:  one is to the left of $\Gamma_i$, and the other one is to the right of $\Gamma_i$. Without loss of generality, let us assume that for all $1 \le i \le l$,  the pre-image of $\omega$ contained in $W_1^i$ is to the left of $\Gamma_i$, and the pre-image of $\omega$ contained in ${W_1^i}'$ is to the right of $\Gamma_i$. Let us also assume that $\Gamma_i$, $1 \le i \le l$,  are ordered from left to right. Then the leftmost pre-image of $\omega$  is contained in $W_1^1$, next to that, for each $1 \le i \le l-1$, there is a pre-image of $\omega$ contained in both $W_1^{i+1}$ and ${W_1^i}'$, and the rightmost one is contained in ${W_1^l}'$.  These are exactly the remaining $l+1$ pre-images of $\omega$. As before  each of these pre-images of $\omega$ is associated to some $L$-admissible pair.

This proves Lemma~\ref{pre-lem} in the Subcase III of Case II. The proof of Lemma~\ref{pre-lem} has been completed.  Lemma~\ref{uniform-w} thus follows.

\subsection{Proof of  Key-Lemma 1}
\begin{lem}\label{dr}
Let $M, \beta > 0$ and $0< \epsilon_0 < 1$. Let $\Psi_{M, \beta, \epsilon_0}$ denote the family of all $(M, \beta, \epsilon_0)$-David homeomorphisms
 of the plane to itself which fix $0$ and $1$. Then there exist positive functions $\vartheta, \iota: (0, 2] \to (0, \infty)$ such that \begin{equation}\label{zero-c}\lim_{\delta \to 0+} \vartheta(\delta) =  \lim_{\delta \to 0} \iota(\delta) = 0\end{equation} and  for any $\phi \in \Psi_{M, \beta, \epsilon_0}$  and any two $z_1, z_2 \in \Bbb T$ we have
\begin{equation}\label{sup-inf}
  |\phi(z_1) - \phi(z_2)|  \ge \iota(|z_1 - z_2|) \hbox{  and  } |\phi^{-1}(z_1) - \phi^{-1}(z_2)| \le \vartheta (|z_1 - z_2|).
 \end{equation}
\end{lem}
\begin{proof}
Let $$\vartheta(\delta) = \max\{\sup |\phi(z_1) - \phi(z_2)|, \sup |\phi^{-1}(z_1) - \phi^{-1}(z_2)|\} $$  and
$$
 \iota(\delta) = \min\{\inf |\phi(z_1) - \phi(z_2)|, \inf |\phi^{-1}(z_1) - \phi^{-1}(z_2)|\}
$$
where ${\rm sup}$  is taken over all $\phi \in \Psi_{M, \beta, \epsilon_0}$ and all  the pairs $z_1, z_2 \in \Bbb T$ with $|z_1 - z_2| \le  \delta$
and ${\rm inf}$ is  taken over  all $\phi \in \Psi_{M, \beta, \epsilon_0}$ and all the pairs $z_1, z_2 \in \Bbb T$ with $|z_1 - z_2| \ge  \delta$.  Then By Lemma~\ref{normal} and a compactness argument it follows that both $\vartheta$ and $\iota$ are positive functions satisfying (\ref{zero-c}) and  (\ref{sup-inf}).
\end{proof}

Recall that $B_f$ is the Blaschke product which models $f$.   Let $H_f: \Delta \to \Delta$ be the David homeomorphism  constructed in $\S4.5$.   Let $\mu$ be the $\widehat{B}_f$-invariant Beltrami differential obtained by pulling back $\mu_{H_f}$ through the iteration of $\widehat{B}_f$.   Then  $\mu$ satisfies the integrability condition (\ref{end-k1}) in Lemma~\ref{uniform-w}.  Let $\phi$ be the David homeomorphism of the complex plane to itself which fixes $0$ and $1$, and satisfies the Beltrami equation $\phi_{\bar{z}} = \mu(z) \phi_z$. Then $\phi$ is a $(\tilde{M}, \tilde{\beta}, \tilde{\epsilon}_0)$-David homeomorphism with $\tilde{M}, \tilde{\beta}, \tilde{\epsilon}_0$ depending only on $d$ and $C$. Since $\alpha \in \Theta_C^b$ is of bounded type, $\phi: \Bbb C \to \Bbb C$ is still a qc homeomorphism. By the combinatorial rigidity of $f$ (cf.  Theorem~\ref{Thurston-Siegel-Ch}),
we have  $$f =\phi \circ \widehat{B}_f \circ \phi^{-1}.$$
Now define
\begin{equation}
\widehat{H}_f(z) =
\begin{cases}

 H_f(z) & \text{for $z \in  \overline{\Delta}$}, \\

[H_f(z^*)]^* & \text{for $z \in \Bbb C \setminus \Delta$}.
\end{cases}
\end{equation}
 Here $w^*$ denote the symmetric image of $w$ with respect to $\Bbb T$. Note that there exist $r > 0$ and $K > 1$ depending only on $d$ and $C$ such that the dilatation of $H_f$ is bounded by $K$ in $B_r(0)$ (cf. $\S$ 4.5). Note also that for any measurable set $E \subset \{z\:|\: r < z < 1\}$, we have $area(E^*) \le L \cdot area(E)$, where $E^*$ denote the symmetric image of $E$ with respect to $\Bbb T$, and $L >1$ is a constant depending only on $r$ and thus depending only on $d$ and $C$, and $area(\cdot)$ denotes the area with Euclidean metric. Thus by increasing $\tilde{M}$ if necessary, we may assume that  $\widehat{H}_f: \Bbb C \to \Bbb C$ is a $(\tilde{M}, \tilde{\beta}, \tilde{\epsilon}_0)$-David homeomorphism.  For such $\tilde{M}, \tilde{\beta}, \tilde{\epsilon}_0$, let $\vartheta$ and $\iota$ be the functions in Lemma~\ref{dr}.  Then define $\lambda_1, \eta_1: (0, 2] \to (0, \infty)$ by setting for any $\delta \in (0, 2]$
 $$
 \lambda_1(\delta) = \vartheta(\min\{\vartheta(\delta), 2\}) \hbox{  and  }  \eta_1(\delta) = \iota(\min\{\iota(\delta), 2\}).
 $$
 By (\ref{zero-c}) we have $\lim_{\delta \to 0+} \lambda_1(\delta) =  \lim_{\delta \to 0} \eta_1(\delta) = 0$.  Suppose $k > m \ge 0$ are two integers such that $\delta'< |e^{2 \pi i k\alpha} - e^{2 \pi i m \alpha}| < \delta$ for some $0 < \delta'< \delta \le 2$. Then
 $$
 |f^k(1) - f^m(1)| = |\phi \circ H_f^{-1} \circ R_{\alpha}^k \circ H_f \circ \phi^{-1}(1) - \phi \circ H_f^{-1} \circ R_{\alpha}^m \circ H_f \circ \phi^{-1}(1)|.
 $$
 Since both $H_f$ and $\phi$ fix $1$, it follows that
\begin{equation}\label{fin}
 |f^k(1) - f^m(1)| = |\phi \circ H_f^{-1} (e^{2 \pi i k \alpha})  - \phi \circ H_f^{-1} (e^{2 \pi i m \alpha})|.
 \end{equation}
 Since $H_f^{-1} (e^{2 \pi i k \alpha}), H_f^{-1} (e^{2 \pi i m \alpha}) \in \Bbb T$ and
 $$\min\{\iota(\delta'), 2\} \le |H_f^{-1} (e^{2 \pi i k \alpha})- H_f^{-1} (e^{2 \pi i m \alpha})|  \le \min\{\vartheta(\delta), 2\},$$  by Lemma~\ref{dr} we have $$
 \iota(\min\{\iota(\delta'), 2\})\le |\phi \circ H_f^{-1} (e^{2 \pi i k \alpha})  - \phi \circ H_f^{-1} (e^{2 \pi i m \alpha})|\le \vartheta(\min\{\vartheta(\delta), 2\}).
 $$
This completes the proof of Key-Lemma 1.

\section{Proofs of  Key-Lemmas 3 and 4}

Let $m , l \ge 1$ be two integers. Let $z = (z_1, \cdots, z_m)$ and $w = (w_1, \cdots, w_l)$ denote the points in $\Bbb C^m$ and $\Bbb C^l$, respectively.

\begin{lem}\label{fundamental-E-a}
Let  $Q_i(z, w)$, $1 \le i \le m$, be $m$ polynomials of $m+l$
complex variables.
Suppose   there exist open sets $U \subset \Bbb C^l$ and  $V \subset \Bbb C^m$   and  $m$ holomorphic functions
$$
z_i = g_i(w_{1}, \cdots, w_{l}), \:\: 1 \le i \le m,
$$ defined in $U$  such that
\begin{itemize}
\item[(1)]  for any $w \in U$, $(g_1(w), \cdots, g_m(w)) \in V$ and $Q_i(g_1(w), \cdots, g_m(w), w) = 0$ for all $1 \le i \le m$,
\item[(2)]   for any points $w \in U$ and $z \in V$, if
 $Q_i(z,w) = 0$ for all $1 \le i \le m$,  then
 $z_i = g_i(w)$ for  all $1\le i \le m$.
\end{itemize}
Then there  exist $m$ irreducible polynomials $P_i$, $1 \le i \le m$,
of $l+1$ variables such that
$$
P_i(g_i(w), w) = 0, \:\forall w \in U, \:\: 1 \le i \le m.
$$
\end{lem}
\begin{proof}
Let $\mathcal{S}$ denote the system of the $m$ polynomials $Q_i(z, w)$, $1 \le i \le m$. By replacing $Q_i$ by one of its irreducible factors if necessary,  we may assume that all $Q_i$, $1 \le i \le m$,  are irreducible.

Let us first  write all $Q_i$ into polynomials of $z_1$ with
coefficients being polynomials of the other $m+l-1$ variables.
If some $Q_i$ has a term with coefficient, say $h(z_2, \cdots, z_m, w)$,
satisfying $h(g_2(w), \cdots, g_{m}(w), w) = 0$ for all $w\in U$, then we add the polynomial
$h$ to  $\mathcal{S}$,  and at the same time, delete the
corresponding term from $Q_i$. In this way we get a new system of polynomials. By replacing a polynomial by one of its irreducible factors if necessary,   we can make sure all the polynomials in the new system are still irreducible. Besides this, if one polynomial is the constant multiple of the other, we just
remove one of them from the system. In the following we will repeat this process to keep  the new system  not redundant. Let us still use $\mathcal{S}$ to denote the new system.  Note that for the new system, the conditions (1)  and (2)  in the lemma are still satisfied.

We claim that there is at least one  polynomial in $\mathcal{S}$ which involves
$z_1$.  Suppose the claim were not true. Then take $z^0 = (z_1^0, \cdots, z_m^0) \in V$ and $w^0 = (w_1^0, \cdots, w_l^0) \in U$ such that $Q(z^0,w^0) = 0$ for all polynomial $Q$ in $\mathcal{S}$.
Now take $z_1^* \ne z_1^0$ such that $(z_1^*, z_2^0, \cdots, z_m^0) \in V$. Then
   for each polynomial $Q $ in $\mathcal{S}$, since $Q$ does not involves $z_1$, we have $Q(z_1^*,  z_2^0,\cdots, z_m^0, w^0) = 0$.
 This contradicts with the condition (2).  The claim has been proved.

Now suppose $Q_1$ is  a polynomial in  $\mathcal{S}$ which involves $z_1$ and  moreover, among all the polynomials in $\mathcal{S}$  which involve $z_1$,  the degree of $Q_1$ with respect to $z_1$ is the  lowest.
Assume  there is some other polynomial, say $Q_2$,  in $\mathcal{S}$, which also involves $z_1$.  Otherwise we go to the next step. let us do the polynomial long division as follows,
\begin{equation}\label{division}
 Q_2 = D \cdot Q_1 + R
 \end{equation}
 where $D$ and $R$ are  polynomials of $z_1$ with
 coefficients being rational functions of $$z_2, \cdots, z_m, w_1, \cdots, w_l.$$
Note that the degree of $R$  with respect to $z_1$ is less than that of $Q_1$, and is thus less than that of $Q_2$. Since all polynomials in $\mathcal{S}$ are irreducible,
  $R$ is not identically zero. Let $h(z_2, \cdots, z_m, w)$ be the coefficient of the leading term of $Q_1$.
From the process of polynomial long division, we  know if $D = D_1/D_2$ and $R = R_1/R_2$,  then  $D_2$ and $R_2$ are  both the powers of $h(z_2, \cdots, z_m, w)$. Since $h(g_2(w), \cdots, g_m(w), w)$ is not identically zero in $U$,  both
$D_2(g_2(w), \cdots, g_m(w), w)$ and   $R_2(g_2(w), \cdots, g_m(w), w)$
are not identically zero in $U$.  Now we replace $Q_2$ by $R_1$ and get a new system of polynomials.  We then repeat the procedure used before to make sure that the polynomials in the new system are all irreducible and the new system is  not redundant.  From (\ref{division}) it follows that
  under the condition that $h(g_2(w), \cdots, g_m(w), w)\ne 0$,   the two equations $R_1 = 0$ and $Q_1 = 0$
  imply $Q_2 = 0$; and on the other hand, the two equations $Q_1 = 0$ and $Q_2 = 0$
  imply $R_1 = 0$.  Since $h(g_2(w), \cdots, g_m(w), w)$ is not identically zero in $U$,  it follows that, by replacing $U$  by an  open subset of $U$ on which $h(g_2(w), \cdots, g_m(w), w)$ does not vanish,  the new system still satisfies the two conditions (1) and (2). Let us still use $\mathcal{S}$ to denote the system of polynomials.

Note that after the above process,   the sum of the degrees of all polynomials in $\mathcal{S}$ with respect to $z_1$, is decreased at least by $1$. So after finitely many steps,  there is only one polynomial in $\mathcal{S}$ which involves $z_1$, and moreover,  the two conditions (1) and (2) are still satisfied with $U$ replaced by an appropriate open subset of $U$. Let us still use $U$ denote this open subset.

Now  we claim that, except $Q_1$,  there is at least one polynomial in the system which involves  $z_2$. Suppose the it were not true.  Take $(z_1^0, \cdots, z_m^0) \in U$ and $(w_1^0, \cdots, w_l^0) \in V$ such that $Q(z_1^0, \cdots, z_m^0, w_1^0, \cdots, w_l^0)=0$ for all $Q$ in $\mathcal{S}$. Let us write $$Q_1(z, w) = h_lz_1^l + \cdots + h_0$$ where $h_i$, $0 \le i \le l$, are  polynomials of $z_2, \cdots, z_m, w_1, \cdots, w_l$. Since $Q_1$ is the only polynomial involves $z_1$, by the condition (2), there is some $h_i$ such that $$h_i(z_2^0, \cdots, z_m^0, w_1^0, \cdots, w_l^0) \ne 0.$$ But then for any $z_2^*$ near $z_2^0$, by Rouche Theorem, there is some $z_1^*$ near $z_1^0$ such that $$Q_1(z_1^*, z_2^*, z_3^0, \cdots, z_m^0, w_1^0, \cdots, w_l^0) = 0.$$ Since the other polynomials in $\mathcal{S}$ does not involve $z_1$ and $z_2$, we have $$Q(z_1^*, z_2^*, z_3^0, \cdots, z_m^0, w_1^0, \cdots, w_l^0) = 0$$ for all $Q$ in $\mathcal{S}$, This contradicts with the condition (2) and the claim has been proved.

Now let $Q_2$ be a polynomial in the system which involves $z_2$ and whose degree with respect to $z_2$ is the lowest. If except $Q_1$ and $Q_2$, there are no other polynomials which involve $z_2$, we go to the next step. Otherwise,   choose a polynomial, say $Q_3$, which involves $z_2$. Then we repeat the process of the polynomial long division as in (\ref{division}) for $Q_2$ and $Q_3$, with respect to $z_2$.  In this way,  after finitely many steps  we get a new system $\mathcal{S}$ such that except $Q_1$, there is only one polynomial $Q_2$  which involves $z_2$, and moreover, the conditions (1) and (2) are satisfied by replacing $U$ by some appropriate open subset of $U$.

Using the same argument as before, one can prove that except $Q_1$ and $Q_2$, there is some other polynomial in $\mathcal{S}$, say $Q_3$, which involves $z_3$.  Repeating the above process,   we finally get an irreducible polynomial of $l+1$ variables, say $Q_m$, such that
$Q_m(g_m(w), w) = 0$.
Let $P_m = Q_m$.  By relabeling each of $z_1, \cdots, z_{m-1}$ as $z_m$ and repeating the above process, we get an irreducible polynomial $P_i$ of $l+1$ variables such that
$P_i(g_i(w), w) = 0, \: 1 \le i \le m-1$.
The proof of Lemma~\ref{fundamental-E-a} is completed.
\end{proof}

Now let  us recall some basic notions of algebraic functions.  For more knowledge in this aspect, the reader may refer to \cite{Ah} and \cite{FK}.  Suppose $P(w, z)$ is an irreducible polynomial.
   Suppose $P$ has degree $m$ with respect to $w$. That is,
$$
P(w, z) = b_0(z) w^m + b_1(z) w^{m-1}+ \cdots +b_m(z)
$$ where $b_i(z)$ are  polynomials of $z$ and $b_0(z)$ is not identically zero.  Let $R(z)$ be the resultant of $P(w, z)$ and $P_w(w, z)$.
Let $\Pi(P) = \{z \in \Bbb C\:|\: R(z) = 0 \hbox{  or  } b_0(z) = 0\}$.   We call $\Pi(P)$ the set of $\emph{algebraic singularities}$ of $P$.  Then for
any $z \in {\Bbb C} \setminus \Pi(P)$, there are exactly $m$ distinct $w$ in $\widehat{\Bbb C}$ such that $P(w, z) = 0$. Thus  the equation $P(w, z) = 0$  determines a multi-valued analytic function $w = w(z)$ in the sense
$P(w(z), z) = 0$.

\begin{lem}\label{algebraic-curve}
Suppose $P_i(w_i,z)$, $1 \le i \le l$, are $l$ irreducible polynomials of two variables. Let $$\Pi = \bigcup_{1 \le i \le l} \Pi(P_i) \cup \{\infty\}.$$   Let $z^0 \in \widehat{\Bbb C} \setminus \Pi$  and  $w_i^0 \in \Bbb C$  with  $P_i(w_i^0, z_0) = 0$ for $1 \le i \le l$. Then there exists a compact Riemann surface $S$ and meromorphic functions $z$   and $w_i$, $1 \le i \le l$, defined in $S$ such that
\begin{itemize}

\item[1.]  there is a $t_0 \in S$ such that $z^0 = z(t_0)$  and $w_i^0 = w_i(t_0)$, $1 \le i \le l$,
\item [2.] $P(w_i(t), z(t)) = 0$ for $t \in S\setminus P$ with $P$ being the set of poles of $z$ and all $w_i$, $1 \le i \le l$.
\end{itemize}
\end{lem}
\begin{proof} It is a classical theorem when $l = 1$ and  the reader may refer to \cite{FK} for a detailed proof. The idea is completely the same for $l > 1$. So let us merely present an outline of the proof as follows.

Let $W(z) = (w_1(z), \cdots, w_l(z))$ be the vector of holomorphic germs at $z^0$ such that $w_i(z^0) = w_i^0$ and $P(w_i(z), z) = 0$ for $1 \le i \le l$.  For any point $\zeta \in \widehat{\Bbb C} \setminus\Pi$ and any path $\gamma \subset \widehat{\Bbb C} \setminus \Pi$ connecting $z$ and $\zeta$, we can analytically continue $W$  along $\gamma$ so that $P_i(z, w_i(z)) = 0$ for $1 \le i \le l$. Here continuing $W$ along a path $\gamma$ means continuing all the components of $W$ simultaneously along $\gamma$.  In this way we get a vector of holomorphic germs  at $\zeta$. Let $S$ denote the set of all such vectors of holomorphic germs. Define $\pi : S \to \widehat{\Bbb C} \setminus\Pi$ by  $(w_1(z), \cdots, w_l(z)) \mapsto z$.  As in the case that $l = 1$, one can first put a topology  and then introduce a complex chart on $S$  so that $S$ becomes into a Riemann surface and the map $\pi: S \to  \widehat{\Bbb C} \setminus\Pi$ is a holomorphic covering map. It is clear that $z$ and all $w_i$ are holomorphic in $S$. Since there can only be finitely many distinct germs over each $z$,  $\pi$ is a finite  covering map. Let  $\deg(\pi)$ denote  the covering degree.

To compactify $S$ we need to fill in  the points lying over $\Pi$. Let $p \in \Pi$ be a point and $U$ be a small disk about $p$ and not containing any other points in $\Pi$. Let $V$ be a component of $\pi^{-1}(U\setminus \{p\})$. Then $\pi|V: V \to U\setminus \{p\}$ is a holomorphic covering map and $(\pi|V)^*(\pi_1(V))$ is a subgroup of $\pi_1(U\setminus \{p\}) = \Bbb Z$. Thus
$(\pi|V)^*(\pi_1(V)) = k \Bbb Z$ for some integer  $k$ with $|k|\le \deg(\pi)$.  But on the other hand, there is a disk $W$ about the origin such that the holomorphic covering map $\tau: W\setminus \{0\} \to U \setminus \{p\}$ given by   $\tau(z) = z^k + p$, also satisfies  $(\tau)^*(\pi_1(W\setminus \{0\})) = k \Bbb Z$. This implies that  there is a holomorphic isomorphism $\phi: W\setminus \{0\} \to V$.  By identifying $z$ with $\phi(z)$ for $z \in W$, we can replace $V$ by $W$ in $S$ and thus fill in the ``hole'' in $V$. Repeat this procedure for all the other components of $\pi^{-1}(U\setminus \{p\}$ and all the other points in $\Pi$. In this way we fill in all the points ``lying'' over $\Pi$ and $S$ is compactified. As in the case that $l=1$, one can introduce  more charts and  sets respectively to the previous  atlas of charts and topology of $S$ so that the ``compactified" becomes into a Riemann surface.

 Since $z$ and all $w_i$ are holomorphic in $S$ except the points lying over $\Pi$, to prove they are meromorphic in $S$, it suffices to show that the points lying over $\Pi$ are either poles or regular points for $z$ and $w_i$.  To see this,  take take an arbitrary   $q \in S$  which lies above some point in $\Pi$. Let $D$ be a small disk about $q$. Since $\pi: S \to \widehat{\Bbb C}$ is a finitely branched covering map, $z$ takes any complex value finitely often in $D$. This implies that $q$ is either a pole or a regular point of $z$. Since $q$ is arbitrary, $z$ is a meromorphic function in $S$ with the poles being exactly those points lying above $\infty$. Now if $q$ is an essentially singularity of some $w_i$, then with at most one exception,
  $w_i$ takes any complex value, say $\alpha$,  infinitely often in $D$. Since $P_i(z, w)$ is irreducible, $P_i(z, \alpha) = 0$ has finitely many roots. This implies that $z$ must take some value infinitely often in $D$. This is a contradiction since $z$ is meromorphic. Thus $w_i$ is also a meromorphic function in $S$.

    The  assertion (1) now follows by taking $t_0$ to be the point represent by the vector $W(z) = (w_1(z), \cdots, w_l(z))$ of holomorphic germs  at $z^0$ such that $w_i(z^0) = w_i^0$, $1 \le i \le l$. The assertion (2) follows by the construction.

\end{proof}

Let us start the proof of Key-Lemma 3. Suppose $f$ has $l$ attracting cycles with multipliers non-zero $t_1^0, \cdots, t_l^0$.
For each $1 \le i \le l$, let $x_i^0$ be one of the points in the $i$-th periodic attracting cycle and
$p_i \ge 1$ be the period.   Thus we have
$$
f_{c_1^0, \cdots, c_{d-2}^0, 1}^{p_i}(x_i^0) - x_i^0 = 0 \hbox{  and } Df_{c_1^0, \cdots, c_{d-2}^0, 1}^{p_i}(x_i^0) - t_i^0 =0.$$
Since the cycle  is attracting, there exist open neighborhoods $U$ of $(c_1^0, \cdots, c_{d-2}^0)$ and  $V$ of $(x_i^0, t_i^0)$,
and two holomorphic functions
\begin{equation}\label{imp-f-1}x_i = \alpha_i(c_1, \cdots, c_{d-2})  \hbox{  and } t_i = \lambda_i(c_1, \cdots, c_{d-2})
\end{equation}
defined in $U$ such that for all $(c_1, \cdots, c_{d-2}) \in U$ we have
\begin{equation}\label{pre-pe}
 f_{c_1, \cdots, c_{d-2}, 1}^{p_i}(x_i) - x_i = 0 \hbox{  and } Df_{c_1, \cdots, c_{d-2}, 1}^{p_i}(x_i) - t_i =0,
\end{equation}
and moreover, for any $(c_1', \cdots, c_{d-2}') \in U$ and $(x_i', t_i') \in V$ satisfying (\ref{pre-pe}),  we have
$x_i' = \alpha_i(c_1', \cdots, c_{d-2}')$ and  $t_i' = \lambda_i(c_1', \cdots, c_{d-2}')$.

\begin{lem}\label{non-de-J}
Let $t_i = \lambda_i(c_1, \cdots, c_{d-2})$, $1 \le i \le l$, be the functions given by (\ref{imp-f-1}).  Then
by relabeling $c_1, \cdots, c_{d-2}$ if necessary,
we have $$\bigg{|}\frac{\partial(t_1, \cdots, t_l)}{\partial(c_1, \cdots, c_l)}
\bigg{|}_{(c_1^0, \cdots, c_{d-2}^0)} \ne 0$$
\end{lem}

By Lemma~\ref{non-de-J} and Implicit function Theorem,  there exist a vector of holomorphic functions
$$(g_1(c_{l+1}, \cdots, c_{d-2},t_1 \cdots, t_l), \cdots,g_l(c_{l+1}, \cdots, c_{d-2},t_1 \cdots, t_l))$$ defined in a polydisk neighborhood $U_0$ of $(c_{l+1}^0, \cdots, c_{d-2}^0, t_1^0, \cdots, t_l^0)$ and taking values in a polydisk neighborhood $V_0$ of $(c_1^0, \cdots, c_l^0)$ such that $t_i = \lambda_i(g_1, \cdots, g_l, c_{l+1}, \cdots, c_{d-2})$ holds in $U_0$, and moreover, for any $(c_{l+1}, \cdots, c_{d-2},t_1, \cdots, t_l) \in U_0$ and $(c_1, \cdots, c_l) \in V_0$, if $t_i = \lambda_i(c_1, \cdots, c_{d-2})$, then $c_i=g_i(c_{l+1}, \cdots, c_{d-2},t_1 \cdots, t_l)$. Now take $t_i= t_i^0$ for $1 \le i \le l$ and  $c_i = c_i^0$ for  $l+2\le i \le d-2$.  Then the systems of equations in (\ref{pre-pe}) uniquely determined a group of holomorphic functions $x_i$ and $c_i$, $1 \le i \le l$, in a small neighborhood of $c_{l+1}^0$.
By multiplying an appropriate power of $c_1 \cdot c_2 \cdots \cdot c_{l}$ on both sides of the equations in (\ref{pre-pe}) we get a system of polynomial equations
\begin{equation}\label{pol-ee}R_i(x_i, c_1, \cdots, c_{l+1}) = 0 \hbox{  and  } S_i(x_i, c_1, \cdots, c_{l+1}) = 0, \:\: 1 \le i \le l.
\end{equation}
Now apply Lemma~\ref{fundamental-E-a}, we have
\begin{lem}\label{multiplier-analytic}
There exist an open neighborhood $U$ of $c_{l+1}^0$ such that for each $1 \le i \le l$, there exist irreducible polynomials $P_i$ and $Q_i$ of two variables satisfying $$P_i(c_i, c_{l+1}) = 0  \hbox{  and  } Q_i(x_i, c_{l+1}) =0$$ for all $c_{l+1} \in U$.
\end{lem}

Now let us  prove Lemma~\ref{non-de-J}.  By (\ref{imp-f-1}) it follows that
$$\Phi: (c_1, \cdots, c_l, c_{l+1}, \cdots, c_{d-2}) \to (t_1, \cdots,
t_l)$$ is a  holomorphic map in a neighborhood $U$
  of $(c_1^0, \cdots, c_l^0, c_{l+1}^0, \cdots, c_{d-2}^0)$ and  $$t_i^0 = \lambda_i(c_1^0, \cdots, c_l^0, c_{l+1}^0, \cdots, c_{d-2}^0), \:\:1 \le i \le l.$$  It suffices to  the existence of a small neighborhood $W$ of $(t_1^0, \cdots, t_l^0)$ and a holomorphic map
$\Psi: W \to U$ such that
 $\Phi\circ \Psi = \hbox{id}$.   It is clear that this will imply Lemma~\ref{non-de-J}.   The construction
of $\Psi$ is as follows.

 For $1 \le i \le l$, let $x_i^0$ and $p_i$ be as before. For each $x_i^0$,  Let $U_i$ be a Jordan domain with real analytic boundary such that $f^{p_i}: U_i \to f^{p_i}(U_i)$ is conjugate to  $z \to t_i^0 z$.  Let $U_i^k = f^k(U_i)$ for $0 \le k \le p_i$. For each $1 \le i \le l$,  the map
$f: U_i^{p_i-1} \to U_i^{p_i}$ is  a holomorphic isomorphism. Let $\phi_i:  \Delta \to U_i^{p_i-1}$ and $\psi_i: \Delta \to U_i^{p_i}$ be  the holomorphic isomorphisms with $\phi_i(0) = f^{p_i-1}(x_i), \phi_i'(0) > 0$, and  $\psi_i(0) = x_i$.  Then we can lift the map $f: U_i^{p_i-1} \to U_i^{p_i}$ to a holomorphic isomorphism $\Lambda_i: \Delta \to \Delta$ with $\Lambda_i(0) = 0$.  It is clear that $\Lambda_i (z) = \lambda  \cdot z$ for some $|\lambda| = 1$. By choosing an appropriate argument of $\psi_i'(0)$, we can assume that $\lambda = 1$ and  $\Lambda_i(z) = z$.

For $r > 0$ let
 $\Delta_r = \{z\:|\:|z| < r\}$. Now take $0< r < 1$ such that
 $\phi_i(\Delta_r) \supset f^{p_i-1}(U_i)$ for all $1 \le i \le l$.  For an $\epsilon > 0$ small we  define $\Lambda_{i,s_i}: \Delta \to \Delta$ for all $|s_i| < \epsilon$ by

$$
\Lambda_{i,s_i}(z) =
\begin{cases}

 (1 + s_i)  z & \text{for $|z| < r$}, \\
 (1 + \frac{1-|z|}{1-r}s_i)z & \text{for $r \le |z| < 1$}.
\end{cases}
$$

For $s = (s_1, \cdots, s_l)$, define a family of quasi-regular map $g_s$, $|s| = \sum_i|s_i| < \epsilon$, by
   $$
g_s(z) =
\begin{cases}
\psi_i \circ \Lambda_{i,s_i} \circ \phi_i^{-1} & \text{ if $z \in U_i^{p_i-1}$}  \\
f(z)   & \text{for otherwise}.
\end{cases}
$$
From the above construction, it follows that for all $|s| < \epsilon$, the $i$-th attracting  cycle of $f$, $1 \le i \le l$,  is an attracting cycle of $g_s$ with multiplier $(1 + s_i)t_i^0$.  Now  pulling back the standard complex structure by the iterations of $g_s$,
  we get a $g_s$-invariant complex structure $\mu_s$ on the whole plane  which depends analytically on $s$. Let $\phi_s$ be the qc homeomorphism of the plane to itself which solves the Beltrami equation given by $\mu_s$ and fixes $0$ and $1$. Then $f_s = \phi_s \circ g_s \circ \phi_s^{-1}$ is a polynomial of degree $d$ which depends  analytically on $s$ and has $l$ attracting cycles with multipliers $(1+s_i)t_i^0$, $1 \le i \le l$, respectively. Note that the critical points of $g_s$ is exactly those of $f$, hence the critical points of $f_s$ are $\phi_s$-images of those of $f$ and thus depend holomorphically on $s$.

Let $$W = \{(1+s_1)t_1^0, \cdots, (1+s_l)t_l^0\:|\: \sum_{1\le i \le l}\big{|}s_i| < \epsilon\}.$$ Then $W$ is an open neighborhood of $(t_1^0, \cdots, t_l^0)$. For $(t_1, \cdots, t_l) \in W$, let $s =  ((t_1-t_1^0)/t_1^0, \cdots, (t_l-t_l^0)/t_l^0)$. Let $c_1, \cdots, c_{d-2}$ be the critical points of $f_s$.  Since all $c_i, 1 \le i \le d-2$, depend holomorphically on $s$ and are thus holomorphic functions in $W$.    This defines a holomorphic function $\Psi: W \to U$ satisfying  $\Phi \circ \Psi = \hbox{id}$. The proof of  Key-Lemma 3 is completed.

\subsection{Proof of  Key-Lemma 4} Let $0 \le l \le d-3$ and suppose $f$ has $l+1$ periodic attracting cycles with non-zero multipliers $t_1^0, \cdots, t_{l+1}^0$ respectively.  For a small  $\epsilon > 0$,  using the same qc surgery  as in the above construction of $\Psi$, but just in the immediate basin of the $(l+1)$-th attracting cycle,   we get a holomorphic family of polynomials $f_{t_{l+1}}$  with $|t_{l+1} - t_{l+1}^0| < \epsilon$, such that $f_{t_{l+1}}$  preserves all the orbit relations and the multipliers of the first $l$ attracting cycles.  Let $c_i$, $1 \le i \le d-1$ be the critical points of $f_{t_{l+1}}$ with $c_{d-1}= 1$.  For $1 \le i \le l+1$,  let $p_i$ be the period of the $i$-th attracting cycle and  $x_i$  be one of the points in the $i$-th cycle.    Then all $c_i$, $1 \le i \le d-2$, and $x_i, 1 \le i \le l+1$,  are holomorphic functions of $t_{l+1}$ for $|t_{l+1} - t_{l+1}^0| < \epsilon$.  Moreover, $c_1, \cdots, c_{d-2}, x_1, \cdots, x_{l+1}, t_{l+1}$, satisfy $d-l+1$ polynomial equations
\begin{equation}\label{k4}
\begin{cases}
f^{p_i}(x_i) = x_i, & \: 1 \le i \le l+1, \\
Df^{p_i}(x_i) = t_i^0, & \: 1 \le i \le l, \\
Df^{p_{l+1}}(x_{l+1}) = t_{l+1}, & \:  \\
f^{k_i}(1) = c_i,  & \: l+2 \le i \le d-2.
\end{cases}
\end{equation} As in (\ref{imp-f-1}), we have a holomorphic function $\lambda_{l+1}$ of $d-2$ variables such that
$$
t_{l+1} = \lambda_{l+1}(c_1(t_{l+1}), \cdots, c_{d-2}(t_{l+1})), \:\:|t_{l+1} - t_{l+1}^0| < \epsilon.
$$
It follows that there is some $1 \le i \le d-2$ with  $c_i'(t^0_{l+1}) \ne 0$.  Without loss of generality, let us assume that  $c_{d-2}'(t^0_{l+1}) \ne 0$. Thus  $c_{d-2}$ is a univalent function of $t_{l+1}$  in a small neighborhood of $t_{l+1}^0$.
  Thus  in a small neighborhood of $c_{d-2}^0$,
  all $c_i$, $1 \le i \le d-3$,  $x_i$, $1 \le i \le l+1$, and $t_{l+1}$ are holomorphic functions of $c_{d-2}$, and satisfy the  polynomial equations in (\ref{k4}).  We can directly check the  conditions in  Lemma~\ref{fundamental-E-a}. In particular the second condition is guaranteed by the rigidity assertion of Theorem~\ref{Thurston-Siegel-Ch}. Thus  by Lemma~\ref{fundamental-E-a}
all $c_i$, $1 \le i \le d-3$, are functions of $c_{d-2}$ determined by some irreducible polynomial equation $P(c_i, c_{d-2}) = 0$.  This completes the proof of  Key-Lemma 4.

\section{topological characterization of the maps in $\Sigma_{\alpha}^d$}

The purpose of this section is to prove Theorem~\ref{Thurston-Siegel-Ch}.  Key-Lemma-2 then follows by a little more effort.
Theorem~\ref{Thurston-Siegel-Ch} can be viewed as an extension of Thurston's characterization theorem for post-critically finite rational maps.  It may have independent interest to  consider the topological characerization of  more general rational maps, for instance, rational maps with Jordan Siegel disks for which the critical orbits are either eventually periodic, or attracted to some attracting or even parabolic  cycles, or intersect the closures of the Siegel disks. But for the purpose of this work, we restrict our attention  only to the maps in $\Sigma_{\alpha}^d$ and this will suffice for our purpose.

\subsection{Some Preliminaries} For readers' convenience, let us introduce some background knowledge for Thurston's characterization theorem for  rational maps,  especially,  the extension of this theorem to sub-hyperbolic rational maps.  The readers may refer to  \cite{CJ}, \cite{CT}, \cite{DH}, \cite{Pi} and \cite{ZJ} for more details in this aspect.

Let $F: \widehat{\Bbb C} \to \widehat{\Bbb C}$ be a finitely branched covering map which preserves the orientation. Let
$$
\Omega_F = \{z\in \widehat{\Bbb C}\:|\: \deg_{z}(F) \ge 2\}
$$ and
$$
P_F = \overline{\bigcup_{k\ge 1}F^k(\Omega_F)}
$$ be the critical set and the post-critical set of $F$, respectively.

We say $F$ is ${\emph geometrically finite}$ if $P_f$ is an infinite set but the accumulation set of $P_F$ is a finite set. It is easy to check that each accumulation point of $P_F$ is a period point of $F$.  We say a geometrically finite branched covering map $F$ is ${\emph sub}$-${\emph hyperbolic}$  ${\emph semi}$-${\emph rational}$ if  each periodic cycle $\mathcal{O}$ in the accumulation set of $P_F$ is  $\emph{holomorphically}$ $\emph{attrracting}$, that is, there is an open neighborhood $U$ of $\mathcal{O}$ such that $F|U$ is holomorphic, and moreover, $|DF^p(x)| < 1$ for each $x \in \mathcal{O}$ where $p \ge 1$ is the period of $\mathcal{O}$.
 Let $p$ be the period of $\mathcal{O}$.  According to Lemma 2.1 of \cite{ZJ}, for each $x \in \mathcal{O}$, one can take a Jordan disk $D_x$ containing $x$,   and an annulus $A_x$ surrounding $D_x$,  such that (1) the inner boundary component of $A_x$  is equal to $\partial D_x$, (2) $F$ maps $\overline{D_x}\cup A_x$ holomorphically into $D_{F(x)}$  and (3) all $D_x$ are disjoint and $\bigcup_{x\in\mathcal{O}}A_x \cap P_F = \emptyset$. We call each $D_x$  a holomorphic disk, and $A_x$ the protective annulus of $D_x$. We may further assume that $F^p: D_x \to D_x$ is holomorphically conjugate either to $z \mapsto \lambda z$ for some $0< |\lambda|  <1$ or to $z \mapsto z^m$ for some $m \ge 2$.

Let $F$ be a sub-hyperbolic semi-rational branched covering map. Let $\gamma \subset \widehat{\Bbb C} \setminus P_F$ be a simple closed curve.
We say $\gamma$ is ${\emph non}$-${\emph peripheral}$ if each component of $\widehat{\Bbb C} \setminus \gamma$ contains at least two points in $P_F$. A ${\emph multi}$-${\emph curve}$  $\Gamma$  is
 a finite family of non-peripheral curves which are disjoint and non-homotopic to each other. We say $\Gamma$ is ${\emph stable}$ if for each $\gamma_i \in \Gamma$, each
 non-peripheral component of $F^{-1}(\gamma)$ is homotopic to some $\gamma_j$ in $\Gamma$.

Suppose $\Gamma = \{\gamma_1, \cdots, \gamma_n\}$ is a stable multi-curve. For $1 \le i, j \le n$, let $\gamma_{i, j, \alpha}$ denote all the non-peripheral components of $f^{-1}(\gamma_i)$ which are homotopic to $\gamma_j$. Let $d_{i,j,\alpha}$ be the covering degree of
$$
F: \gamma_{i,j,\alpha} \to \gamma_i.
$$
Let
$$
a_{i,j}  = \sum_{\alpha} \frac{1}{d_{i,j,\alpha}}.
$$
The matrix $A = (a_{i,j})$ is called Thurston linear transformation matrix.  It is non-negative and thus has a maximal positive eigenvalue $\lambda > 0$. If $\lambda \ge  1$ we call $\Gamma$ a ${\emph Thurston\: obstruction}$.

\begin{defi}\label{dfdf}
Two sub-hyperbolic  semi-rational  maps $F$ and $G$ are called CLH-equivalent (combinatorially and locally holomorphically equivalent) if there exist a pair of homeomorphisms of the sphere $\phi, \psi: \widehat{\Bbb C} \to \widehat{\Bbb C}$ such that
\begin{itemize}
\item[1.] $\phi \circ F = G \circ \psi$,
\item[2.] for each holomorphic disk $D_i$, $\phi|D_i = \psi|D_i$ and both of them are holomorphic,
    \item[3.] $\phi$ is isotopic to $\phi$ rel $P_F \cup \cup_{1 \le i \le l} \overline{D_i}$ where $D_i$, $1 \le i \le l$, are all the holomorphic disks of $F$.
\end{itemize}
\end{defi}
The following  is an extension of Thurston's characterization theorem for post-critically finite rational maps to sub-hyperbolic rational maps.
\begin{thm}[\cite{CT} \& \cite{ZJ}]\label{CT-ZJ}
Let $F$ be a  semi-rational branched covering map. Then $F$ is CLH-equivalent to a sub-hyperbolic rational map $G$  if and only if $F$ has no Thurston obstructions.
\end{thm}

Two different proofs of Theorem~\ref{CT-ZJ} are provided in \cite{CT} and \cite{ZJ}, respectively.
Let us just sketch the idea of the proof in \cite{ZJ} as follows,  which will be helpful in the sequel discussion.    Let $D_i$ denote all the holomorphic disks of $F$. Let
$$
D = \bigcup_i D_i \hbox{  and  } P_1 = P_F\setminus D.
$$
Let $T_F$ denote the Teichmuller space modeled on $(\widehat{\Bbb C} \setminus (P_1 \cup \overline{D}), P_1 \cup \partial D)$.  Then  $F$ induces an analytic map $\sigma_F: T_F \to T_F$.  It turns out that the existence of the desired  rational map $G$ is equivalent to the existence of a fixed point of $\sigma_F$.  The proof in \cite{ZJ}  is divided into two steps. In the first step it was proved  that the
non-existence of Thurston obstructions implies certain bounded geometry condition.   In the second step it was proved that the bounded geometry condition implies the existence of a fixed point of $\sigma_F$.

Now let us introduce some notations.
Let $\mu_0$ be the standard complex structure. Let $\mu_k$ be the complex structure which is the pull back of $\mu_0$ by $F^k$. Let $\phi_k:\widehat{\Bbb C} \to \widehat{\Bbb C}$ be the quasiconformal homeomorphism which solves the Beltrami equation given by $\mu_k$ and fixes $0$, $1$ and $\infty$. It is important to note that for each holomorphic disk $D_i$ and its protective annulus $A_i$, $\mu_k = 0$ on $\overline{D_i}\cup A_i$, and thus $\phi_k$ is conformal in  $\overline{D_i}\cup A_i$.  Consider the hyperbolic Riemann surface
$$
X_k = \widehat{\Bbb C} \setminus \phi_k(P_1 \cup \overline{D}).
$$
For any non-peripheral curve $\gamma \subset \widehat{\Bbb C} \setminus (P_1 \cup  \overline{D})$,   there  is a unique simple closed geodesic $\eta$ in $X_k$ which is homotopic to $\phi_k(\gamma)$. Let $[\mu_k]$ denote the Teichm\"{u}ller class of $\mu_k$ in $T_F$ and  $l_{[\mu_k]}(\gamma)$ the length of $\eta$ with respect to the hyperbolic metric in $X_k$. We say $\gamma$ is a $[\mu_k]$-geodesic if $\phi_k(\gamma)$ is a geodesic in $X_k$.
   \begin{defi}{\rm
   We say a sub-hyperbolic semi-rational map $F$ has bounded geometry if there exists a $\delta > 0$ such that for every non-peripheral curve $\gamma \subset \widehat{\Bbb C} \setminus (P_1 \cup  \overline{D})$ and all $k \ge 0$, one has $l_{[\mu_k]}(\gamma) > \delta$.}
   \end{defi}

\begin{thm}[cf. \cite{CJ}, \cite{Pi}]\label{cn}
Let $F$ be a sub-hyperbolic semi-rational map. Then \begin{itemize} \item[1.] there is a $\delta > 0$ such that for any non-peripheral curve $\gamma \subset \widehat{\Bbb C} \setminus (P_1 \cup  \overline{D})$, either $l_{[\mu_k]}(\gamma) > \delta$ for all $k \ge 0$ or $l_{[\mu_k]}(\gamma) \to 0$ as $k \to 0$. \item[2.] The multi-curve $\Gamma$ which represents the homotopy classes of all $\gamma$ such that  $l_{[\mu_k]}(\gamma) \to 0$ as $k \to \infty$ is a Thurston obstruction. Such Thurston obstruction is called the canonical Thurston obstruction of $F$.  \item[3.]
 $F$ has a Thurston obstruction if and only if $F$ has a canonical Thurston obstruction. \end{itemize}

\end{thm}

 \subsection{Proof of Theorem~\ref{Thurston-Siegel-Ch}}

Let $f \in \mathcal{T}_\alpha^d$.  If $f$ has a Thurston obstruction in the exterior of $\Delta$, by a result of McMullen (cf. Appendix B, \cite{McM1}), $f$ can not be CLH-equivalent to any $g \in \Sigma_\alpha^d$. In the following let us assume that  $f$ has no Thurston obstructions in the exterior of $\Delta$. Let us first prove that  there is a $g \in \Sigma_\alpha^d$ such that $f$ is CLH-equivalent to$g$. After that, we prove  the uniqueness of $g$  up to a linear conjugation.

For $w \in \widehat{\Bbb C}$, let $w^*$ denote the symmetric image of $w$ about $\Bbb T$. Define
$$ F(z) =
\begin{cases}
f(z) & \hbox{  for  } |z| \ge 1 \\
[f(z^*)]^{*} & \hbox{ for } |z| < 1.
\end{cases}
$$
Note that $F|\Bbb T$ is the rigid rotation $z \mapsto e^{2 \pi i \alpha}z$ and
 $F: \widehat{\Bbb C} \to \widehat{\Bbb C}$ is a branched covering map of degree $2d-1$ and is symmetric with respect to $\Bbb T$.
 Let $\alpha_n = p_n/q_n$ be the sequence of convergents of $\alpha$. Then $\alpha_n \to \alpha$ as $n \to \infty$.
By definition,  $1 \in \Omega_f$.  Thus $1 \in \Omega_F$.
Let $\epsilon > 0$ and $$H = \{z\:|\:(1 + \epsilon)^{-1}  < |z| < 1 + \epsilon\}.$$  By taking $\epsilon > 0$ small enough we may assume that  $$(H - \Bbb T) \cap (\Omega_F \cup P_F) =
\emptyset.$$
We can perturb $F$ in $H$ to get a sequence of  sub-hyperbolic semi-rational maps $F_n$   such that
\begin{itemize}
\item[1.] $1 \in \Omega_{F_n}$  and  $F_n(z^*) = [F_n(z)]^*$,
\item[2.] $F_{n}(z) = F(z)$ for all $z \notin H$,
\item[3.] $F_n(z) = e^{2 \pi i \alpha_n}z$ for all $z \in \Bbb T$,

\item[4.] $(H - \Bbb T) \cap (\Omega_{F_n} \cup P_{F_n}) = \emptyset$,

\item[5.] $P_{F_n}\cap \Bbb T =  \mathcal{O}_n$ is a periodic orbit of $F_n$ with period $q_n$,
\item[6.]  $F_n \to F$ uniformly with respect to the spherical metric.
 \item[7.]
 the degree of $F_n$ and the number of the critical points of $F_n$ in $\Bbb T$, are respectively the same as those of $F$, \item[8.]  for each critical point $c$ of $F$, there is corresponding   critical point of $F_n$, say $c_n$, such that $c_n \to c$ as $n \to \infty$, and the local degree of $F_n$ at $c_n$ is the same as that of $F$ at $c$. \end{itemize}

\begin{lem}\label{no-th-ob}
 $F_{n}$ has no Thurston obstructions for all $n$ large enough.
\end{lem}

\begin{proof} Our proof is by contradiction.
Suppose $F_n$ has a Thurston obstruction for some large $n$.
By Theorem~\ref{cn}  $F_n$ would have a canonical Thurston obstruction $\Gamma$. That is,  for any $\delta > 0$, there is a $k$ such that $\Gamma$ is the set of all $[\mu_k]$-geodesic with $l_{[\mu_k]}(\gamma) < \delta$.
Since $F_n$ is symmetric about $\Bbb T$ and two short simple closed geodesics are disjoint (cf. Corollary 6.6 of \cite{DH}), we may assume the following:    For any $\gamma \in \Gamma$,  $\gamma \cap \Bbb T = \emptyset$ implies   $\gamma^* \in \Gamma$; and  $\gamma \cap \Bbb T \ne \emptyset$ implies $\gamma = \gamma^*$.  Here $\gamma^*$ denote the symmetric image of $\gamma$ about $\Bbb T$.

We first claim that all curves in $\Gamma$ intersect $\Bbb T$.  The proof is by contradiction.
Let  $\Gamma' =\{\gamma \in  \Gamma\:|\:\gamma \cap \Bbb T =  \emptyset\}$. Let us assume  that  $\Gamma' \ne \emptyset$.  Let $\Gamma'' = \Gamma\setminus \Gamma'$. Then any  $\xi \in \Gamma''$ intersects $\Bbb T$ and is thus symmetric about $\Bbb T$. Thus
both the two components of $\widehat{\Bbb C}\setminus \xi$ contain at least two points in $P_{F_n}$ which  either are   symmetric about $\Bbb T$, or belong to $\Bbb T$. This implies that   $\xi$ can not be homotopic  to any  curve which is disjoint with $\Bbb T$.  Now Let $\gamma \in \Gamma'$.  Since $F_n$ maps $\Bbb T$ to $\Bbb T$, any non-peripheral component of  $F_n^{-1}(\gamma)$ must be disjoint with $\Bbb T$ and is homotopic to some element $\eta \in \Gamma$. Note that we have just proved that $\eta \notin \Gamma''$. It follows that $\eta \in \Gamma'$. This implies that  $\Gamma'$ is  stable. By Theorem~\ref{cn} $l_n(\gamma) \to 0$ for all $\gamma \in \Gamma'$. Thus  $\Gamma'$ is also a Thurston obstruction.  Now let $\Gamma_1 = \{\gamma \in \Gamma'\:|\:\gamma \in \Bbb C \setminus \overline{\Delta}\}$ and $\Gamma_2 = \{\gamma \in \Gamma_1\:|\: \gamma \subset \Delta\}$. Then $\Gamma' = \Gamma_1 \cup \Gamma_2$.   Let us show that
 both $\Gamma_1$ and $\Gamma_2$ are stable. By symmetry it suffices to prove that $\Gamma_2$ is stable. Suppose $\Gamma_2$ is not stable. Then $\Gamma_2$ would contains a  $\gamma$ such that one of the non-peripheral components of $F_n^{-1}(\gamma)$, say $\eta$, is homotopic to some $\xi$ in $\Gamma_1$. Thus $\eta$ encloses at least two points in $P_{F_n} \setminus \overline{\Delta}$ which are mapped to the inside of $\gamma$. In particular, the two points are mapped to the interior of $\Delta$. By the construction of $F$ and $F_n$, we have $P_{F_n} \setminus \overline{\Delta} = P_{F} \setminus \overline{\Delta} = P_{f} \setminus \overline{\Delta}$  and $F_n(P_{F_n} \setminus \overline{\Delta}) = f(P_{f} \setminus \overline{\Delta}) \subset P_f \setminus \overline{\Delta}$.   This is a contradiction.  It follows that  $\Gamma_2$ is stable. By symmetry  $\Gamma_1$ is stable also. Since $\Gamma' = \Gamma_1 \cup \Gamma_2$ is a Thurston obstruction, by symmetry and the fact that both $\Gamma_1$ and $\Gamma_2$ are stable, it follows that both $\Gamma_1$ and $\Gamma_2$ are Thurston obstructions of $F_n$. Since $F_n$ behaves like $f$ in the exterior of $\Delta$, it  follows that $\Gamma_1$ is a Thurston obstruction of $f$ in the exterior of $\Delta$. This contradicts the assumption.  So  $\Gamma' = \emptyset$ and all the curves in $\Gamma$ intersect $\Bbb T$.

  Now  for any $k \ge 1$, $G_{n,k} = \phi_{k-1} \circ F_{n} \circ \phi_{k}^{-1}$  is a rational map of degree $2d-1$.  By symmetry $G_{n,k}$ is a Blaschke product.
Note that the restriction of each $G_{n,k}$  to $\Bbb T$ is a circle homeomorphism and all the zeros of $G_{n,k}$, except the origin, are all contained in the exterior of the unit disk. Thus $G_{n,k} \in \mathcal{H}_d$ where $\mathcal{H}_d$ is the Herman family defined in (\ref{e-H}).  By a lemma of Herman (cf. $\S15$ of \cite{He}),
 the sequence $\{G_{n,k}\}$ is equicontinuous in an annular neighborhood of $\Bbb T$.

We now claim that every $\gamma$ in $\Gamma$  encloses at least two points in $P_1$, say $z$ and $z^*$,  or two  holomorphic disks, say $D_i$ and $D_i^*$, which are symmetric about $\Bbb T$.  Assume that the claim were not true.  Then by symmetry,  $\gamma$  would  enclose two adjacent  points in $\mathcal{O}_n$. To get a contradiction,  it suffices to show  the existence of a $d_0 > 0$  independent of $k$ such that for any two adjacent  points $x, y  \in \mathcal{O}_n$ and any $k \ge 1$, one has
 \begin{equation}\label{dis-p}
 {\rm dist}(\phi_k(x), \phi_k(y)) > d_0,
   \end{equation}
   where ${\rm dist}(\cdot, \cdot)$ denotes the distance with respect to the Euclidean metric. This is because if $\gamma$ encloses two points $x$ and $y$ in $\mathcal{O}_n$, since $\phi_k$ is a plane homeomorphism, $\phi_k(\gamma)$ will enclose $\phi_k(x)$ and $\phi_k(y)$.  But by (\ref{dis-p}) it follows that $l_{[\mu_k]}(\gamma)$  have a positive lower bound for all $k \ge 1$.  This contradicts the assumption that $l_{[\mu_k]}(\gamma) \to 0$ as $k \to \infty$.

Note that $d_{T_{F_n}}([\phi_k], [\phi_{k-1}]) \le d_{T_{F_n}}([\phi_1], [\phi_0])$ for all $k \ge 1$ (cf. Corollary 4.1 of \cite{ZJ}).    Let $\delta_0  = d_{T_{F_n}}([\phi_1], [\phi_0])$. Then  by Proposition 7.2 of \cite{DH},  we have
  \begin{equation}\label{geo-in}e^{-2\delta_0} \cdot l_{[\mu_{k-1}]}(\gamma) \le  l_{[\mu_k]}(\gamma) \le e^{2\delta_0} \cdot l_{[\mu_{k-1}]}(\gamma).
  \end{equation}  Because all $\phi_k$ fix $0$, $1$ and $\infty$,  the above inequality implies the existence of  an $\epsilon > 0$  such that
  for any two integers $1 \le k' \le k'' \le k'+q_n$,  and any two  adjacent points $a$ and $b$ in  $\mathcal{O}_n$,  if ${\rm dist}(\phi_{k'}(a), \phi_{k'}(b)) < \epsilon$, then ${\rm dist}(\phi_{k''}(a), \phi_{k''}(b)) < 2\pi/q_n$. In fact,  if ${\rm dist}(\phi_{k'}(a), \phi_{k'}(b))$ is small, then there exists a non-peripheral curve $\gamma$ in $\widehat{\Bbb C}\setminus (P_1 \cup \overline{D})$ which encloses $a$ and $b$ such that $l_{[\mu_{k'}]}(\gamma)$ is small. In the post-critically finite case, the existence of such $\gamma$ is obvious. In the sub-hyperbolic semi-rational case,  note that  each holomorphic disk $D_i$ has  a protective annulus $A_i$ surrounding  it such that for all $k \ge 1$,
  $\phi_k$ is conformal  in  $\overline{D_i} \cup A_i$. Thus by Koebe's distortion theorem, we have \begin{equation}\label{ess-in}{\rm diam}(\phi_k(D_i)) \preceq
  {\rm dist}(\phi_k(D_i), \Bbb T).\end{equation} The existence of such $\gamma$ also follows.    By (\ref{geo-in}) and $k' \le k'' < k'+q_n$ we get $l_{[\mu_{k''}]}(\gamma) < e^{2q_n \delta_0} \cdot l_{[\mu_{k'}]}(\gamma)$ and thus $l_{[\mu_{k''}]}(\gamma)$ is small also.  So $d(\phi_{k''}(a), \phi_{k''}(b))$  must be small.

  Now let us go back to the proof of the existence of $d_0$ so that (\ref{dis-p}) holds.  For the above $\epsilon > 0$, since  $\{G_{n,k}\}$ is equicontinuous,   we  have a $d_0 > 0$   such that for any $x$ and $y$ in $\mathcal{O}_n$,  any $N > q_n$ and any $1\le l \le q_n$, if ${\rm dist}(\phi_N(x), \phi_N(y)) \le d_0$, then
   $$
   {\rm dist}(\phi_{N-l}(F_n^l(x), \phi_{N-l}(F_n^l(y))) = {\rm dist}(B(\phi_N(x)), B(\phi_N(y)) < \epsilon
   $$
 where $B = G_{n,N-l+1}\circ \cdots \circ G_{n,N}$. Let us show that the $d_0$ is the desired constant.
     Let $N > q_n$ be an arbitrary integer and $x, y$ be any two adjacent points in $\mathcal{O}_n$. Suppose  ${\rm dist}(\phi_N(x), \phi_N(y)) \le d_0$.   Since  $\mathcal{O}_n$ is a periodic orbit of  period $q_n$,
   there is some $1 \le l \le q_n$ such that ${\rm dist}(\phi_N(F_n^l(x)), \phi_N(F_n^l(y))) \ge 2 \pi/q_n$. But from the choice of $d_0$, we have $ {\rm dist}(\phi_{N-l}(F_n^l(x), \phi_{N-l}(F_n^l(y)))< \epsilon$. From the choice of $\epsilon$, we have   ${\rm dist}(\phi_N(F_n^l(x)), \phi_N(F_n^l(y))) < 2 \pi/q_n$. This is a contradiction.
 This proves (\ref{dis-p}) and thus completes the proof of the claim.

Let $N_1$ be the number of the points in $P_1$ and $N_2$ be the number of the holomorphic disks $D_i$. From the claim we just proved, each $\gamma \in \Gamma$ either encloses two points in $P_1$ or two holomorphic disks, it follows that the number of the curves in $\Gamma$ is less than $N_0 = N_1 + N_2$. Note that $N_0$  is independent of $n$. In the following we will show $\Gamma$ contains an arbitrarily large number of curves provided that $n$ is large enough. This is a contradiction and completes the proof of Lemma~\ref{no-th-ob}.

To see this, for each $x \in \mathcal{O}_n$, let $R_x$ and $L_x$ denote the two interval components of $\Bbb T \setminus \mathcal{O}_n$ such that $\partial R_x \cap \partial L_x = \{x\}$. Let $S_x = R_x \cup L_x$. Since $\theta$ is irrational, by taking $n$ large enough, we may assume that for any $x \in \mathcal{O}_n$, the intervals $F_n^i(S_x)$, $0 \le i \le N_0$, are disjoint with each other.  Now take $\gamma \in \Gamma$. Then $\gamma$ encloses two points $z$ and $z^*$  in $P_{F_n}\setminus \mathcal{O}_n$ which are symmetric about $\Bbb T$.  Since $l_{[\mu_k]}(\gamma) \to 0$ we have ${\rm dist}(\phi_k(z), \Bbb T) \to 0$ and ${\rm dist}(\phi_k(z), \Bbb T) \to 0$. Take $N$ large. Then there exists an $a \in \mathcal{O}_n$  such that
\begin{equation}\label{bi-1}
{\rm dist}( \phi_N(z), \phi_N(S_a)) \asymp {\rm dist}(\phi_N(z^*), \phi_N(S_a)) \ll {\rm diam}(\phi_N(S_a)) \asymp 1.
\end{equation}
Note that ${\rm diam}(\phi_N(S_a)) \asymp 1$ is implied by (\ref{dis-p}). For $1 \le l \le N_0$,
we have
$$
{\rm dist}(\phi_{N-l}(F_n^l(z)), \phi_{N-l}(F_n^l(S_a)))  =  {\rm dist}(B(\phi_N(z)), B(\phi_N(S_a)))
$$
and
$$
{\rm dist}(\phi_{N-l}(F_n^l(z^*)), \phi_{N-l}(F_n^l(S_a)))  =  {\rm dist}(B(\phi_N(z^*)), B(\phi_N(S_a)))
$$ where $B = G_{N-l+1}\circ \cdots \circ G_{n,N}$.
Since $\{G_{n,k}\}$ is equicontinuous, by taking $N$ large, from (\ref{bi-1}) we have
$$
 {\rm dist}(\phi_{N-l}(F_n^l(z)), \phi_{N-l}(F_n^l(S_a)))  \asymp {\rm dist}(\phi_{N-l}(F_n^l(z^*)), \phi_{N-l}(F_n^l(S_a))) $$$$ \asymp {\rm dist}(\phi_{N-l}(F_n^l(z)), \phi_{N-l}(F_n^l(z^*)))   \ll {\rm diam}(\phi_{N-l}(F_n^l(S_a))) \asymp 1.
$$

Again ${\rm diam}(\phi_{N-l}(F_n^l(S_a))) \asymp 1$ is implied by (\ref{dis-p}). This implies that for  $0 \le l \le N_0$, there is a short simple closed geodesic $\eta_l$ in $X_{N-l} = \widehat{\Bbb C} \setminus \phi_{N-l}(P_1 \cup \overline{D})$ such that $\eta_l$ encloses $\phi_{N-l}(F_n^l(z))$ and $\phi_{N-l}(F_n^l(z^*))$ and intersects $\phi_{N-l}(F_n^l(S_a))$ (Note that the existence of such $\eta_l$ relies on (\ref{ess-in})).  The hyperbolic length of $\eta_l$ in $X_{N-l}$ can be arbitrarily small provided that $N$ is large enough. Thus there is a $\gamma_l \in \Gamma$ which is homotopic to $\phi_{N-l}^{-1}(\eta_l)$ in $\widehat{\Bbb C}\setminus (P_1 \cup \overline{D})$. Then $\gamma_l$ must intersect $F_n^l(S_a)$. Since all $F_n^l(S_a)$, $0\le l \le N_0$,  are disjoint, it follows that all $\gamma_l$, $0 \le l \le N_0$, are distinct elements of $\Gamma$.  This implies that the number of the curves in $\Gamma$ is greater than $N_0$. This is a contradiction.
 The proof of Lemma~\ref{no-th-ob} is completed.
\end{proof}

By Lemma~\ref{no-th-ob} and Theorem~\ref{CT-ZJ}, it follows that for all $n$ large enough $F_n$ is CLH-equivalent to a rational map.  Let $\phi_{n}$ and $\psi_{n}$ be a pair of homeomorphisms of the sphere which fix $0$, $1$ and $\infty$ such that \begin{itemize}
\item[1.] $\phi_{n} \circ F_{n} \circ \psi_n^{-1}$  is a rational map, and
\item[2.] $\phi_{n}$ is isotopic to $\phi_{n}$ rel $P_{F_{n}}$, moreover, for each holomorphic disk $D_i$, $\phi_{n}|D_i = \psi_{n}|D_i$ and both are holomorphic.\end{itemize}
By symmetry of $F_n$, $\phi_n$ and $\psi_n$ can be taken such that both of them  are symmetric about $\Bbb T$. Let
$$
G_n = \phi_{n} \circ F_{n} \circ \psi_n^{-1}.
$$
Then $G_n$ is a Blaschke product. Since all the zeros of $G_n$, except the origin,  belong to the exterior of the unit disk, we have
\begin{equation}\label{form-p}
G_n(z) = \lambda_n z^d \prod_{i=1}^{d-1} \frac{z - p_{n,i}}{1 - \overline{p_{n,i}}z}
\end{equation} where $|\lambda_n|= 1$  and $|p_{n,i}| > 1$, $1 \le i \le d-1$.

\begin{lem}\label{compact-s}
There exist $0 < r(d) < \sigma(d) < 1 < \kappa(d) <  R(d) < \infty$ depending only on $d$ such that for any $z$ with $|z| > R(d)$,  $G_n^k(z) \to \infty$ as $k \to \infty$,  and for any $z$ with $|z| < r(d)$,  $G_n^k(z) \to 0$ as $k \to \infty$.  Moreover,  we have  $$\kappa(d) \le |p_{n,i}| \le R(d)$$ for all $n$ large enough.
\end{lem}
\begin{proof}
Note that all the finite poles of $G_n$ belong to the interior of $\Delta$ and except the origin,
all the zeros of $G_n$ belong to the exterior of $\Delta$.  By a lemma of Herman (cf. $\S15$ of \cite{He}),
all the poles, and thus all the zeros by symmetry,  are uniformly bounded away from $\Bbb T$.  This implies the existence of $\kappa(d) > 1$ such that $\kappa(d) \le |p_{n,i}|$ holds for all $1 \le i \le d-1$ and all $n$ large enough. So to prove the lemma, it suffices to show  the existence of $R(d) > 1$ such that   $G_n^k(z) \to \infty$ for all $|z| > R(d)$. Then  the existence of the desired $r(d)$ follows by symmetry. It is also clear that for such $R(d)$ we must have $|p_{n,i}| \le R(d)$  for all $1 \le i \le d-1$ and all $n$ large enough.

Now let  $$\epsilon = (-1)^{d-1}\lambda_n \prod_{i=1}^{d-1}\frac{1}{\overline{p_{n,i}}}.$$ Then near infinity $G_n(z) = \epsilon z^d + o(z^d)$. Take $\eta$ such that $\eta^{d-1} = \epsilon$. Consider the map
 $$
 \tilde{G}_n(z) = \eta \cdot G_{n}(\frac{z}{\eta}).
 $$ Then $\tilde{G}_n(z) = z^d + o(z^d)$ near infinity.  Let
$\Phi(z) = z(1 + o(1))$ be a holomorphic map in a neighborhood of infinity  which conjugates $\tilde{G}_n$  to $z\mapsto z^d$. Since all critical orbits of $G_n$, and thus of $\tilde{G}_n$, are bounded, it follows that
$\Phi$ can be extended to a Riemann isomorphism   from the immediate attracting basin of infinity for $\tilde{G}_n$  to the exterior of the unit disk.   By Koebe's $1/4$-theorem, the immediate attracting basin of $\tilde{G}_n$ at infinity contains the exterior of $\{z\:|\: |z| < 4\}$.  Go back to $G_n$,   it follows that the immediate attracting basin of ${G}_n$ at infinity contains the exterior of $\{z\:|\: |z| < 4/|\eta|\}$, and in particular,
\begin{equation}\label{zero-in}
|p_{n,i}| \le \frac{4}{|\eta|},  \:\: 1 \le i \le d-1.
\end{equation} It follows that $|p_{n,i}|^{d-1} \le 4^{d-1} \prod_{j=1}^{d-1}|p_{n,j}|$ for any $1 \le i \le d-1$. Thus
$$
\min_{1 \le i \le d-1} |p_{n,i}| \le \max_{1 \le i \le d-1} |p_{n,i}| \le 4 \min_{1 \le i \le d-1} |p_{n,i}|.
$$
This implies that   that  $|p_{n,i}|, 1 \le i \le d-1, $ are all large provided that one of them  is large. From
 (\ref{form-p}), we see if all $|p_{n,i}|$, $1 \le i \le d-1$, are large enough, then  $G_n$ is holomorphic in $H = \{z\:|\: 1/2 < |z|< 2\}$,  and moreover, $G_n$ can be arbitrarily close to the linear map $z\mapsto a  z$ in $H$  with $|a| = 1$  provided that $|p_{n,i}|$ are all large enough. But this would imply $G_n$ has no critical point in $\Bbb T$. This is a contradiction. This implies the existence of some constant depending only on $d$, say $\beta(d) > 1$ such that  $|p_{n,i}| < \beta(d)$ for all $1 \le i \le d-1$. Since the immediate attracting basin of ${G}_n$ at infinity contains the exterior of $\{z\:|\: |z| < 4/|\eta|\}$, from the definition of $\eta$, we can take $R(d) = 4\beta(d)$.  The proof of Lemma~\ref{compact-s} is completed.

\end{proof}

By taking a subsequence, we may assume that $G_{n}$ converges to $G$ uniformly with respect to the spherical metric. Then $G|\Bbb T$ is a critical circle homeomorphism with rotation number $\alpha$. Since $\alpha$ is of bounded type, by Herman-Swiatek's theorem\cite{P2}, there is a qs circle homeomorphism $h$ which conjugates $G|\Bbb T$ to the rigid rotation $R_{\alpha}: z \to e^{2 \pi i \alpha} z$.
Let $\phi_{n}, \psi_{n}: \widehat{\Bbb C} \to \widehat{\Bbb C}$ be the pair of homeomorphisms  defined before (\ref{form-p}).

\begin{lem}\label{circle}
$\phi_{n}|\Bbb T \to h$ and $\psi_{n}|\Bbb T \to h$ uniformly.
\end{lem}
\begin{proof}
To simplify the notations, let us write $\phi_n|\Bbb T$ and $\psi_n|\Bbb T$ as $\phi_n$ and $\psi_n$ respectively.  We claim that $\phi_n$ and $\psi_n$ converge uniformly. Let us first prove Lemma~\ref{circle} by assuming the claim.  Let $\phi$ and $\psi$ be the limit maps respectively.
 Since the convergence is uniform, it follows that both $\phi$ and $\psi$ are circle homeomorphisms.  Since $\phi_n|P_{F_n} =  \psi_n|P_{F_n}$, for any $k \ge 0$ we have
  $\phi_n (F_n^k(1)) = \psi_n (F_n^k(1)) = G_n^k(1)$.  Since $F_n \to F$ and $G_n \to G$ uniformly on $\Bbb T$, let $n \to \infty$ we get $\phi(F^k(1)) = \psi(F^k(1)) = G^k(1)$ for all $k \ge 0$. This implies that $\phi$ and $\psi$ coincide on $\{F^k(1)= e^{2 \pi i k \alpha}\}_{k=0}^\infty$, which is a dense subset of $\Bbb T$. It follows that $\phi = \psi$. Since $\phi(1) = \psi(1) = h(1) = 1$,  it follows that $\phi = \psi = h$.

It remains to show that $\phi_n$ and $\psi_n$ converge uniformly.  Let us do this only for $\phi_n$ since the same argument works for $\psi_n$.
By construction, the point $1$ is a critical point for $F$, $G$, $F_{n}$ and $G_{n}$.
 For any $N \ge 1$,
 since  $F_{n} \to F$  and  $G_{n} \to G$ uniformly,   by taking $n$  large enough, we can assume that the orbit segments
 $$\mathcal{O}_N(G) = \{1, G^1(1), \cdots, G^N(1)\}\hbox{  and  } \mathcal{O}_N(G_{n}) =\{1, G_{n}^1(1), \cdots, G_{n}^N(1)\}$$ have the same order,
 and   $$\mathcal{O}_N(F) =\{1, F^1(1), \cdots, F^N(1)\}\hbox{  and  }\mathcal{O}_N(F_{n}) = \{1, F_{n}^1(1), \cdots, F_{n}^N(1)\}$$ have the same order. Since $\phi_{n}$ ($\psi_{n}$) is a circle homeomorphism and maps $\mathcal{O}_N(F_n)$ to  $\mathcal{O}_N(G_n)$ with the order being preserved, thus $\mathcal{O}_N(F_n)$ and  $\mathcal{O}_N(G_n)$ have the same order. All of these implies that for any fixed $N$, by taking $n$ large enough, all the four orbit segments,  $\mathcal{O}_N(F)$,  $\mathcal{O}_N(F_n)$,  $\mathcal{O}_N(G)$  and $\mathcal{O}_N(G_n)$, have the same order.

Let $\epsilon > 0$ be an arbitrary small number. Since $G|\Bbb T$ and $F|\Bbb T$ are circle homeomorphisms with irrational rotation number, by taking  $N = N(\epsilon)$ large enough we may assume that the length of
each component of  $\Bbb T \setminus \mathcal{O}_N(F)$ and $\Bbb T \setminus \mathcal{O}_N(G)$ is less than $\epsilon/4$.   For such $N$,  since $F_{n} \to F$ and $G_{n} \to G$ uniformly, there exists an $M_1 > 1$ such that for all $n \ge M_1$,  the length of
each component of  $\Bbb T \setminus \mathcal{O}_N(F_n)$ and $\Bbb T \setminus \mathcal{O}_N(G_n)$ is less than $\epsilon/3$.   For a component $I$ of $\Bbb T \setminus \mathcal{O}_N(F_n)$, we use $I^l$ and $I^r$ to denote the components of $\Bbb T \setminus \mathcal{O}_N(F_n)$, which are adjacent to $I$ from the left and right respectively. Similarly, for a component $J$ of $\Bbb T \setminus \mathcal{O}_N(G_n)$, we use $J^l$ and $J^r$ to denote the components of $\Bbb T \setminus \mathcal{O}_N(G_n)$, which are adjacent to $J$ from the left and right respectively.
For such $N$,  since $F_n \to F$ and $G_n \to G$ uniformly on $\Bbb T$,  there exists an $M > M_1$ such that for all $m, n \ge M$,
 the following holds: For any component $I$ of $\Bbb T \setminus \mathcal{O}_N(F_n)$ and any component $J$ of $\Bbb T \setminus \mathcal{O}_N(G_n)$, let $\tilde{I}$ and $\tilde{J}$ be the corresponding components of $\Bbb T \setminus \mathcal{O}_N(F_m)$ and $\Bbb T \setminus \mathcal{O}_N(G_m)$ respectively,  then \begin{equation}\label{ref-in} \overline{I} \subset \overline{\tilde{I}^l \cup \tilde{I} \cup \tilde{I}^r} \hbox{  and  } \overline{J} \subset \overline{\tilde{J}^l \cup \tilde{J} \cup\tilde{J}^r}.\end{equation}

For the above $M$, let $n, m > M$ be any two integers.  Let  $z \in \Bbb T$ be an arbitrary point and  $I$ be a component of $\Bbb T \setminus \mathcal{O}_N(F_n)$ such that $z \in \overline{I}$.  Since $\overline{I} \subset \overline{\tilde{I}^l \cup \tilde{I} \cup \tilde{I}^r}$ by (\ref{ref-in}), we  have $z \in \overline{\tilde{I}^l \cup \tilde{I} \cup \tilde{I}^r}$.   Let $J = \phi_n(I)$  and  $\tilde{J} = \phi_m(\tilde{I})$. It follows from (\ref{ref-in}) that
$$
\phi_n(z) \in \overline{J} \subset \overline{\tilde{J}^l \cup \tilde{J} \cup\tilde{J}^r}.
$$ and
$$
\phi_m(z) \in \phi_m(\overline{\tilde{I}^l \cup \tilde{I} \cup \tilde{I}^r}) = \overline{\tilde{J}^l \cup \tilde{J} \cup\tilde{J}^r}.
$$ Then $|\phi_n(z) - \phi_m(z)| \le |\tilde{J}^l|+| \tilde{J}|+ |\tilde{J}^r| < \epsilon$.  This proves that $\phi_n$ converges uniformly.  The proof of  Lemma~\ref{circle}  is completed.
\end{proof}

Recall that $D_i$  denote all the holomorphic disks of $F$ and $F_n$. Let
$$
P = P_{F_{n}} \setminus (\Bbb T \cup  \cup_i \overline{D_i}) = P_{F}\setminus (\Bbb T \cup  \cup_i \overline{D_i}) .
$$  Then $P$ is a finite set.
\begin{lem}\label{bounded-geometry}
There exists a $K > 1$ such that for all $n \ge 1$,
\begin{itemize}
\item[(1)] $\phi_{n}(P \cup \cup_i D_i) \subset \{z\:|\: 1/K < |z| < K\}$,
\item[(2)] for every  $D_i$, ${\rm diam}(\phi_{n}(D_i)) > 1/K$,

\item[(3)] for every point $z \in P$ and every  $D_i$, ${\rm dist}(\phi_{n}(z), \phi_{n}(D_i)) > 1/K$,

\item[(4)] for every  $D_i$, ${\rm dist}(\phi_{n}(D_i), \Bbb T) > 1/K$,

\item[(5)] for every two distinct  $D_i$ and $D_j$, ${\rm dist}(\phi_{n}(D_i), \phi_{n}(D_j)) > 1/K$,
\item[(6)] for every point $z \in P$, ${\rm dist}(\phi_{n}(z),\Bbb T) > 1/K$,
\item[(7)] for every two distinct points $z$ and $w$ in $P$, ${\rm dist}(\phi_{n}(z), \phi_{n}(w)) > 1/K$.
\end{itemize}
\end{lem}
\begin{proof}
By Lemma~\ref{compact-s} it follows that there is a $K > 1$ such that  $$\phi_{n}(P \cup \cup_i D_i) \subset \{z\:|\: |z| < K\}.$$ By symmetry we further
 have  $\phi_{n}(P \cup \cup_i D_i) \subset \{z\:|\: 1/K< |z| < K\}$. This proves (1).

 Now let us prove (2).  Let us fix a $D_i$.  By symmetry we may assume that $D_i$ belongs to the exterior of $\Delta$.  Since $F_n$ is obtained by perturbing $F$ in a thin annular neighborhood of $\Bbb T$, we may assume that  $F_n|D_i = F|D_i$.   Then $D_i$ contains a periodic point of $F_n$, say $x$. Suppose the period of $x$ is $p \ge 1$.  There are two cases.

 In the first case, $0< |DF_n^p(x)| = |DF^p(x)|< 1$.   Since $\phi_n$ and $\psi_n$ are holomorphic and identified with each other
 on all holomorphic disks, $F_n^p$ and $G_n^p$ are holomorphic conjugate on $D_i$.  In particular, $0< |DG_n^p(\phi_n(x))| = |DF^p(x)| < 1$. By Lemma~\ref{compact-s} and a compactness argument, there is an $r > 0$ independent of $n$ such that $DG_n^p \ne 0$ in the disk $B_r(\phi_n(x))$.  Now let $U_0 = \phi_n(D_i)$.   For $k \ge 0$, define $U_{k+1}$ to be the component of $G_n^{-p}(U_k)$ which contains $\phi_n(x)$. Then we have a sequence of increasing domains $U_0 \subset U_1 \subset \cdots$.  Let $l \ge 1$ be the least integer such that $U_l$ contains a critical point of $G_n^p$. Then ${\rm diam}(U_l) \ge r$.  If (2) were not true, then ${\rm diam}(U_0)$ could be arbitrarily small for some $n$. But by Lemma~\ref{compact-s} and a compactness argument, this would imply that $l$ could be arbitrarily large provided that ${\rm diam}(U_0)$  is small enough. Go back to $F_n$, this means that some critical point of $F_n$ goes through an arbitrarily long orbit  segment  before it enters $D_i$. But this is impossible. This is because $F_n$ is obtained by modifying $F$ in a thin annular neighborhood of $\Bbb T$ which does not intersect $P_{F}\cup \Omega_F$, so there is an $L \ge 1$ independent of $n$ such that for any critical point $c$ of $F_n$, either $F_n^k(c) \notin D_i$ for all $k \ge 0$, or there is some $0 \le k \le L$ such that $F_n^k(c) \in D_i$. This proves (2) in the first case.

In the second case, $|DF_n^p(x)| = |DF^p(x)| = 0$.  By taking a subsequence, we may assume that $\phi_n(x) \to z$. It follows that $DG^p(z) = 0$.  Now suppose $\phi_n(D_i) \subset B_r(z)$ for some $r > 0$. Then
 $$\frac{{\rm diam}(G_n^p(\phi_n(D_i)))}{{\rm diam}(\phi_n(D_i))} \le  \max_{w \in B_r(z)} |DG_n^p(w)|.$$
But on the other hand, since $G_n^p(\phi_n(D_i)) = \phi_n(F_n^p(D_i))= \phi_n(F^p(D_i))$ and since $\phi_n$ is holomorphic in $\overline{D_i} \cup A_i$, by Koebe's distortion theorem,
$$\frac{{\rm diam}(G_n^p(\phi_n(D_i)))}{{\rm diam}(\phi_n(D_i))} = \frac{{\rm diam}(\phi_n(F^p(D_i)))}{{\rm diam}(\phi_n(D_i))} \succeq \frac{{\rm diam}(F^p(D_i))}{{\rm diam}(D_i)}.$$   This implies that $$ \max_{w \in B_r(z)} |DG_n^p(w)| \succeq \frac{{\rm diam}(F^p(D_i))}{{\rm diam}(D_i)}.$$  Since  $G_n$ converges to $G$ uniformly in a small neighborhood of $z$ and $DG^p(z) = 0$, it follows that $\max_{w \in B_r(z)} |DG_n^p(w)| \to 0$ as $r\to 0$ and $n \to \infty$. This implies that $r > 0$ can not be too small and ${\rm diam}(\phi_n(D_i))$ has a positive lower bound as $n \to \infty$.  This completes the proof of (2).

 Since each  $\overline{D_i}$ has a protective annulus  $A_i$  which does not contain points in $P_{F_n}$ and
 on which $\phi_n$ is holomorphic, by Koebe's distortion theorem, the thickness of $\phi_n(A_i)$ is $\succeq {\rm diam}(\phi_n(D_i))$.  Since ${\rm diam}(\phi_n(D_i))>  1/K$ by (2), it follows that  $\phi_n(A_i)$ has definite thickness. This implies (3), (4) and (5).

By symmetry it suffices to prove (6) for $z \in P_{F_n}$  and $|z| > 1$.  Note that the forward orbits of
all critical points of $G$ in the exterior of the unit disk converge to some attracting or super-attracting cycles of $G$ and are thus  bounded away from the unit circle.
 Since for any $R > 1$, $G_n \to G$ uniformly in  the compact set $\{z\:|\: 1 \le |z| \le R\}$,
 it follows that  an attracting or super-attracting periodic cycle of $G_n$ in the exterior of the unit disk converges to the corresponding one for $G$, and the forward orbits of all critical points of $G_n$
 in the exterior of the unit disk converge uniformly, with respect to $n$,
  to these attracting or super-attracting cycles of $G_n$.   This implies that all the critical orbits of $G_n$ which belong to the exterior of the unit disk, are uniformly bounded away from $\Bbb T$.  This proves (6).

Suppose (7) were not true. By taking a subsequence we may assume that there exist $x, y \in P$ such that ${\rm dist}(\phi_n(x), \phi_n(y)) \to 0$ as $n \to \infty$. By definition of $\mathcal{T}_\alpha^d$,
there is an integer $k \ge 0$ such that $w = F_n^k(x)=F^k(x)$ either belongs to a holomorphic disk $D_i$ or belongs to a periodic cycle containing a critical point, which is  not a holomorphic attracting cycle.  In the first case, since  $\phi_n$ is holomorphic in  $\overline{D_i}\cup A_i$,
by (2) and Koebe's distortion theorem, there is a $\delta > 0$ such that for any $z \in P_{F_n}$ with $z \ne w$,  one has ${\rm dist}(\phi_n(w), \phi_n(z)) > \delta$  for all $n \ge 1$. In the second case,  $\phi_n(w)$ belongs to a super-attracting cycle of $G_n$ which does not attract any other critical orbit. By Lemma~\ref{compact-s} and a compactness argument, it follows that there is an $r > 0$ such that the immediate attracting basin of this cycle contains the disk $B_r(\phi_n(w))$ for all $n \ge 1$. This implies that for all $z \in P_{F_n} \setminus \Bbb T = P_F \setminus \Bbb T$ with $z \ne w$,
 ${\rm dist}( \phi_n(w),\phi_n(z)) > r$  for all $n \ge 1$.   Now let  $0\le l < k$  be the largest integer  such that there exists a $\xi \in P_{F_n}$ with $\xi \ne \zeta  = F_n^l(x) = F^l(x)$ and  $ {\rm dist}(\phi_n(\xi), \phi_n(\zeta)) \to 0$ as $n \to \infty$.  Note that $F(\xi) = F_n(\xi)$ and $F(\zeta) = F_n(\zeta)$. Since ${\rm dist}(\phi_n(F(\zeta)), \phi_n(F(\xi))) = {\rm dist}(G_n(\phi_n(\zeta)), G_n(\phi_n(\xi))) \to 0$,  By the maximal property of $l$ we must have $F(\xi) = F(\zeta)$.  By taking a subsequence we may assume that $\phi_n(\xi)$ and $\phi_n(\zeta)$ converge to a  point, say $c$, and $\phi_n(F(\zeta)) =\phi_n(F(\xi))$ converge to a point, say $v$.  Then $c$ must be a critical point of $G$ and $G(c) = v$.  Since $G_n \to G$ uniformly, by taking a subsequence if necessary,  there are critical points of $G_n$ near $c$, say $c_1^n, \cdots, c_m^n$,   all of which converge to $c$ as $n \to \infty$ such that
\begin{equation}\label{degree-in}
\deg_c G -1 = \sum_{i=1}^m(\deg_{c_i^n} G_n -1),
\end{equation} where $\deg_x $ denotes the local degree of the map at $x$.  Again by taking a subsequence if necessary, we may assume that  $c^n_i = \phi_n(c_i)$  with $c_i, 1 \le i \le m$, being critical points of $F_n$ (also of $F$). Then
 ${\rm dist}(\phi_n(F(c_i)), \phi_n(F(\zeta))) = {\rm dist}(G_n(c_i^n), G_n(\phi_n(\zeta)) \to 0$. By the definition of $\zeta$, we have $F(c_i) = F(\zeta)$ for all $1 \le i \le m$.  This implies that $G_n(c_i^n) = \phi_n(F(\zeta))$ for all $1 \le i \le m$. Note that $\phi_n(\zeta) \ne \phi_n(\xi)$ and  $G_n(\phi_n(\zeta) = G_n(\phi_n(\xi)) = \phi_n(F(\zeta))$. Since $G_n \to G$ uniformly in a neighborhood of $c$ and since $\phi_n(F(\zeta))$ converges to $v$,  the number of the pre-images of $ \phi_n(F(\zeta))$ under $G_n$ in a small neighborhood of $c$, counting by multiplicities,  must be equal to $\deg_c G$. On the other hand, since $G_n(c_i^n) = \phi_n(F(\zeta))$ for all $1 \le i \le m$,  this number is at least $\sum_{i=1}^m \deg_{c_i^n}G_n$.
From (\ref{degree-in}) it follows  that $m = 1$ and $\deg_c G = \deg_{c_1^n} G_n$. But $c_1^n, \phi_n(\zeta)$ and $\phi_n(\xi)$ are all mapped to $\phi_n(F(\zeta))$ by $G_n$ and $\phi_n(\xi) \ne \phi_n(\zeta)$.  So the number of the pre-images of $ \phi_n(F(\zeta))$ under $G_n$ in a small neighborhood of $c$ is at least $\deg_{c_1^n}G_n + 1 = \deg_c G  +1$. This is a contradiction. This proves (7) and completes the proof of Lemma~\ref{bounded-geometry}.

\end{proof}

 \begin{lem}\label{premodel}
 There exist a pair of homeomorphisms $\phi, \psi:\widehat{\Bbb C} \to \widehat{\Bbb C}$ which fix $0$, $1$ and $\infty$ such that
 \begin{itemize}
 \item[1.] $\phi \circ F = G \circ \psi$,
 \item[2.] $\phi$ is isotopic to $\psi$ rel $P_F \cup \cup_i \overline{D_i}$
 and $\phi|D_i = \psi|D_i$ is holomorphic for each  $D_i$.
 \end{itemize}
 \end{lem}
\begin{proof}
By (1) and (2) of Lemma~\ref{bounded-geometry}  and the fact that $\phi_n = \psi_n$ is holomorphic in all  $\overline{D_i} \cup A_i$,  by taking a subsequence if necessary,     we may assume that for each $D_i$,  $\phi_n= \psi_n$   converges uniformly to a univalent map on some domain  containing $\overline{D_i}$. Let $U_i = \lim_{n\to \infty} \phi_n(D_i)$.
Then by Lemmas~\ref{circle} and ~\ref{bounded-geometry} it follows that when restricted to the set $P_F \cup \cup_i \overline{D_i}$ (Note that $P_{F_n} \setminus \Bbb T = P_F \setminus \Bbb T$),
 $\phi_n$ and $\psi_n$  converge uniformly to some homeomorphism
\begin{equation}\label{lm}
 \sigma: P_F \cup \cup_i \overline{D_i} \to P_G \cup \cup_i \overline{U_i}.
\end{equation}
Note that $\sigma|\Bbb T = h$ where $h: \Bbb T \to \Bbb T$ is the circle homeomorphism in Lemma~\ref{circle} and $\sigma$ is holomorphic in each $D_i$, and moreover,  \begin{equation}\label{lm'} \sigma \circ F|P_F \cup \cup_i \overline{D_i} = G \circ \sigma|P_F \cup \cup_i \overline{D_i}.\end{equation}

The maps $\phi$ and $\psi$ are constructed by perturbing $\phi_n$ and $\psi_n$ for a large $n$. Before that, by Lemma~\ref{bounded-geometry} we may assume that there exists a $\delta> 0$ such that for every $n \ge 1$,
\begin{itemize}
\item[1.] the closures of
 the  $\delta$-neighborhoods of all points in $P$,   the closures of  the $\delta$-neighborhoods of all holomorphic disks $D_i$, and the closure of the $\delta$-neighborhood of $\Bbb T$, are all disjoint with each other.
 \end{itemize}
 Since when restricted to $P_F \cup \cup_i \overline{D_i}$, $\phi_n = \psi_n$ converges uniformly to a homeomorphism $\sigma:P_F \cup \cup_i \overline{D_i} \to P_G \cup \cup_i \overline{U_i}$,    by deforming $\phi_n$ rel $P_{F_n}$ we may further assume
\begin{itemize}
    \item[2.] when restricted to the closure of each of the above $\delta$-neighborhoods, $\phi_n$ converges uniformly  to a homeomorphism $\chi$, \end{itemize} Since $F_n \to F$ and $G_n \to G$ uniformly with respect to the spherical metric, from $\phi_n \circ F_n = G_n \circ \psi_n$,
\begin{itemize} \item[3.]
  there is an $\eta > 0$ such that  when restricted to the $\eta$-neighborhood of $P_F \cup \cup_i\overline{D_i}$, $\psi_n$ converge uniformly to some homeomorphism $\tau$.
 \end{itemize}
 It is clear that $\chi|P_F \cup \cup_i \overline{D_i} = \tau|P_F \cup \cup_i \overline{D_i} = \sigma$.

  Now let $\epsilon > 0$ be an arbitrarily small number and fixed.  By taking $n$ large enough, form the above assumption on $\phi_n$,   there exists a homeomorphism $\phi: \widehat{\Bbb C} \to \widehat{\Bbb C}$ which fixes $0$, $1$ and $\infty$,    such that ${\rm dist}(\phi, \phi_n) < \epsilon$, and moreover,
when restricted to  the closure of each of the above $\delta$-neighborhoods, $\phi = \chi$.

Let us now construct the homeomorphism $\psi: \widehat{\Bbb C} \to \widehat{\Bbb C}$.  Let  $\Pi$ denote the set of the critical values of $F$.  Take $0< \rho < \delta$  such that all the disks, $B_\rho(v), v \in \Pi$, are disjoint.  Let
 $$
 X = \bigcup_{v \in \Pi} B_{\rho}(v) \hbox{  and  } Y = F^{-1}(X).
 $$
Note that $Y$  is the union of finitely many Jordan domains  whose closures are disjonit from each other.   Since $v$ is the only critical value of $F$ contained in $B_\rho(v)$, each of these Jordan domains  either contains exactly one critical point of $F$, which is mapped to $v$ by $F$,  or contains no critical point of $F$.    In the following we will first define $\psi$ on $\widehat{\Bbb C} \setminus Y$
and then extend it to  $Y$.

Now take an arbitrary   $x \in \widehat{\Bbb C} \setminus Y$. Since the degree of $G$ is $2d-1$,
$G^{-1} (\phi \circ F(x))$ contains $2d -1$ points, counting by multiplicities. Define
 $\psi(x)$ to be the one which is closest, with respect to the spherical metric,   to $\psi_n(x)$ among these $2d-1$ points.  We need to explain that such definition does not cause any ambiguity.  This comes from the following two observations.  The first one is  that  when $\epsilon > 0$ is  small and $n$ is  large, the set
$G^{-1}(\phi\circ F(x))$ is close to the set $G_n^{-1}(\phi_n \circ F_n(x))$.   The  second one is that any two points in $G^{-1}(\phi \circ F(x))$ are uniformly bounded away from each other, because
 $\phi \circ F(x)$  is bounded away from the set of the critical values of  $G$.       Thus  $\psi$  can be well defined in $\widehat{\Bbb C} \setminus Y$.  From the definition it is clear that on $\widehat{\Bbb C} \setminus Y$, $\psi$ is locally homeomorphic  and satisfies  $\phi \circ F = G \circ \psi$.

  Let $U$ be one of the Jordan components of $Y$.  Note that  $\psi$ has been defined on $\partial U$  and satisfies $\phi \circ F = G \circ \psi$. Then $\psi (\partial U)$ is a component of $G^{-1}(\phi \circ F (\partial U))$. Since $\phi \circ F (\partial U)$ contains no critical values of $G$,   $\psi (\partial U)$ must be a Jordan curve. Let $V$ be the Jordan domain bounded by this curve.   Note that  $U$ and $V$ do not depend on the chocie of $\epsilon$ and $n$,  and that $\partial U$ contains no critical points of $F$, and $\partial V$ contains no critical point of $G$.  Let $V_n = \psi_n(U)$. By taking $\epsilon > 0$ small and $n$ large enough, $\psi$ and $\psi_n$ can be arbitrarily close to each other on $\partial U$. Thus $\partial V_n$ and $\partial V$ can be arbitrarily close to each other. We have two cases.

  In the first case, $U$ contains no critical points of $F$, and thus contains no critical points of $F_n$ for all $n$ large enough. This implies that  $V_n$ contains no critical points of $G_n$ for all $n$ large enough and thus $V$ contains no critical points of $G$. Then for any $z \in U$, there is a unique point $w \in V$ such that $\phi (F(z)) = G(w)$. Define $\psi(z) = w$. It is easy to see that $\psi: U \to V$ is a homeomorphism and $\phi \circ F = G \circ \psi$.

  In the second case,
   $U$ contains exactly one critical point of $F$, say $c$.  Then $U$ contains exactly one critical point of $F_n$, say $c_n$, which has the same local degree as $c$ and $c_n \to c$ as $n \to \infty$. Thus $V_n$ contains exactly one of the critical points of $G_n$, $\psi_n(c_n)$, which has the same local degree as that of $c_n$.  Since $G_n \to G$ uniformly with respect to the spherical metric, it follows that $V$ has exactly one   critical point of $G$, which has the same degree as that of $\psi_n(c_n)$, and thus has the the same local degree as that of $c$.  Then there is an obvious way to extend $\psi$ to $U$ such that $\psi: U \to V$ is a homeomorphism and $\phi \circ F = G \circ \psi$. In particular, $\psi$ maps the critical point of $F$ to a critical point of $G$.

From the construction we have $\phi \circ F = G \circ \psi$. Note that by taking $\epsilon > 0$ small and $n$ large, $\phi$ and $\phi_n$, $F$ and $F_n$, and $G$ and $G_n$ can be arbitrarily close to each other.   From $\phi_n \circ F_n = G_n \circ \psi_n$ and  $\phi \circ F = G \circ \psi$, it follows that $\psi$ and $\psi_n$ can be arbitrarily close to each other provided that $\epsilon > 0$ is  small enough and $n$ is large enough. Since $\psi_n$ fixes $0$, $1$ and $\infty$, and  since $\psi$ can be arbitrarily  close to $\psi_n$ and maps critical points of $F$ to critical points of $G$, it follows that $\psi$ fixes $0$, $1$ and $\infty$ also.
From the construction,   it follows  that  $\psi: \widehat{\Bbb C} \to \widehat{\Bbb C}$ is locally homeomorphic and thus a covering map.   By Riemann-Hurwitz formula, the covering degree must be equal to one. So  $\psi: \widehat{\Bbb C} \to \widehat{\Bbb C}$ is a homeomorphism.

 By the definition of $\phi$, it follows that $\phi|P_F\cup \cup_i \overline{D_i} = \sigma$  where $\sigma$ is the map  in (\ref{lm}).  This, together with   (\ref{lm'}) and $\phi \circ F = G \circ \psi$, implies  that  $ G \circ \sigma|P_F\cup \cup_i \overline{D_i} = G \circ \psi|P_F\cup \cup_i \overline{D_i}$.   Since $\psi_n|P_F\cup \cup_i \overline{D_i} \to \sigma$ and $\psi_n$ can be arbitrarily close to $\psi$  provided that $\epsilon > 0$ is small and $n$ is large,  it follows that $\psi|P_F \cup \cup_i \overline{D_i}$ can be arbitrarily close to  $\sigma$ provided that $\epsilon > 0$ is small enough and $n$ is large enough. From  $G \circ \sigma|P_F\cup \cup_i \overline{D_i} = G \circ \psi|P_F\cup \cup_i \overline{D_i}$ we have $\psi|P_F\cup \cup_i \overline{D_i} = \sigma$.  In particular, $\phi|D_i = \psi|D_i$ are holomorphic.

  It remains to show that there is an isotopy between $\phi$ and $\psi$ rel $P_F \cup \cup_i\overline{D_i}$. Since for orientation preserving surface homeomorphisms, homotopy implies isotopy (cf. \cite{Ba},\cite{Ep}), it suffices to show the existence of a homotopy between $\phi$ and $\psi$ rel $P_F \cup \cup_i\overline{D_i}$.

   Recall that $P_{F_n} \cap \Bbb T = \mathcal{O}_n$ and $\phi_n|\mathcal{O}_n = \psi_n|\mathcal{O}_n$.   Since the arc components of $\Bbb T \setminus \mathcal{O}_n$ can be arbitrarily small provided that $n$ is large,   we can construct  a homeomorphism $\omega_n: \widehat{\Bbb C} \to \widehat{\Bbb C}$ by deforming $\psi_n$ in a small neighborhood of $\Bbb T$ such that \begin{itemize} \item[1.]  $\omega_n|\Bbb T = \phi_n|\Bbb T$, \item[2.]  $\omega_n$ is isotopic to $\psi_n$ rel $P_{F_n} \cup \cup_i\overline{D_i}$, \item[3.]  $\omega_n$ can be arbitrarily close to $\psi_n$ provided that $n$ is large enough. \end{itemize}  Since $\phi_n$ is isotopic to $\psi_n$  rel $P_{F_n} \cup \cup_i \overline{D_i}$, it follows that  $\phi_n$ is isotopic to $\omega_n$ rel $P_{F_n} \cup \cup_i \overline{D_i}$. Let $H(t, \cdot)$, $0 \le t \le 1$, be the isotopy between $\phi_n$  and $\omega_n$. Since $\phi_n|\Bbb T = \omega_n|\Bbb T $, $H$ can be constructed such that $H(t, \cdot)|\Bbb T =\omega_n|\Bbb T = \phi_n|\Bbb T$ for  all $0 \le t \le 1$.

Now let $\xi: \widehat{\Bbb C} \to \widehat{\Bbb C}$ be a homeomorphism  such that
 $$
 \phi |P_F  \cup \cup_i\overline{D_i}= \xi \circ  \phi_n |P_F  \cup \cup_i\overline{D_i} = \xi \circ \omega_n|P_F  \cup \cup_i\overline{D_i}.
 $$ Note that
  when restricted to $P_F  \cup \cup_i\overline{D_i}$,  $\omega_n = \phi_n$ converges to $\phi$. So  the $\xi$ can be chosen to be arbitrarily close to  $\rm id$  provided that $n$ is large enough.   This implies that by taking $\epsilon > 0$ small enough and  $n$ large enough,  $\phi$ and $\xi \circ  \phi_n$ can be arbitrarily close to each other.  Thus $(\xi \circ \phi_n) \circ \phi^{-1}$ can be arbitrarily close to $\rm id$, and  by Lemma~\ref{homotopy},  they are homotopic to each other rel $P_G \cup \overline{U_i}$.  So $\phi$ and $\xi \circ  \phi_n$ are homotopic to each other  rel $P_F \cup \cup_i \overline{D_i}$.

     Since by taking $n$ large enough and $\epsilon > 0$ small enough, $\xi$ can be arbitrarily close to $\rm id$,
  $\omega_n$ can be arbitrarily close to $\psi_n$, and $\psi_n$ can be arbitrarily close to $\psi$,  it follows that
  $\xi \circ \omega_n$ can be arbitrarily close to $\psi$.  Again by  Lemma~\ref{homotopy}, $(\xi \circ \omega_n)\circ \psi^{-1}$ is homotopic to $\rm id$ rel $P_G \cup_i \overline{U_i}$.  So $\psi$ and $\xi \circ  \omega_n$ are homotopic to each other  rel $P_F \cup \cup_i \overline{D_i}$.

  From the above and the fact  that $\xi \circ H(t, \cdot)$,   $0\le t \le 1$,  is an isotopy between $\xi \circ \phi_n$ and $\xi \circ \omega_n$ rel $P_F \cup \cup_i \overline{D_i}$, it follows that $\phi$ and $\psi$ are homotopic to each other  rel $P_F \cup \cup_i \overline{D_i}$. The proof of Lemma~\ref{premodel} is completed.

\end{proof}

Now we can prove the existence part of  Theorem~\ref{Thurston-Siegel-Ch} by performing qc surgery on $G$. The process is routine. Let $h: \Bbb T \to \Bbb T$ be the qs circle homeomorphism in Lemma~\ref{circle}. Let $H: \Delta \to \Delta$ be the Douady-Earle extension of $h$. Define
\begin{equation}\label{model}
\widehat{G}(z) =
\begin{cases}
G(z) & \hbox{  for  } |z| \ge 1 \\
H\circ R_{\alpha} \circ H^{-1}(z) & \hbox{ for } |z| < 1
\end{cases}
\end{equation}
Let $\mu_0$ be the complex structure in $\Delta$ obtained by pulling back  the standard complex structure in $\Delta$ by $H^{-1}$. We then pull back $\mu_0$ to the whole plane by the iteration of $\widehat{G}$ and get a $\widehat{G}$-invariant complex structure $\mu$. Let $\xi$ be the qc homeomorphism of the plane to itself which solves the Beltrmai equation given by $\mu$ and fixes  $1$ and maps $H(0)$ to $0$.  Then
$g = \xi \circ \widehat{G} \circ \xi^{-1} \in \Sigma_\alpha^d$.
   Let us show that $f$  is CLH-equivalent to $g$.  Recall that $\phi \circ F = G \circ \psi$. Define
$\widehat{\phi}: \widehat{\Bbb C} \to \widehat{\Bbb C}$ by setting $\widehat{\phi}(z) = \phi(z)$  for $|z| \ge 1$
and $\widehat{\phi}(z) = H(z)$ for $|z| < 1$. Similarly, define $\widehat{\psi}: \widehat{\Bbb C} \to \widehat{\Bbb C}$ by setting $\widehat{\psi}(z) = \psi(z)$  for $|z| \ge 1$
and $\widehat{\psi}(z) = H(z)$ for $|z| < 1$. The isotopy between $\phi$ and $\psi$ rel $P_F \cup \cup_i \overline{D}_i$ induces an isotopy between $\widehat{\phi}$ and $\widehat{\psi}$ rel $P_f \cup \cup_i\overline{D_i}$, where the later union contains only those $D_i$ in the outside of the unit disk.  Let $\Omega_i, 1 \le i \le m$, denote all the components of $f^{-1}(\Delta)$ other than $\Delta$.   By only changing $\widehat{\psi}$ in the interior of  each $\Omega_i$ we can get a homeomorphism, say $\widetilde{\psi}$, such that
$\widehat{\phi} \circ f = \widehat{G} \circ \widetilde{\psi}$.  Since   each
$\Omega_i$ is a Jordan domain which does not intersect $P_f \cup \cup_i \overline{D_i}$, and $\widetilde{\psi}|\partial \Omega_i = \widehat{\psi}|\partial \Omega_i$, it follows that  $\widetilde{\psi}$ is isotopic to $\widehat{\psi}$, and is thus isotopic to $\widehat{\phi}$ rel $P_f \cup \cup_i \overline{D_i}$.  From $g = \xi \circ \widehat{G} \circ \xi^{-1}$ and
$\widehat{\phi} \circ f = \widehat{G} \circ \widetilde{\psi}$ we get
$$\xi \circ \widehat{\phi}\circ f = g \circ \xi \circ \widetilde{\psi}.$$

Note that  on each holomorphic disk $D_i$ of $f$, $\widehat{\phi} = \widehat{\psi} = \widetilde{\psi}$ is holomorphic, and that
in the attracting basin of  each attracting cycle of $G$ which lies in the exterior of $\Delta$, $\mu = 0$ and thus $\xi$ is holomorphic. This implies that $\xi \circ \widehat{\phi} = \xi \circ \widetilde{\psi}$ is holomorphic on each holomorphic disk $D_i$ of $f$. Since $\widehat{\phi}$ is isotopic to $\widetilde{\psi}$ rel $P_f \cup \cup_i \overline{D_i}$, $\xi \circ \widehat{\phi}$ is isotopic to $\xi \circ \widetilde{\psi}$ rel $P_f \cup \cup_i \overline{D_i}$.   Since $\widehat{\phi}|\Delta = \widetilde{\psi}|\Delta = H$, by the definition of $\xi$, it follows that $\xi \circ \widehat{\phi}|\Delta = \xi \circ \widetilde{\psi}|\Delta$ is holomorphic.  All of these implies that $f$ is CLH-equivalent to $g$. This proves the existence part of Theorem~\ref{Thurston-Siegel-Ch}.

Now it remains to prove that the $g$ is unique up to a linear conjuation.
Let  $\widehat{G}$ be the modified Blaschke product  defined in (\ref{model}).  Let $K_{\widehat{G}}$ be the set of all points whose forward orbits under the iteration of $\widehat{G}$ is bounded. Let $J_{\widehat{G}} = \partial K_{\widehat{G}}$. We call $J_{\widehat{G}}$  the Julia set of $\widehat{G}$. Let us first  prove the uniqueness part of Theorem~\ref{Thurston-Siegel-Ch} by assuming the following lemma.
\begin{lem}\label{zm}
The set $J_{\widehat{G}}$ has zero Lebesgue measure.
\end{lem}
Suppose $f$ is also CLH-equivalent to a Siegel polynomial $h \in \Sigma_\alpha^d$.
 By Shishikura's theorem, the boundary of the Siegel disk of $h$ is also a quasi-circle. Since $f$ is CLH-equivalent to both $g$ and $h$,  we have a pair of qc homeomorphisms $\phi_1, \phi_2: \widehat{\Bbb C} \to \widehat{\Bbb C}$ such that
 \begin{itemize}
 \item[1.] $\phi_1 \circ g  = h\circ \phi_2$,
 \item[2.] $\phi_1$ is isotopic to $\phi_2$ rel $P_g$ where $P_g$ is the post-critical set of $g$,
 \item[3.] when restricted to the interior of the Siegel disk and an open neighborhood of each attracting periodic cycle of $g$, $\phi_1 = \phi_2$ is holomorphic.
 \end{itemize}
Note that both $\phi_1$ and $\phi_2$ must map the center of the Siegel disk for $g$ to the center of the Siegel disk for $h$, and map the critical points of $g$ to those of $h$. So by a linear conjugation if necessary, we may assume that both $\phi_1$ and $\phi_2$ fix $0$, $1$, and $\infty$.  For $k \ge 2$ suppose  $\phi_k: \widehat{\Bbb C} \to \widehat{\Bbb C}$ is a qc homeomorphism  which is isotopic to $\phi_1$ rel $P_g$. Since  $\phi_1 \circ g  = h\circ \phi_2$,  we can define  $\phi_{k+1}$ by lifting $\phi_k$ through
\begin{equation}\label{lim-c}
\phi_k \circ g = h \circ \phi_{k+1}.
\end{equation} It is clear that $\phi_{k+1}$ is isotopic to $\phi_2$ and is thus isotopic to $\phi_1$ rel $P_g$. By induction we get a sequence of qc homeomorphisms $\phi_k: \widehat{\Bbb C} \to \widehat{\Bbb C}$ which fix $0$, $1$ and $\infty$ and satisfy the above equation.
Note that the qc constants of each $\phi_k$ is bounded  by that of $\phi_1$. Let $\mu_k$ be the Beltrami coefficient of $\phi_k$. Since all the points in the Fatou set of $g$ is either attracted to some attracting cycle of $g$, or is eventually mapped to the interior of the Siegel disk of $g$,  $\mu_k \to 0$ on the Fatou set of $g$. Since $J_g$ is the qc image of $J_{\widehat{G}}$, by Lemma~\ref{zm} $J_g$ has zero Lebesgure measure. So $\mu_k \to 0$ a.e. This implies that $\phi_k$ converges to $\rm id$ uniformly in any compact set of the plane.  From  (\ref{lim-c}) it follows that  $g = h$.  This implies the uniqueness part of Theorem~\ref{Thurston-Siegel-Ch}.

 Now let us prove Lemma~\ref{zm}.  The proof is by contradiction. Suppose $J_{\widehat{G}}$ has positive Lebesgue measure.  Let $z_0$ be a Lebesgue point of $J_{\widehat{G}}$.  For $n \ge 1$ let $z_n = G^n(z_0)$ (Note that $\widehat{G} = G$ in the exterior of $\Delta$).
By  Proposition 1.14 of \cite{Ly}, $z_n$ accumulates to $\Bbb T$.   The idea of the proof is adapted from \cite{McM2} and \cite{Zh3}.  We
   first show that there exist cones spanned at the critical points in $\Bbb T$ and   a subsequence $z_{n_j}$ such that  each $z_{n_j}$ belongs to  one of these cones. Then for each $n_j$, we can take a small disk, say  $B_j$, in the cone such that
   \begin{itemize}
   \item[(1)] $G(B_j) \subset \Delta$,
   \item[(2)] ${\rm dist}(B_j, z_{n_j}) \preceq {\rm dist}(B_j, \Bbb T) \asymp {\rm diam}(B_j) \asymp {\rm dist}(z_{n_j}, \Bbb T)$.
   \end{itemize}
  Let $P_G$ denote the post-critical set of $G$.  From (2) we can take a Jordan domain $A_j$ which is disjoint with $P_G$  and contains both $B_j$ and $z_{n_j}$ such that
  $${\rm diam}(A_j)\asymp {\rm dist}(A_j, \Bbb T)
  $$  where ${\rm diam}(\cdot)$ denotes the diameter with respect to the Euclidean metric.
  Let $X$ be the unbounded component of $\widehat{\Bbb C} \setminus P_G$.
  The above relation implies that $${\rm diam}_{X}(A_j) < K$$ for some uniform constant $K > 0$, where ${\rm diam}_X(\cdot)$ denotes the diameter with respect to the hyperbolic metric in $X$.  Now we pull back $A_j$ along the orbit $z_0, \cdots, z_{n_j}$, and denote the component of $G^{-k}(A_j)$ which contains $z_{n_j -k}$ by $A_j^{n_j -k}$. Let $X_k$ be the unbounded component of $G^{-k}(X)$. Then $G^k: X_k \to X$ is a holomorphic covering map and $X_k \subset X$. Note that $A_j^{n_j -k} \subset X_k$. It follows that $${\rm diam}_{X}(A_j^{n_j -k})< {\rm diam}_{X_k}(A_j^{n_j -k}) = {\rm diam}_{X}(A_j) < K.$$  For each $1 \le i \le j$,
  since $A_j^{n_i}$ contains $z_{n_i}$ and $z_{n_i}$ is contained in some cone spanned at some critical point, from the above inequality it follows that $A_j^{n_i}$ is contained in a definite cone spanned at this critical point. Thus the inverse branch of $G^{-1}$ which maps $A_j^{n_i +1}$ to $A_j^{n_i}$ contracts the hyperbolic diameter by a definite factor $0< \delta < 1$ (cf. Lemma 1.11 of \cite{P1} or Lemma 3.2 of \cite{Zh3}), that is,
  $$
  {\rm diam}_X(A_j^{n_i}) < \delta \cdot {\rm diam}_X(A_j^{n_i+1}).
  $$
This implies that ${\rm diam}_X(A_j^0) < \delta^j \cdot {\rm diam}_X(A_j) < \delta^j \cdot K \to 0$  as $j \to \infty$.  It follows  that $${\rm diam}(A_j^0) \to 0  \hbox{  as  }j \to \infty.$$   In addition, since ${\rm dist}(A_j, P_G) \asymp {\rm diam}(A_j)$, the distortion of $G^{-n_j}$ on $A_j$ is uniformly bounded.  Let $B_{j}^0$ denote the component of $G^{-n_j}(B)$ which is contained in $A_j^0$. We have $$area(B_j^0) \succeq {\rm diam}(A_j^0)^2.$$ Since  $B$ is disjoint from $J_{\widehat{G}}$, it follows that $B_j^0$ is disjoint from $J_{\widehat{G}}$.  This is a contradiction with the assumption that $z_0$ is a Lebesgue point of $J_{\widehat{G}}$.

 It remains to show  the existence of the cones  and the subsequence $n_j$ satisfying the conditions in the last paragraph.   For each open arc $I \subset \Bbb T$,   consider the space
$$
\Omega_I = \widehat{\Bbb C} \setminus (P_G \setminus I).
$$  Let $d_{\Omega_I}(\cdot,\cdot)$ denote the hyperbolic distance in $\Omega_I$ and for $d_0 > 0$, let
$$
\Omega_{d_0}(I) = \{z \in \Omega_I \:|\:d_{\Omega_I}(z, I) < d_0\}.
$$
As in Lemma~\ref{hn-ps},  when $I$ is small, $\Omega_{d_0}(I)$ is like the hyperbolic
neighborhood in the slit plane, that is, it is almost like the domain bounded two arcs of Euclidean circles which are symmetric about each other and such that the four exterior angles formed by the two arcs and $\Bbb T$ are all equal to $\sigma$ with $d_0 = \ln \cot (\sigma/4)$.  Define
$$
H_{d_0}(I) = \{z\:|\: |z| > 1 \hbox{  and  } z \in \Omega_{d_0}(I)\}.
$$
Then $H_{d_0}(I)$ is bounded by $I$ and  $S =  \partial \Omega_{d_0} \setminus \Delta$. Take $d_0 > 0$ such that the two exterior angles formed by $\Bbb T$ and $S$  are equal to $\big{(}1- \frac{1}{4(2d-1)}\big{)}\pi$. It follows that for any $z \in \Bbb T \setminus I$, if $V$ is a cone spanned at $z$ such that the angles formed by the two rays and $\Bbb T$ are equal to $\frac{\pi}{3(2d-1)}$, then $V$ does not intersect $H_{d_0}(I)$.

Now  let $h: \Bbb T \to \Bbb T$ be the circle homeomorphism such that $G|\Bbb T = h^{-1} \circ R_{\alpha} \circ h$ and $h(1)  = 1$.
For each $z_n$, let $I_n \subset I$ be the arc such that $z_n \in \overline{H_{d_0}(I_n)}$ and moreover, $I_n$ is the smallest one in the following sense
$$
|h(I_n)| = \min\{|h(I)|\:|\: I \subset \Bbb T \hbox{  and  } z_n \in \overline{H_{d_0}(I)}\}.
$$
Since $z_n \to \Bbb T$, we have $|I_n| \to 0$  and thus $$|h(I_n)| \to 0  \hbox{  as  }n \to \infty.$$
So there  is an increasing subsequence of integers, say  $m_j$,   such
that
$$
|h(I_{m_j})| <  |h(I_n)| \:\:\:\hbox{  for all } 1 \le n < m_j.
$$

Let $n_j = m_j -1$. We claim $\{n_j\}$ is the desired subsequence. Let us prove the claim. Since $|I_{m_j}| \to 0$,  by disregarding finitely many $m_j$ we may assume  that each $\overline{I_{m_j}}$ contains at most one critical value of $G$. So
these are two cases. In the first case, $\overline{I_{m_j}}$ contains no critical value. In the second case, $\overline{I_{m_j}}$ contains
exactly one critical value. Let $J \subset \Bbb T$ be the arc such that $G(J) = I_{m_j}$.

In the first case,   Let $K$ be the component of $G^{-1}(\overline{H_{d_0}(I_{m_j})})$ which is attached to $\overline{J}$. By Schwarz lemma it follows that $K \subset \overline{H_{d_0}(J)}$.  By the minimal property of $I_{m_j}$, it follows that
$z_{m_j -1} = z_{n_j} \notin \overline{H_{d_0}(J)}$. This implies that $z_{m_j}$ is near a critical value in $\Bbb T$, and $\overline{H_{d_0}(J)}$ is near some critical point $c$ in $\Bbb T$. By the choice of $d_0$, $\overline{H_{d_0}(J)}$ belongs to the angle domain bounded by $\Bbb T$ and a ray starting from $c$ such that the angle formed by the ray and $\Bbb T$ at $c$ is equal to $\frac{\pi}{3(2d-1)}$.  Let $m \ge 3$ be  the local degree of $G$ at $c$. Then $m \le 2d-1$. Now in a small neighborhood of $c$, we may regard $G$ approximately as the map $z \mapsto \lambda \cdot (z - c)^m + v$ where $\lambda \ne 0$ is some constant and $v = G(c)$.  Let  $w \in \overline{H_{d_0}(J)}$ be such that $G(w) = G(z_{n_j}) = z_{m_j}$. Then the smaller angle between $\Bbb T$ and $[c, w]$ is   $z_{m_j}$ must belongs to an angle domain at $v = G(c)$ bounded by $\Bbb T$ and a ray starting from $v$ with angle being equal to $\frac{\pi}{3(2d-1)}$. This implies that for any other pre-image of $z_{m_j}$ near $c$, say $w'$, the smaller angle between $\Bbb T$ and $[c, w']$ is not less than $$\frac{\pi}{m} - \frac{\pi}{2(2d-1)} > \frac{\pi}{2(2d-1)} \bigg{(}> \frac{\pi}{m} - \frac{\pi}{3(2d-1)}\bigg{)}.$$   Let $V$ be the cone spanned at $c$ such that the exterior angles formed by the
two rays of $V$ and $\Bbb T$ are equal to  $\frac{\pi}{2(2d-1)}$. Since $z_{n_j}$ is one of the pre-images of $z_{m_j}$ near $c$,  it follows that $z_{n_j}$ is contained in $V$.

 In the second case, $\overline{I_{m_j}}$ contains
exactly one critical value $v$. Let $c$ be the critical point in $\Bbb T$ such that $G(c) = v$. Then $z_{n_j} = z_{m_j-1}$ is near $c$.    If $\angle([v, z_{m_j}], \Bbb T)  \ll \pi$ where $\angle([v, z_{m_j}], \Bbb T)$ denotes the smaller angle formed by $[v, z_{m_j}]$ and $\Bbb T$ at $v$,   we would have an arc $I \subset \Bbb T$ such that $z_{m_j} \subset H_{d_0}(I)$ with $|I|  \ll |I_{m_j}|$ and $I \cap I_{m_j} \ne \emptyset$.  Since $h$ is qusi-symmetric,  this  would imply $|h(I)| < |h(I_{m_j})|$ (cf. Lemma 4.8 of \cite{Zh3}).  This is a contradiction with the minimal property of $I_{m_j}$. So $\angle([v, z_{m_j}], \Bbb T)  \succeq \pi$.  This implies that   $\angle([c, z_{n_j}], \Bbb T)   \succeq \pi$ where $\angle([c, z_{n_j}], \Bbb T)$ denotes the smaller angle formed by $[c, z_{n_j}]$ and $\Bbb T$ at $c$.     Thus  $z_{n_j}$ must belong to a cone spanned at $c$ such that the two rays of the cone form a definite angle with $\Bbb T$.  This proves Lemma~\ref{zm}.  The proof of Theorem~\ref{Thurston-Siegel-Ch} is thus completed.

\subsection{Proof of  Key Lemma 2}

Let $g \in \Sigma_\alpha^d$ be the Siegel polynomial in Key Lemma 2. Let $f \in \mathcal{T}_\alpha^d$ such that $f$ is CLH-equivalent to $g$. Then $f$ has exactly $m$ distinct critical points in $\Bbb T$, say $c_i^0$  with $c_1^0 = 1$, $1 \le 1 \le m$.  Let $H$ be a thin neighborhood of $f$ such that $H \setminus \Bbb T \cap (\Omega_f \cup P_f) = \emptyset$. By perturbing $f$ in $H$  we  get a sequence $f_n \in \mathcal{T}_\alpha^d$ such that for each $f_n$, there are exactly $m$ distinct critical points $c_i^n$ in $\Bbb T$ with $c_1^n = 1$, $1 \le i \le m$,  each of which has the same local degree as the corresponding one of $f$, and moreover, there are $m$ integers $k^n_i \ge 0$, $1 \le i \le m$, such that $f_n^{k_i^n}(1) = c_i^n$. Since $f_n$ is different from $f$
 only in a thin neighborhood of $\Bbb T$ which does not intersect $(\Omega_f \cup P_f) \setminus \Bbb T$, $f_n$ still has no Thurston obstruction in the exterior of $\Delta$. By Theorem~\ref{Thurston-Siegel-Ch}, we get a sequence $g_n \in \Sigma_\alpha^d$ such that $f_n$ is CLH-equivalent to $g_n$.
 Let $K_n$ be the filled Julia set of $g_n$. Since all the critical orbits of $g_n$ are bounded, $K_n$ is connected.
 Then the B\:{o}ttcher map $\Phi: \widehat{\Bbb C} \setminus K_n \to \widehat{\Bbb C} \setminus \overline{\Delta}$ conjugates $g_n$ to the map $z \mapsto z^d$. Near infinity, $\Phi(z) = \alpha z + O(1/z)$ with $\alpha^{d-1}$ being the leading coefficient of $g_n$. By Koebe's $1/4$-theorem, $K_n$, and thus all the critical points of $g_n$,  are contained in the disk $\{z\:|\: |z| < 4/|\alpha|\}$. Let $c_i^n$, $1 \le i \le d-1$, be all the critical points of $g_n$. By (\ref{form-1}) and (\ref{form-2}), we have
 $$
|c_i^n| \le 4 \bigg{(}\prod_{1 \le j \le d-1}|c_j^n|\bigg{)}^{1/d-1} \hbox{  for all } 1 \le i \le d-1.
 $$This implies that $|c_i^n| < 4^{d-1}$ for all $1 \le i \le d-1$. By Lemma~\ref{plmd}, there is a $\delta > 0$ independent of $n$ such that $|c_i^n| > \delta$ for all $1 \le i \le d-1$.  So by taking a subsequence we may assume that $c_i^n \to c_i^0$ with $c_i^0 \in \Bbb C \setminus \{0\}$, $1 \le i \le d-1$. Let $g_0$ be the polynomial given by (\ref{form-1}) and (\ref{form-2}).  Since $g_n$ converges to $g_0$ uniformly in any compact set of the plane, all the attracting cycles converges to attracting cycles of $g_0$ with the same multipliers, and each critical point of $g_n$, which is attracted to some attracting cycle of $g_n$, will converge to a critical point of $g_0$, which is attracted to the corresponding attracting cycle of $g_0$. Besides this, the boundary of the Siegel disk of $g_0$, say $D$,  must contain all the critical points $c_i^0$, $1 \le i \le m$. Since otherwise, if some $c_i^0 \notin \partial D$, then by Lemma~\ref{continuous-moving}, $c_i^n \notin \partial D_n$ where $D_n$ is the Siegel disk of $g_n$ centered at the origin. This is a contradiction.

 Now it suffices to prove that $g_0 = g$ up to a linear conjugation, that is, $g(z) = \frac{1}{a}g_0(az)$ with $a = c_i^0$ for some $1 \le i \le m$. This is because Key Lemma 2
 will follow from this by taking $\tilde{g}(z) = \frac{1}{c_i^n}g_n(c_i^n z)$  for some $n$ large enough.  To see this, by the rigidity part of Theorem~\ref{Thurston-Siegel-Ch}, it suffices to prove that $f$ is CLH-equivalent to $g_0$ also. From our proof, $f_n$ is CLH-equivalent to $g_n$.  So for each $n$, there exist  two plane homeomorphisms $\phi_n, \psi_n:  {\Bbb C} \to  {\Bbb C}$ such that
\begin{itemize} \item[1.] $\phi_n|\Delta = \psi_n|\Delta$ is holomorphic,
\item[2.] for each holomorphic attracting cycle $\mathcal{O}$ of $f_n$ if there is any, there is an open neighborhood $U$ of $\mathcal{O}$ such that $\phi_n|U = \psi_n |U$ is holomorphic,
\item[3.] $\phi_n$ is isotopic to $\psi_n$ rel $P_{f_n} \cup \cup_i \overline{D_i}$ where $D_i$ are  open neighborhoods of all holomorphic attracting cycles of $f_n$,
    \item[4.] $\phi_n \circ f_n = g_n \circ  \psi_n$.
\end{itemize}
By taking a subsequence, we may assume that there is a $1 \le j \le m$ such that
$\phi_n(1) = \psi_n(1)= c_j^n$. Define $\tilde{g}_n$  by $\tilde{g}_n(z) = \frac{1}{c_j^n}g_n(c_j^n z)$. By considering $\tilde{g}_n$ instead of $g_n$,  we may assume that $\phi_n(1) = \psi_n(1) = 1$. Define $\tilde{g}_0$ by $\tilde{g}_0(z) =  \frac{1}{c_j^0}g_0(c_j^0 z)$. Then $\tilde{g}_n$ converges to $\tilde{g}_0$ uniformly in any compact set of the plane.
Now it suffices to prove that there exist  two plane homeomorphisms $\phi, \psi: {\Bbb C} \to {\Bbb C}$ which fixes $0$  and  $1$, such that
\begin{itemize} 
\item[(i)] $\phi|\Delta = \psi|\Delta$ is holomorphic,
\item[(ii)] for each holomorphic attracting cycle $\mathcal{O}$ of $f$ if there is any, there is an open neighborhood $U$ of $\mathcal{O}$ such that $\phi|U = \psi |U$ is holomorphic,
\item[(iii)] $\phi$ is isotopic to $\psi$ rel $P_{f} \cup \cup_i \overline{D_i}$ where $D_i$ are  open neighborhoods of all holomorphic attracting cycles of $f$,
\item[(iv)] $\phi \circ f = \tilde{g}_0 \circ  \psi$.
\end{itemize}
Recall that $f_n$ is obtained by perturbing $f$ in a thin neighborhood of $f$, we may assume that $f_n|D_i = f|D_i$. So we may assume that 
$\phi_n|\cup_i D_i = \psi_n|\cup_i D_i$ uniformly converges to a univalent map $\sigma: \cup_i D_i \to \cup_i U_i$ with $\cup U_i $ containing all the attracting cycles of $\tilde{g}_0$. Let $\tilde{D}_n$ and $\tilde{D}$ denote the Siegel disks of $\tilde{g}_n$ and $\tilde{g}_0$, respectively.
It is clear that $\phi_n| \Delta = \psi_n|\Delta: \Delta \to \tilde{D}_n$ uniformly converges to a holomorphic isomorphism from $\Delta \to \tilde{D}$. Let us also use $\sigma: \Delta \to \tilde{D}$ denote this limit map.

Now we can use the same argument as in the proof of Lemma~\ref{premodel} to construct the desired homeomorphisms $\phi$ and $\psi$ by perturbing $\phi_n$ and $\psi_n$ for some large $n$. Let us just outline the proof.  Let  $\epsilon > 0$ be a small number. By taking  $n$ large enough, we can perturb $\phi_n$ to get a plane homeomorphism  $\phi: \Bbb C \to \Bbb C$ such that (1) $\phi|\Delta \cup \cup_i D_i = \sigma$  and (2) ${\rm dist}(\phi, \phi_n) < \epsilon$.
As in the proof of Lemma~\ref{premodel}, provided that $\epsilon > 0$ is small enough (and  thus $n$ must be large enough), we can construct a plane homeomorphism $\psi: \Bbb C \to \Bbb C$ such that the properties (i), (ii) and (iv) hold. Moreover, $\psi$ can be arbitrarily close to $\psi_n$  provided that $\epsilon > 0$ is small enough.  Then by the same argument as in the proof of Lemma~\ref{premodel},  it follows that   $\phi$ and $\psi$ are isotopic to each other rel $P_f \cup \overline{\cup_i D_i}$,  provided that $\epsilon > 0$ is small enough. This is the property (iii).  The proof of Key Lemma 2 is completed.

\section{Appendix}

\subsection{Appendix A}
\begin{lem}\label{plmd}
Let $f$ be a degree-$d$ polynomial map with a Siegel disk centered at
the origin. Then there exist $M > 1$ depending only on $d$
 such that if $f$ has two critical points $c$ and $c'$  with
$|c|/|c'| > M$, then there exist a pair of Jordan domains $U \Subset V$ containing the Siegel disk of $f$ centered at the origin
such that the tuple $(U, V, f)$ is a polynomial-like map which is qc conjugate to a polynomial map $g$ of degree less
than $d$. Moreover,  there exists a $L > 1$ which depends only on $d$ such that
$$
{\rm diam}(D) \le L \cdot \min_{c \in \Omega_f} |c|
$$ where $D$ is the Siegel disk of $f$ centered at the origin and $\Omega_f$ is the set of all critical points of $f$.
\end{lem}
\begin{proof}  Let $c_1, \cdots, c_{d-1}$ be the critical points of
$f$.  Through a linear conjugation, we may assume that
$1 = |c_{d-1}| \le |c_{d-2}| \le \cdots \le |c_1|$.

 Claim. For each $M_l \ge  1$,   there exists an $M_{l+1} > M_l$ such that if
 \begin{equation}\label{con-ass}
 1 = |c_{d-1}| \le |c_{d-2}| \le \cdots \le |c_{d-l}| \le  M_l  < M_{l+1} <  |c_{d-l-1}|\le \cdots \le |c_1|,
\end{equation}
 then there exist domains $U \Subset V$ containing the origin such that $(f, U, V)$ is a polynomial-like map which is qc conjugate
 to some polynomial of degree $l+1$.   Let us prove the Claim first.  Recall that
$$
f(z) =  f_{c_1, \cdots, c_{d-1}}(z) =  \sum_{i=1}^{d} a_i z^i
$$
with $$a_i = e^{2 \pi i
\alpha} \cdot \bigg{(}\frac{(-1)^{i-1}}{i} \bigg{)} \cdot
\frac{Q_{d-i}(c_1, \cdots, c_{d-1})}{c_1 \cdots c_{d-1}}$$ where
$Q_{d-i}$ is the degree-$(d-i)$ elementary polynomials of $c_1,
\cdots, c_{d-1}$.  If (\ref{con-ass}) holds,  a simple calculation shows that  there is some  constant $$C = C(d, M_l) > 1$$
 depending only on $d $ and $M_l$  such that by taking $M_{l+1}$ large enough, the following inequalities hold.
\begin{itemize}
 \item[(i)] $C^{-1} < |a_{l+1}| < C$, \item[(ii)]  $|a_k| \le C$  for $1 \le k \le l$, \item[(iii)]
$|a_k|  < d! \cdot M_{l+1}^{-1}$ for $ l+2 \le k \le d$.  \end{itemize}

  From  (i) and (ii) it follows that  by taking $R>0$ large enough,  we can make sure that
\begin{equation}\label{First-in}
\sum_{i=1}^{l} |a_i| |z|^i \ll |a_{l+1} z^{l+1}| \:\:\hbox{    for    } |z| \ge \bigg{(}\frac{  R}{2C}\bigg{)}^{1/(l+1)}.
\end{equation}
 Fix such an $R$. Then from (iii)  it follows that if we take $$M_{l+1} > R$$  large enough, then we have
\begin{equation}\label{Second-in}
\sum_{i=l+2}^{d} |a_i| |z|^i \ll |a_{l+1} z^{l+1}| \:\:\hbox{    for    } |z| \le |C\cdot R^{l+1}|.
\end{equation}
Now define a quasi-regular map $F$ as follows.
$$ F(z) =
\begin{cases}
f(z) \hbox{  for  } |z| \le R \\
f(z) - \frac{|z|-R}{|a_{l+1} R^{l+1}| - R} (f(z) - a_{l+1}z^{l+1}),  \hbox{ for } R <|z|< |a_{l+1} R^{l+1}| \\
a_{l+1}z^{l+1}$ if $|z| \ge |a_{l+1} R^{l+1}|.
\end{cases}
$$
By definition, $F$ is holomorphic for $|z| < R$ and $|z| > |a_{l+1} R^{l+1}|$.  From (i-iii) and a simple calculation we get
$$
|F_{\bar{z}}| \ll |F_z|  \hbox{ for } R < |z| < |a_{l+1} R^{l+1}|.
$$  provided that $R$  and $M_{l+1}$ are
 large enough (the choice of $M_{l+1}$ depends on $R$).
 This implies that the real dilatation of $F$ in $\{z\:|R < |z| < |a_{l+1} R^{l+1}|\}$ can be arbitrarily close to $1$. Note that the forward orbit of any point $z$ passes through the annulus $\{z\:|R \le |z| \le |a_{l+1} R^{l+1}|\}$ at most two times.  By a routine argument it follows  that $F$ is conjugate to a polynomial of degree $l+1$ through some $K$-qc homeomorphism where $K > 1$ can be arbitrarily close to $1$ provided that $R$ and $M_{l+1}$ are chosen appropriately large.

 Let $V = \{z\:|\:|z| < R\}$ and $U$ be the component of $f^{-1}(V)$
 which contains the origin.  From (\ref{First-in}), it follows that $U \Subset V$ and $f: U \to V$ is of degree $l+1$
 and thus $(U, V, f)$ is a polynomial-like map (whose Julia set may not be connected).
  From the last paragraph it follows that $(U, V, f)$ is $K$-qc conjugate to some polynomial of degree $l+1$. This proves the Claim.

 Now let $M_{1} = 1$. By successively applying the Claim we get $$1 = M_{1} < M_2 < \cdots< M_{d-1}.$$  Let  $M = M_{d-1}$.   Suppose $1 = |c_{d-2}| \le \cdots \le |c_1|$.

 In the first case,  $|c_i| \le M$ for all $1 \le i \le d-2$.   The first assertion obviously holds in this case.  Note that the absolute value of the
 leading coefficient of $f$ is $1/|c_1| \cdot \cdots \cdot |c_{d-1}| \ge 1/M^{d-1}$, and that all the other coefficients of $f$ are bounded above by some constant $K(d, M)$  depending on $d$ and $M$. This implies the existence of an $R(d, M)$ such that $f$ is dominated by the leading term for all $|z| > R(d, M)$. In particular, this implies that the diameter of the Siegel disk is not greater than $R(d, M)$. Since $M$ depends only on $d$, the second assertion of the lemma follows.

 In the second case, there is some $c_i$ such that $|c_i| > M$. Let $1 \le l \le d-2$ be the least integer such that $|c_{d-l-1}| > M_{l+1}$. Then we have  (\ref{con-ass}).  The first assertion follows from the Claim.  For the second assertion,  let us go back to the proof of the Claim. We see the Siegel disk centered at the origin is contained in $V = \{z\:|\: |z| < R\}$. Since $R < M_{l+1} \le M$, it follows that the diameter of the Siegel disk is not greater than $2 M$. Since $M$ depends only on $d$, the second assertion follows.

 The proof of Lemma~\ref{plmd} is completed.
\end{proof}

Let
$0< \alpha < 1$ be a bounded type irrational number and $d \ge 2$ be an integer. Let
$\mathcal{P}_{\alpha}^d$ denote the class of all the polynomial maps
$f$ such that
$$f(z) = e^{2 \pi i \alpha} z + \alpha_2 z + \cdots + \alpha_d z^d$$
with $\alpha_d \ne 0$ and $f'(1) = 0$. Let $\Delta$ denote the unit disk and $D$ denote the Siegel disk of $f$ centered at the origin.

\begin{lem}[\cite{Sh}, Shishikura]\label{shi} There exists a $K = K(\alpha, d) > 1$ depending only on $d$ and $\alpha$
such that for any polynomial map $f \in \mathcal{P}_{\alpha}^d$,  if $\phi: \Delta \to D$ is the holomorphic isomorphism such that $\phi^{-1} \circ f \circ \phi (z) = e^{2 \pi i \alpha} z$ for $z \in \Delta$, then $\phi$ can be extended to a $K$-qc
homeomorphism of the plane. In particular, the boundary of
the Siegel disk of $f$ centered at the origin is a $K$-quasicircle.
\end{lem}
For a proof, see \cite{Sh} or \cite{Zh1}.

\begin{lem}\label{continuous-moving}The boundaries  of the Siegel disks of $$f \in
\bigcup_{2 \le l \le d} \mathcal{P}_{\alpha}^l$$ at the origin moves  continuously  with respect to the Hausdorff metric on the spaces of non-empty compact sets of the plane and   the topology of $\bigcup_{2 \le l \le d} \mathcal{P}_{\alpha}^l$ is given by open-compact topology, that is, $f_n \to f$  with respect to this topology means that
$f_n$ uniformly converges to $f$ in any compact set of the plane.
\end{lem}
\begin{proof}
Let $f \in \bigcup_{2 \le l \le d} \mathcal{P}_{\alpha}^l$. Assume that $f_n \to f$.
Let $D$ and $D_n$ be
respectively the Siegel disks of $f$ and $f_n$ which are centered at
the origin.  It suffices to prove that
$\partial D_g$ is close to $\partial D_f$ with respect to the
Hausdorff metric for all $n$ large enough. It is known that both of them contains critical
points.

We may assume that $D \ne D_n$ since otherwise there is nothing to
prove. Since both $D$ and $D_n$ contains the origin as an interior
point there is a point $w \in (D \cap
\partial D_n) \cup (D_n \cap \partial D)$.  Without loss of generality, let us assume that $w \in D\cap \partial Dn$. Let $\Gamma_{w} \subset D$ be the $f$-invariant curve containing $w$.
Let $\mathcal{O}_f(w) = \{f^k(w)\}_{k\ge 0}$ and $\mathcal{O}_g(w) =
\{g^k(w)\}_{k\ge 0}$. Then $\mathcal{O}_f(w)$ and $\mathcal{O}_g(w)$
are dense in $\Gamma_w$ and $\partial D_n$ respectively. For any integer
$m \ge 1$, the two finite orbit segments
$$
\{f^k(w), 0 \le k \le m\}; \:\: \hbox{ and } \:\: \{f^k(w), 0 \le k \le m\}
$$
can be arbitrarily close to each other provided that $n$ is
 large enough.  By Lemma~\ref{shi} there exist  two $K(\alpha, d)$-qc homeomorphisms of the plane which fix $0$ and $\infty$, say $\phi$ and $\psi$ such that $\phi^{-1} \circ f \circ \phi(z) = e^{2 \pi i \alpha} z$ and  $\psi^{-1} \circ g \circ \psi(z) = e^{2 \pi i \alpha} z$ for all $z \in \overline{\Delta}$.
$\phi(\Bbb T_r) = \Gamma_w$ and $\psi(\Bbb T) = \partial D_g$.   Suppose $w = \phi(z_0)$ for some $z_0$ with $0 < |z_0| < 1$ and $w = \psi(\zeta_0)$ for some $\zeta_0 \in \Bbb T$. Then
Each component of $\Gamma_w \setminus \{f^k(w), 0 \le k \le m\}$ is the $\phi$-image of a component of
$$
\{z|\: |z| = |z_0|\} \setminus \{ e^{2 k \pi i \alpha} \cdot z_0\:|\: 0 \le k \le m\}
$$ and each component of $\partial D_n \setminus \{f_n^k(w), 0 \le k \le m\}$ is the $\psi$-image of a component of
$$
\Bbb T \setminus \{ e^{2 k \pi i \alpha}\cdot \zeta_0 \:|\: 0 \le k \le m\}.
$$
By the second assertion of Lemma~\ref{plmd} the Siegel disks of $f$ and $f_n$ are contained in some compact set of the plane. Since $\phi$ and $\psi$ fix $0$ and $\infty$ and are $K$-qc homeomorphisms of the plane for some $K$ depending only on $d$ and $\alpha$, there exist $C > 1$ and $0< \eta < 1$ which are independent of $n$ such that for any $z_1, z_2$ with $|z_1|, |z_2| \le 1$ we have
$$
|\phi(z_1) - \phi(z_2)| < C \cdot |z_1 - z_2|^\eta \hbox{  and  } |\psi(z_1) - \psi(z_2)| < C \cdot |z_1 - z_2|^\eta.
$$
Since the components of $$
\{z|\: |z| = |z_0|\} \setminus \{ e^{2 k \pi i \alpha} \cdot z_0\:|\: 0 \le k \le m\}
$$ and $$
\Bbb T \setminus \{ e^{2 k \pi i \alpha}\cdot \zeta_0 \:|\: 0 \le k \le m\}
$$ can be arbitrarily small provided that $m$ is large enough, it follows that $\Gamma_w$ and $\partial D_n$ can be arbitrarily to each other provided that $n$ is large enough.

Now we claim that $\Gamma_w$ is close to $\partial D$ also provided that $n$ is large enough. This
is because if not, then by Lemma~\ref{shi} it follows that the modulus of the
annulus bounded by $\partial D$ and $\Gamma_w$ has a  positive lower bound.  But since $\partial D_n$ is close to $\Gamma_w$ and
contains a critical point of $f_n$,  it follows that there is a
critical point of $f_n$ contained in $D$ and is bounded away from
$\partial D$. But since $f_n \to f$ uniformly in $D$,  there would be a
critical point of $f$ contained in $D$. This is impossible. This
completes the proof of Lemma~\ref{continuous-moving}.
\end{proof}
\begin{lem}\label{compact-class}
Suppose $f_n \in \bigcup_{2 \le l \le d} \mathcal{P}_{\alpha}^l$ is a sequence such that for each $f_n$,
the boundary of the Siegel disk centered at the origin passes through the critical point $1$.
Then $f_n$ has a subsequence $f_{n_k}$ which converges to some $f \in \bigcup_{2 \le l \le d} \mathcal{P}_{\alpha}^l$.
 Moreover, the Siegel disk of the limit polynomial map $g$ centered at the origin also passes through the critical point $1$,
 and  for any $k > m \ge 0$, we have
$$
\lim_{n \to \infty} \sigma_{k,m}(f_n) = \sigma_{k,m}(g).
$$
\end{lem}
\begin{proof}
By the second assertion of Lemma~\ref{plmd}, the critical points of all $f_n$ are uniformly bounded away from the origin; that us,
there is a uniform $L >  0$ such that for each $f_n$, the critical points of $f_n$ are contained in the outside of the disk $\{z\: |\: |z| > L\}$.
For each $f_n$, let us label the critical points of $f_n$ by
$$
c_1^n, \cdots, c_{d-2}^n, c_{d-1}^n = 1.
$$
By taking a subsequence, we get $1 \le i_1 < \cdots < i_l \le d-2$ and $d-1-l$
such that for all $1 \le j \le l$,
$$
c_{i_j}^n \to \infty \hbox{  and  } n \to \infty,
$$ and for all $k \ne i_j, 1 \le j \le l$ and $1 \le k \le d-1$, we have
$$
c_{k}^n \to c_k^*
$$ where $c_k^*$ is some non-zero complex number. Let $g$ denote the polynomial of degree $d-l$
 which has critical points at these $c_k^*, 1 \le k \le d-1$ and $k \ne i_j, 1 \le j \le l$.
 It is clear that $f_{n_k}$ converges to $g$ uniformly in any compact set of the plane. This proves the first assertion of Lemma~\ref{compact-class}.  Let us prove that  the boundary of the Siegel disk of $g$ centered at the origin must also contain the critical point $1$. Suppose this were not true. Then
 the critical point $1$ is bounded away from the boundary of the Siegel disk $g$.  Since $f_n \to g$, by Lemma~\ref{continuous-moving}, the boundary of the Siegel disk of $f_n$ centered at the origin can be arbitrarily close to the boundary of the Siegel disk of $g$ centered at the origin. This would imply for all $n$ large enough, the boundary of the Siegel disk of $f_n$ centered at the origin does not pass through the critical point $1$. This is a contradiction. For $k > m\ge 0$ given, $f_n^k \to g^k$ and $f_n^m \to g^m$ uniformly in any compact set of the plane. Thus $\sigma_{k,m}(f_n) \to \sigma_{k,m}(g)$ as $n \to \infty$. This implies the second assertion. The proof of Lemma~\ref{compact-class} is completed.
\end{proof}

\subsection{Appendix B}

\begin{lem}\label{homotopy}
Let $\Delta$ be the unit disk and
 $U_i$, $1 \le i \le l$,  be  Jordan domains  with $\overline{U_i} \subset \Delta$ and $\overline{U_i} \cap \overline{U_{i'}} = \emptyset$ for $1 \le i \ne i' \le l$.   Let $Q = \{q_1, \cdots, q_m\} \subset \Delta$ and $X = {\Delta} \setminus (Q \cup \cup_{1 \le i \le l} \overline{U_i})$. Suppose
\begin{itemize}
\item[1.] ${\rm diam}(U_i) > d_0$  for $1 \le i \le l$,
\item[2.] ${\rm dist}(U_i, q_j) > d_0$   for $1 \le i \le l$ and $1 \le j \le m$,
\item[3.] ${\rm dist}(q_j, q_{j'}) > d_0$   for $1 \le j \ne j' \le m$,
\item[4.] ${\rm dist}(q_j, \Bbb T) > d_0$  for $1 \le j \le m$.
\end{itemize}
Then there exists a $\tau > 0$ depending only on $d_0$ such that for any homeomorphism $h: \overline{X} \to \overline{X}$,
if ${\rm dist}(h, {\rm id}) < \tau$ and $h|\partial X = {\rm id}$, then $h$ is homotopic to {\rm id} rel $\partial X$.
\end{lem}
\begin{proof}
For $1 \le j \le m$ and $r > 0$, let $B_j(r) = \{z\:|\: |z - q_j| < r\}$.
Then  there is a $\eta > 0$ depending only on $d_0$
such that for any two points $a$ and $b$ in $\Delta \setminus  \overline{\cup_i U_i \cup  \cup_j B_j(d_0/4)}$, if the Euclidean distance between $a$ and $b$ is less than $\eta$,  there is a unique shortest geodesic segment in $X$  which connecting $a$ and $b$.

We first show that $h$ can be homotopic to a continuous map $h_0: \overline{X} \to \overline{X}$ such that $h_0|B_j(d_0/3) = {\rm id}$ for all $1 \le j \le m$.   Let us construct the homototy $H$
 as follows.

For $0< r \le d_0/3$, let $\Delta_{j,r} = \{z\:|\: z + q_j \in h(B_j(r))\}$. Let
$\Phi_{j,r}: \Delta \to \Delta_{j,r}$ be the holomorphic isomorphism with $\Phi_{j,r}(q_j) = 0$ and $\Phi_{j,r}'(0) > 0$. Then there is a continuous function $\theta: (0, d_0/3]\times [0, 2\pi] \to \Bbb R$ such that $\Phi_{j,r}^{-1}(h(q_j+r e^{i\alpha})-q_j) = e^{i \theta(r, \alpha)}$.  Note that $\theta(r, 0) + 2 \pi = \theta(r, 2\pi)$ for  $0< r < d_0/3$.

(1) If $z \in B_j(d_0/3)$ for some $j$, let $z =  q_j + r e^{i \alpha}$. If $z = q_j$, define $H(t, z) = q_j$ for all $0 \le t \le 1$. Otherwise, define
$$
H(t, z) =
\begin{cases}
q_j + \Phi_{j,r}((1 - 2t) e^{i \theta(r, \alpha)}) \cdot  \frac{\Phi_{j,r}'(0) + 2t (r - \Phi_{j,r}'(0))}{\Phi_{j,r}'(0)(1-2t)}, 0 \le t  <  1/2 \\

q_j+ r e^{i(\theta(r, \alpha) + 2(t-1/2)(\alpha - \theta(r, \alpha)))},  1/2 \le t \le 1.
\end{cases}
$$

(2) If $z \in B_j(d_0/2)\setminus B_j(d_0/3)$, let $z =  q_j + r e^{i \alpha}$. Define
$$
H(t, z) = h(z) + (H(t, q_j+d_0e^{i\alpha}/3)-h(q_j+d_0e^{i\alpha}/3))\frac{d_0/2-r}{d_0/6}.
$$

(3) If $z \notin \overline{X} \setminus \cup_{1 \le i \le m} B_j(d_0/2)$, define $H(t, z) = h(z)$ for all $0 \le t \le 1$.

It is clear that $H(\cdot, \cdot): [0, 1] \times \overline{X} \to \overline{X}$ is continuous and $h = H(0, \cdot)$. Let $h_0 = H(1, \cdot)$. From the construction it follows that $h_0|B_j(d_0/3) = {\rm id}$ for all $1 \le j \le m$, and moreover, $h_0$ can be arbitrarily close to $\rm id$ provided that $h$ is close enough to $\rm id$.  For $\eta > 0$ be the number guaranteed in the beginning of the proof,  let $\tau > 0$ be the constant such that ${\rm dist}(h_0, {\rm id}) < \min\{\eta, d_0/12\}$ provided that ${\rm dist}(h, {\rm id}) < \tau$.

Now suppose ${\rm dist}(h, {\rm id}) < \tau$. Then  ${\rm dist}(h_0, {\rm id}) < \min\{\eta, d_0/12\}$. It suffices to construct a homotopy $H_0$ between $h_0$ and $\rm id$. For $z \in \partial X \cup \cup_j B_j(d_0/3)$, define $H_0(t, z) = z$ for all $0 \le t \le 1$. For $z \in X \setminus  \cup_j B_j(d_0/3)$,  both $z$ and $h_0(z)$ belongs to $\Delta \setminus  \overline{\cup_i U_i \cup  \cup_j B_j(d_0/4)}$. By the definition of $\eta$,  there is a unique shortest geodesic segment $\gamma$  in $X$ connecting $z$ and $h_0(z)$.  Now define $H_0(t, z)$ to  be the point in $\gamma$  which divides $\gamma$ into two geodesic segments with the ratio of their length being equal to $t: (1 - t)$. This defines the homotopy $H_0$ between $h_0$ and $\rm id$.  The proof of lemma~\ref{homotopy} is completed.

\end{proof}


\begin{thebibliography}{6}


\bibitem{Ah} L. V. Ahlfors, {\em  Complex Analysis}, McGRAW-HILL BOOK COMPANY, INC. 1953.

\bibitem{ABC}   A. Avila, X. Buff and A. Cheritat,{\em  Siegel disks with smooth boundaries},  Acta Mathematica, 2004, Vol. 193, pp 1-30.



\bibitem{Ba} R. Baer, {\em Isotopien von Kurven auf orientierbaren, geshlossenen F\"{a}chen}, Journal f\"{u}r die Reine und Angewandte Mathematik, 159, 101-116 (1928)

\bibitem{BC}X. Buff and A. Cheritat, {\em How regular can the boundary of a quadratic Siegel disk be? }, Proceedings of the AMS, 2007, Vol 135, pp. 1073-1080.

\bibitem{Ch1} A. Ch\'{e}ritat, {\em  Relatively compact Siegel disks with non-locally connected boundaries}, Math. Ann. , vol. 349, no. 3, pp. 529-542, 2011.

\bibitem{Ch2} A. Cheritat, {\em  Uniformity of the Swiatek's distortion for compact family of Blaschke products}, English version of Herman's manuscript III.



\bibitem{CJ} T. Chen and Y. Jiang, {\em Canonical Thurston obstructions  for
sub-hyperbolic semi-rational branched  coverings}.  Conformal geometry and dynamics, Vol 17 (2013), pp 6-25.


\bibitem{CT} G. Cui and L. Tan, {\em A characterization of hyperbolic rational maps}. Invent. Math.,
Vol. 183 (2011), No. 3, 451-516.



\bibitem{Da} G. David, {\em Solutions de l'equation de Beltrami avec $\|\mu\|
_{\infty} = 1$}, Ann. Acad. Sci. Fenn. Ser. A I Math., 13(1988)
25-70.

\bibitem{dFdM} E. de Faria, W. De Melo, {\em Rigidity of critical circle mappings I }, J. Eur. Math. Soc.(JEMS), 1(1999)339-392.

\bibitem{DS} A. Douady, {\em Systemes dynamiques holomorphes, Seminar Bourbaki}, Asterisque, 105-106 (1983) 39-
63.

\bibitem{Do} A. Douady, {\em Disques de Siegel et anneaux de Herman}, Seminar Bourbaki,
Ast\'{e}erisque, 152-153 (1987) 151-172.

\bibitem{DH} A. Douady and J. Hubbard, {\em A proof of Thurston's topological characterization of rational functions}, Acta Math., 171 (1993), 263-297.


 \bibitem{Ep}   D. B. A. Epstein, {\em Curves on 2-manifolds and isotopies}, Acta Math., 115, 83-107 (1966)

\bibitem{FK} H. M. Farks and I. Kra, {\em Riemann surfaces}, Springer-Verlag New York Inc, 1980.

\bibitem{Ha} P. Haissinsky, {\em  Chirurgie parabolique},  C. R. Acad. Sci. Paris S¨¦r. I Math. 327 (1998), 195-198.

\bibitem{He} M. Herman, {\em Uniformity of the  Swiatek distorsion
of the compact family of  Blaschke products}. Unpublished
manuscript.



\bibitem{IS} H. Inou and M. Shishikura, {\em The renormalization for parabolic fixed
points and their perturbation}, manuscript, 2008.

\bibitem{Kh}  A. Ya. Khinchin, {\em Continued fractions}, Dover, New York, 1992.

\bibitem{Ly} M. Lyubich, {\em The dynamics of raitonal transforms: the topological picture}, Russian Math. Surveys 41:4 (1986) 43-117.

\bibitem{McM1} C. McMullen, {\em Complex dynamics and renormalization},  Princeton University Press, 1994.

\bibitem{McM2} C. McMullen, {\em Self-similarity of Siegel disk and Hausdorff dimension of Julia set},  Acta Mathematica,
180 (1998), 247-292.

\bibitem{P1}C. Petersen, {\em Local connectivity of some Julia sets containing a circle with an irrational rotation}, Acta Math., 177(1996)163-224.
\bibitem{P2} C. Petersen, {\em Herman-Swiatek Theorems with Applications}, London Mathematical Society Lecture Note Series. 274 (2000), 211-225.

\bibitem{P3} C. Petersen, {\em On holomorphic critical quasi-circle
maps}, Ergodic Theory Dynam. Systems 24 (2004), no. 5, 1739-1751.


\bibitem{Per} R. Perez-Marco, {\em Siegel disks with quasi-analytic boundary}, Manuscript, 1997.


\bibitem{PZ} C. Petersen, S. Zakeri,
{\em On the Julia set of a typical quadratic polynomial with a
Siegel disk}, Annals of Mathematics , Vol 159 (2004), NO. 1, 1-52.



\bibitem{Pi} K. M. Pilgrim, {\em Cannonical Thurston obstruction},  Advances in Mathematics,
158 (2001), 154-168.




\bibitem{Ro} J. T. Rogers, {\em Singularities in the boundaries of local Siegel disks}, Ergodic Theory and Dynamical Systems,  Vol 12(1992), no. 4, pp. 803-821.

\bibitem{Sh} M. Shishikura, {\em On bounded type Siegel disks of polynomial
maps}. Unpublished manuscript. 2001.








\bibitem{Tu} P. Tukia, {\em Compactness properties of $\mu$-homeomorphisms}. Ann. Acad.
Sci. Fenn. Ser. A I Math., Vol. 16 (1991) 47-69.





\bibitem{Yo1} J.-C. Yoccoz, {\em $\Pi_{0}$n'y a pas de contre-exemple de Denjoy analytique},
C. R. Acad. Sci. Paris, 298(1984), 141-144.
\bibitem{Yo2} J.-C. Yoccoz, {\em Structure des orbites des hom$\acute{e}$omorphisms analytiques posedant un point critique},
Prepint, 1989.


\bibitem{Za1} S. Zakeri, {\em Dynamics of cubic Siegel polynomials}, Comm. Math. Phys.,
206 (1999) 185-233.




\bibitem{Zh1} G. Zhang, {\em All bounded type Siegel disks of rational maps are
quasi-disks}. Invent. Math.  Vol. 185, NO. 2, 421-466.



\bibitem{Zh2} G. Zhang, {\em On PZ type Siegels of the sine
family}.  To appear in Ergodic Theory \& Dynamical Systems.

\bibitem{Zh3} G. Zhang, {\em On the non-escaping set of $e^{2 \pi i \theta} \sin(z)$}, Israel J. Math., 165 (2008), 233-252.

\bibitem{ZJ} G. Zhang and Y. Jiang, {\em Combinatorial characterization of sub-hyperbolic rational
maps}. Advances in Mathematics, 221 (2009), 1990-2018.

\end{thebibliography}
\end{document}